\documentclass{article}

\usepackage[english]{babel}

\usepackage[letterpaper,top=2cm,bottom=2cm,left=3cm,right=3cm,marginparwidth=1.75cm]{geometry}

\usepackage{amsthm}
\usepackage{amssymb}
\newtheorem{definition}{Definition}
\newtheorem{theorem}{Theorem}

\newtheorem{remark}{Remark}
\usepackage{indentfirst}
\usepackage{amsmath}
\usepackage{graphicx}
\usepackage{amsfonts}
\usepackage{booktabs}
\usepackage{subcaption}
\usepackage{float}
\usepackage[colorlinks=true, allcolors=blue]{hyperref}

\newtheorem{lemmaA}{Lemma}

\usepackage{mathtools}

\DeclarePairedDelimiter{\norm}{\lVert}{\rVert}

\title{Inertial Manifold Neural Operator for Dissipative Time-Dependent Partial Differential Equations}

\author{Xiaoyang Xie\footnote{Program in Applied and Computational Mathematics (PACM), Princeton University, Princeton, 08540, New Jersey, United States. Email: xiaoyangxie@princeton.edu.}~ and
Clarence W. Rowley\footnote{Department of Mechanical and Aerospace Engineering, Princeton
University, Princeton, 08540, New Jersey, United States. Email: cwrowley@princeton.edu.}
}

\begin{document}
\maketitle

\section*{Abstract}
In this paper, we introduce the Inertial Manifold Neural Operator (IMNO) for solving dissipative time-dependent partial differential equations (PDEs). The long-time dynamics of such systems often exhibit an effective low-dimensional structure due to dissipation. Unlike standard neural operator architectures such as the Fourier Neural Operator (FNO), IMNO explicitly leverages the low-dimensional structure to achieve better physical interpretability, accuracy, and stability in long-horizon autoregressive training and prediction for nonlinear dissipative PDEs.
For shift-equivariant PDEs, we further introduce a shift-equivariant variant (IMNO-SE) of the proposed neural operator, ensuring that a spatial shift in the input induces the same spatial shift in the output. This symmetry-preserving inductive bias substantially improves its performance in shift-equivariant PDEs.
Extensive benchmark experiments are presented to evaluate IMNO's performance numerically.

\section{Introduction}
Dissipative partial differential equations model the evolution of systems in which certain functionals of the state—such as energy or enstrophy—are dissipated by mechanisms like viscosity, diffusion, or relaxation. 
This dissipative structure is central to
models in many areas, such as fluid dynamics, chemistry, and biology  \cite{kondo2010reaction,temam2012infinite,turing1990chemical}. Understanding their behavior is therefore important for modeling and predicting complex systems in nature and engineering.

Although these equations are typically infinite-dimensional and may generate intricate nonlinear dynamics, dissipation often suppresses transient degrees of freedom over time, leading to a simpler finite-dimensional description of the long-time dynamics \cite{robinson1995finite}.

Reduced-order modeling (ROM) is a practical way to exploit such low-dimensional structure by approximating the original system with a low-dimensional model. This can greatly reduce simulation cost while also improving understanding of the dominant mechanisms in the systems. A large body of work has been developed along this direction, ranging from classical approaches such as proper orthogonal decomposition (POD) \cite{Berkooz1993POD,rowley2004model,sirovich1987turbulence} to later data-driven methods such as dynamic mode decomposition (DMD) \cite{tu2014dynamic,williams2015data}.

Inertial manifold theory provides a rigorous theoretical framework for understanding why the long-time behavior of many dissipative partial differential equations can be described by finite-dimensional dynamics \cite{constantin2012integral,foias1988inertial,langa1999determining,robinson1996asymptotic,rosa1996inertial}. Its central idea is that, although the dynamics of such PDEs are defined on an infinite-dimensional function space, their eventual behavior are governed by a finite-dimensional invariant manifold that exponentially attracts all trajectories. So, the asymptotic dynamics can be reduced, at least in principle, to a finite-dimensional system evolving on that manifold. Some ROMs have also been developed based on ideas from inertial manifold theory, aiming to approximate the effective long-time dynamics of dissipative PDEs on a low-dimensional manifold \cite{de2023data,linot2020deep}. However, because a ROM is ultimately a low-dimensional model, it can only represent dynamics on a low-dimensional manifold itself. As a result, it inevitably neglects the transient behavior outside that manifold, even though such behavior may still be important for accurately recovering the full dynamics.

In recent years, neural operators have emerged as a powerful class of methods in scientific machine learning \cite{chang2026unsupervised, cheng2026podno,dummer2026ronom,duruisseaux2025fourier,li2020fourier,li2021learning,liu2023tipping,lu2021learning}. Unlike conventional neural networks that operate on finite-dimensional vectors, neural operators are designed to learn mappings between infinite-dimensional function spaces, and can be trained to approximate solution operators of partial differential equations. Once trained on data, a neural operator can efficiently evaluate new input functions and can often be applied across different spatial discretizations. Because of these advantages, they have been successfully applied to many scientific problems, including fluid dynamics, materials science, and climate modeling \cite{jiang2023efficient,wang2024prediction,you2022learning}.

Among neural-operator architectures, the Fourier Neural Operator (FNO) has emerged as one of the most influential and effective frameworks \cite{li2020fourier}. By parameterizing the operator in Fourier space, FNO can efficiently capture nonlocal interactions and exhibit good accuracy, while maintaining a relatively small computational complexity. These advantages make FNO an effective and widely used framework for operator learning across many PDE tasks. However, despite its success, FNO still has several limitations. First, like many purely data-driven neural operators, it lacks physical interpretability: the learned dynamics are represented in a largely black-box manner and do not explicitly reflect the underlying structure of the PDE. Second, although FNO often performs well in one-step prediction, its training process can become unstable in multi-step autoregressive rollout, especially over long time horizons. 

% For dissipative time-dependent PDEs, the long-time behavior is often governed by an effectively low-dimensional structure. 

% The discussion above suggests a combination of the structural interpretability of reduced-order modeling with the expressive power of neural operators. Motivated by this observation, we develop the inertial manifold neural operator (IMNO), a new neural-operator architecture that learns both the dominant long-time dynamics and the function-valued residual correction to preserve model's accuracy in the transient regime. By decomposing the solution into a manifold component and a residual component, IMNO learns a more physically interpretable representation of the system. Moreover, this decomposition makes the learning problem more structured and better conditioned, which can improve training stability in long-horizon autoregressive rollout.

% IMNO can be viewed as a bridge between reduced-order modeling and neural operators. Unlike traditional reduced-order models, which mainly describe the dynamics of the reduced system and often neglect the remaining components outside it, IMNO explicitly incorporates both the manifold component and the residual component. This allows the model to retain the interpretability and compact structure of reduced-order modeling, while using neural operators to preserve the expressive power needed to reconstruct the full PDE dynamics.

The discussion above suggests a combination of the structural interpretability of reduced-order modeling with the expressive power of neural operators. Motivated by this perspective, we develop the Inertial Manifold Neural Operator (IMNO), a new neural-operator architecture that learns both the dominant long-time dynamics on a low-dimensional manifold and a function-valued residual correction for transient off-manifold behavior. Unlike traditional reduced-order models, which mainly describe the dynamics of the reduced system and often neglect the remaining components outside it, IMNO explicitly decomposes the solution into a manifold component and a residual component. This allows the model to retain the interpretability and compact structure of reduced-order modeling, while using neural operators to preserve the expressive power needed to reconstruct the full PDE dynamics. Moreover, this decomposition makes the learning problem more structured and better conditioned, which can improve training stability in long-horizon autoregressive rollout.

Extensive numerical experiments demonstrate the performance of the proposed neural operator. We observe that IMNO remains stable in long-horizon multistep rollout training, achieves good accuracy, and is able to effectively extract the underlying low-dimensional long-time dynamics.

The remainder of this paper is organized as follows. Section 2 reviews the necessary background on neural operators and inertial manifold theory. Section 3 provides a detailed description of the Inertial Manifold Neural Operator (IMNO). Section 4 introduces the shift-equivariant variant of the proposed neural operator. Section 5 presents numerical experiments on several dissipative PDE benchmarks. Section 6 concludes the paper, and the appendix gives a proof of the universal approximation theorem for our formulation.

\section{Background}
\subsection{Problem Setting}

We consider nonlinear dissipative time-dependent partial differential equations posed on a spatial domain $\Omega \subset \mathbb{R}^d$. Let $H$ be a Banach or Hilbert space of functions on $\Omega$, and let $u(t)\in H$ denote the state of the system at time $t$. The evolution is governed by a PDE that generates a solution semigroup
\[
S(t): H \to H, \qquad u_0 \mapsto u(t;u_0),
\]
where $u(t;u_0)$ is the solution at time $t$ with initial condition $u_0$.

In operator-learning tasks, the objective is not merely to approximate a finite-dimensional map between discretized vectors, but rather to learn the infinite-dimensional solution operator associated with the PDE. In particular, given a time step $\Delta t > 0$, we are interested in the one-step evolution operator
% \[
% \mathcal{G}^{\dagger}: H \to H, \qquad \mathcal{G}^{\dagger}(u_t)=u_{t+\Delta t}=S(\Delta t)u_t,
% \]
\[
\mathcal{G}: H \to H, \qquad \mathcal{G}(u)=S(\Delta t)u,
\]
repeated application of this operator then produces an autoregressive rollout of the full trajectory.

We assume access to $N$ trajectories generated by the PDE,
\[
\mathcal{D}=\left\{\left\{u^{(j)}_n\right\}_{n=0}^{T}\right\}_{j=1}^{N},
\
u^{(j)}_n := u(t_n;u^{(j)}_0),\ \ t_n=n\Delta t,
\]
where each initial condition $u^{(j)}_0$ is sampled from an underlying
probability distribution on $H$. 
% This dataset provides supervised
% input--output pairs for one-step evolution,
% \[
% u^{(j)}_n \mapsto u^{(j)}_{n+1}.
% \]
% These pairs provide training samples from the underlying one-step solution operator \(\mathcal{G}^{\dagger}\).

Accordingly, our goal is to construct a parametric neural operator
\[
\mathcal{G}_{\theta}: H \to H, \qquad \mathcal{G}_{\theta}(u)\approx S(\Delta t)u
\]
that approximates $\mathcal{G}$. At inference time, for any given initial condition $u_0$, the learned operator is applied
autoregressively to generate a multi-step forecast:
\[
{u}_{n+1}
=
\mathcal{G}_{\theta}({u}_n),
\qquad n=0,1,\ldots,N_t-1,
\]
producing the predicted trajectory
$\{{u}_n\}_{n=0}^{N_t}$.

\paragraph{Difference between Operator Learning and Vector-to-Vector Function Learning} Unlike standard vector-to-vector function learning, operator learning aims to approximate a function-to-function map; therefore, the learned model should satisfy two key properties.

First, it should be resolution-independent. In
practice, input functions and output functions are only available as samples on spatial
grids, but these sampled vectors are discretizations of underlying
functions in $H$. Hence, if the model is trained on one spatial
resolution, it should remain valid when evaluated on different
resolutions.

Second, the model class should be genuinely operator-valued rather than a
fixed finite-dimensional surrogate. Concretely, for suitable parameters, the image
$\{\mathcal{G}_{\theta}u: u\in H\}$ should not be contained in any finite-dimensional
linear subspace of $H$.

% The central premise of this work is that, for a broad class of dissipative PDEs, the long-time dynamics are effectively governed by a low-dimensional inertial structure together with rapidly decaying transverse components. This motivates the development of neural operators that explicitly exploit such structure, rather than treating the full evolution operator as an unstructured black box.

\subsection{Neural Operators}\label{sec:neural_operators}
One of the most popular neural operator frameworks is the kernel-based operators introduced in \cite{kovachki2023neural}, which 
provides the architectural foundation for the operator developed in this 
work. 

These kernel-based operators can extend to the more general setting where the input and output belong to different function spaces. More specifically, let $\Omega$ denote a bounded domain in $\mathbb{R}^d$, and let $\mathcal{X}(\Omega ; \mathbb{R}^k)$ and $\mathcal{Y}(\Omega ; \mathbb{R}^{k^{\prime}})$ denote Banach spaces of vector-valued functions over $\Omega$. The neural operator aims to learn a mapping between infinite-dimensional 
function spaces,
$$
\mathcal{G} : \mathcal{X}(\Omega; \mathbb{R}^k) \rightarrow 
\mathcal{Y}(\Omega; \mathbb{R}^{k'}),
$$
% in contrast to classical neural networks, which approximate 
% finite-dimensional functions.  
% In the context of PDEs, such an operator typically represents the solution map that sends an initial condition $u_0$ 
% (or a forcing field, coefficient field, etc.) to the corresponding 
% solution $u$ at a later time.  
% Neural operators are therefore designed to approximate operators 
% rather than functions, enabling resolution-invariant learning of PDE 
% dynamics.

The neural operator is defined as a mapping $\mathcal{G}_{\theta}: \mathcal{X}(\Omega ; \mathbb{R}^k) \rightarrow \mathcal{Y}(\Omega ; \mathbb{R}^{k^{\prime}})$ which can be expressed as a composition
$\mathcal{G}_{\theta} = \mathcal{Q} \circ \mathcal{L}_L \circ \cdots \circ \mathcal{L}_1 \circ \mathcal{P}$,
where $\mathcal{P}$ is a lifting layer, 
$\mathcal{L}_\ell$ are hidden layers, and $\mathcal{Q}$ is a projection layer.

The lifting layer $\mathcal{P}$ is given by a mapping

$$
\mathcal{P}: \mathcal{X}\left(\Omega ; \mathbb{R}^k\right) \rightarrow \mathcal{V}\left(\Omega ; \mathbb{R}^{d_c}\right), \quad u(x) \mapsto P(u(x), x),
$$

where  $\mathcal{V}\left(\Omega ; \mathbb{R}^{d_c}\right)$ is an intermediate function space, $P: \mathbb{R}^k \times \Omega \rightarrow \mathbb{R}^{d_c}$ is a learnable neural network acting between finite dimensional Euclidean spaces. 

Each hidden layer $\mathcal{L}_\ell$ takes the form
$$
\mathcal{L}_\ell(v)(x)
    = \sigma \big( W_\ell v(x) + b_\ell + (\mathcal{K}_\ell v)(x) \big),
$$
where $\mathcal{K}_\ell$ is a nonlocal integral operator defined by a specific kernel $K_{\ell}(x, y) \in \mathbb{R}^{d_c \times d_c}$, 
$$
(\mathcal{K}_\ell v)(x) 
    = \int_{\Omega} K_\ell(x,y) v(y)\,dy,
$$

Each hidden layer defines a mapping $\mathcal{L}_{\ell}: \mathcal{V}\left(\Omega ; \mathbb{R}^{d_c}\right) \rightarrow \mathcal{V}\left(\Omega ; \mathbb{R}^{d_c}\right)$. For $\ell= 1, \ldots, L$ the matrices $W_{\ell} \in \mathbb{R}^{d_c \times d_c}$, and bias $b_{\ell} \in \mathbb{R}^{d_c}$ are learnable parameters. The activation function $\sigma: \mathbb{R} \rightarrow \mathbb{R}$ acts component-wise on inputs.

Finally, the projection layer $\mathcal{Q}$ is given by a mapping,

$$
\mathcal{Q}: \mathcal{V}(\Omega ; \mathbb{R}^{d_c}) \rightarrow \mathcal{Y}(\Omega ; \mathbb{R}^{k^{\prime}}), \quad v(x) \mapsto Q(v(x)),
$$

% $$
% \mathcal{Q}: \mathcal{V}\left(\Omega ; \mathbb{R}^{d_c}\right) \rightarrow \mathcal{X}\left(\Omega ; \mathbb{R}^{k}\right), \quad v(x) \mapsto Q(v(x), x),
% $$

where $Q: \mathbb{R}^{d_c} \rightarrow \mathbb{R}^{k^{\prime}}$ is also a learnable neural network acting between finite dimensional Euclidean spaces. 

\paragraph{Fourier Neural Operators}
Among the kernel-based neural operator architectures, the Fourier Neural Operator (FNO)
has emerged as one of the most influential and effective frameworks\cite{li2020fourier}.
FNO can be viewed as a special case of kernel-based neural
operators in which the kernel is restricted to be convolutional and
Fourier–parametrized. Specifically, FNO assumes a translation-invariant kernel
$K_\ell(x,y)=K_\ell(x-y)$,
which is represented by a truncated Fourier expansion
\[
K_\ell(x-y)
  = \sum_{|k|\le k_{\max}} R_{\ell}(k)\, e^{ik(x-y)} ,
\]
where each coefficient $R_{\ell}(k)\in\mathbb{C}^{d_c\times d_c}$
is a trainable matrix.
This corresponds to learning the kernel directly in frequency space.

Recall that a general kernel operator takes the form
\[
({\mathcal K}_\ell v)(x)
  = \int_\Omega K_\ell(x,y)\, v(y)\, dy ,
\]
Substituting the Fourier representation of $K_\ell(x,y)$ into this expression
shows that ${\mathcal K}_\ell$ acts diagonally in Fourier space: each Fourier
mode of $v$ is multiplied by the learned response $R_{\ell}(k)$:
\[
\widehat{({\mathcal K}_\ell v)}(k)
  = R_{\ell}(k)\, \widehat v(k),
  \qquad |k|\le k_{\max}.
\]

Transforming back to physical space yields the spectral convolution:
\[
({\mathcal K}_\ell v)(x)
  =
  {\mathcal F}^{-1}\!\left(
       R_{\ell}(k)\,\widehat v(k)
  \right)(x).
\]

Thus each FNO hidden layer (Fourier block) takes the form

$$
{\mathcal L}_\ell(v)(x)
  = \sigma\!\Big(
        W_\ell v(x)
      + b_\ell
      + {\mathcal F}^{-1}( R_\ell \cdot {\mathcal F}(v))(x)
    \Big).
$$

The Fourier parameterization endowed by this construction gives FNO
several practical advantages. First, the use of Fourier modes
provides each layer with a global receptive field, allowing
the architecture to capture nonlocal
interactions that commonly arise in PDEs from fluid dynamics and other physical systems. Second, the spectral convolution
${\mathcal F}^{-1}(R_\ell(k)\,{\mathcal F}(v)(k))$ can be implemented
using FFTs, yielding an ${\mathcal O}(N\log N)$ computational cost per
layer and enabling FNO to scale efficiently to high-dimensional and
high-resolution grids. Third, because only a fixed number of Fourier
coefficients are learned, the representation of the operator is largely
independent of the discretization of $\Omega$; a model trained on one
resolution can often be evaluated on finer or coarser meshes without
retraining. This resolution-invariance property distinguishes FNO from
traditional convolutional or multilayer perceptron (MLP)-based architectures whose parameters
are tied directly to the grid. Finally, FNO has demonstrated strong
empirical performance across a broad range of operator-learning
benchmarks—including Darcy flow, Burgers equation, and moderate-Reynolds
Navier–Stokes—where it often achieves state-of-the-art accuracy with fewer parameters than competing architectures. These
properties together make FNO an efficient and widely adopted neural
operator model for learning mappings between function spaces.

\subsection{Autoregressive Training Stability and Dissipative Priors}

Many operator-learning tasks of practical interest, such as forecasting
solutions of time-dependent PDEs, require models that remain stable 
under long-horizon autoregressive rollout, so that they can faithfully 
capture the long-term dynamics of dissipative or even chaotic systems
\cite{li2021learning}. Although predictions are ultimately made autoregressively—
each predicted state is fed back into the model to generate the next
one—the training procedure itself may follow two distinct paradigms.

In the classical teacher-forcing approach, the model is trained to predict
the next state using the ground-truth input at every step. In contrast,
the autoregressive training paradigm recursively feeds the model's own 
predictions back into the input window during training. As advocated in
\cite{ye2025recurrent}, the latter approach encourages the model to learn the intrinsic multi-step 
evolution law of the underlying PDE rather than merely fitting short 
transitions, and has been shown to improve long-term predictive accuracy in many PDE systems.

However, despite these conceptual advantages, training models such as FNO over long autoregressive horizons can become markedly unstable. Models that train reliably under short horizons become
difficult to optimize when the autoregressive time window is extended, suggesting that the current structure of FNO may be insufficient to represent the intrinsic structure required for
stable multi-step rollouts and may also hinder effective backpropagation during autoregressive training.

Long-horizon autoregressive prediction is particularly natural for PDE
systems whose dynamics are dissipative, or at least non-expansive in a
suitable sense. Without such stability, one-step prediction errors may grow
substantially under repeated autoregressive iteration, which can severely
limit the long-horizon accuracy of learning-based solvers.
Thus, when discussing autoregressive training for long-time PDE prediction,
dissipative systems form a central and practically important setting. This suggests that incorporating dissipative
prior structure directly into neural-operator design is not only natural but
potentially essential. For instance, \cite{li2021learning} introduces additional
dissipative regularization and dissipative constraints to encode such priors
during training.

Our work adopts a stronger structural hypothesis: the underlying PDE
admits an inertial manifold, namely, an invariant finite-dimensional manifold that exponentially attracts all trajectories. Many canonical
dissipative PDEs are known to have this structure in
appropriate settings. By explicitly incorporating inertial-manifold structure into the neural-operator architecture, we obtain a more stable, accurate, and physically interpretable model for long-horizon autoregressive learning.
% By explicitly embedding inertial-manifold
% geometry into the neural-operator architecture, we obtain several
% desirable properties, including substantially improved training
% stability for long-horizon autoregressive learning.

\subsection{Inertial Manifolds and Asymptotic Completeness}
\subsubsection{The Evolution Equation}

We consider a general class of dissipative evolution equations posed on a
separable Hilbert space $H_0$,
\begin{equation}
\frac{du}{dt} + Au = f(u), \qquad u(0)=u_0\in H_0,
\label{eq:abstract_evolution}
\end{equation}
where $A:D(A)\subset H_0\to H_0$ is a positive, self-adjoint, unbounded linear
operator with compact inverse, and $f$ is a nonlinear term with some Lipschitz condition.

The operator $A$ models the dominant dissipative mechanism of the system.
Its positivity and compact inverse imply a discrete spectrum
$0<\lambda_1\le \lambda_2\le\cdots$ with $\lambda_k\to\infty$.
so, under the linear flow generated by $\frac{du}{dt} = -Au$, the high-frequency modes decay rapidly.
The nonlinearity $f$ couples these modes and, under appropriate assumptions,
preserves the dissipative structure imposed by $A$.

To describe regularity in a precise way, we introduce the fractional domains
of $A$.
For $\gamma\ge 0$, the fractional domain of $A$ is defined by
\[
D(A^\gamma)={u\in H_0: A^\gamma u\in H_0}.
\]
Equipped with the inner product
\[
\langle u,v\rangle_{D(A^\gamma)}=
\langle A^\gamma u,A^\gamma v\rangle_{H_0},
\]
the space $D(A^\gamma)$ is itself a Hilbert space. 
in concrete PDE settings, they can often be identified, up to equivalence of
norms, with fractional Sobolev spaces.

We assume that the nonlinearity $f$ is locally Lipschitz from $D(A^\alpha)$
into $D(A^\beta)$ for some $0<\alpha-\beta<1$, reflecting the fact that
nonlinear effects typically lower regularity, while the linear semigroup
generated by $-A$ has a smoothing effect.

Under these assumptions, \eqref{eq:abstract_evolution} generates a continuous
semigroup $\{S(t)\}_{t\ge 0}$ on $D(A^\alpha)$.
Solutions depend continuously on initial data in appropriate
fractional norms: on bounded sets, small differences in initial conditions
remain controlled over finite time intervals \cite{robinson1996asymptotic}.

Here and below, we denote
\[
H:=D(A^\alpha)
\]
as the Hilbert phase space.
With this convention, the equation is understood as an evolution equation on $H$, and the semigroup is written as
\[
S(t):H\to H,\qquad t\ge 0.
\]

% This convention is useful for two reasons. First, in many dissipative PDEs the
% nonlinearity is not naturally Lipschitz on the base space $H_0$ itself, but is
% well controlled on a fractional-domain phase space. Thus the space in which
% the semigroup is well posed and the space in which attraction and tracking are
% measured may be $D(A^\alpha)$ rather than $H_0$. Second, after this choice of
% phase space has been made, the inertial-manifold theory can be stated in a
% uniform way: all distances, attraction estimates, tracking estimates, and
% operator-learning maps below are understood with respect to the norm of this
% chosen Hilbert space $H$. In the special case $\alpha=0$, this convention simply
% reduces to $H=H_0$.

In this work, we are interested in dissipative systems: there exists a bounded
absorbing set $B\subset H$ such that every bounded subset of $H$ is eventually
mapped into $B$ by the flow.
If $B$ is compact and positively invariant, standard results imply the
existence of a global attractor $\mathcal{A}\subset H$.
The global attractor is compact, positively invariant, and attracts all
bounded sets as $t\to\infty$.
Although the phase space $H$ is infinite-dimensional, under some additional assumptions, the attractor has
finite dimension and captures all possible long-time behaviors of
the system \cite{robinson1995finite}.

This abstract framework provides a natural setting for studying finite-dimensional
descriptions of dissipative dynamics.

\subsubsection{Inertial Manifold}
Although the existence of a global attractor ensures that the long-time
dynamics are confined to a compact set, this alone
does not necessarily provide a convenient low-dimensional description
of the dynamics.

A genuine dimensional reduction becomes possible when the attractor is embedded
in a smooth finite-dimensional invariant manifold. On such a manifold, the
long-time dynamics admit an intrinsic finite-dimensional representation and can
be described by a closed system of ordinary differential equations. In
dissipative systems, inertial manifold theory provides a rigorous framework for
this idea: an inertial manifold is a finite-dimensional, positively invariant,
Lipschitz or smoother manifold that contains the global attractor and attracts
all trajectories at an exponential rate. Thus, it provides a
finite-dimensional structure of the long-time dynamics, with exponential
attraction replacing the generally slower convergence toward the global
attractor.

% Following the classical formulation of inertial manifolds
% \cite{foias1988inertial}, consider the dissipative evolution equation
% \eqref{eq:abstract_evolution}, and assume that it generates a semigroup
% \(\{S(t)\}_{t\ge 0}\) on \(H\).

\begin{definition}[Inertial manifold]
A set $\mathcal{M}\subset H$ is called an \emph{inertial manifold} for
\eqref{eq:abstract_evolution} if it satisfies the following:
\begin{enumerate}
\item $\mathcal{M}$ is a finite-dimensional Lipschitz
manifold;
\item $\mathcal{M}$ is positively invariant under the flow, i.e.,
\[
S(t)\mathcal{M} \subset \mathcal{M}, \qquad \forall\, t\ge 0;
\]
\item $\mathcal{M}$ attracts all trajectories at an exponential rate:
there exists a $\mu>0$ such that for every $u_0 \in H$ there is a $C=C(u_0)$ satisfying
\[
\mathrm{dist}\big(S(t)u_0,\mathcal{M}\big)
\le C e^{-\mu t},
\qquad \forall\, t\ge 0.
\]
\end{enumerate}
\end{definition}

The exponential attraction property is stronger than the attraction property
of a general global attractor. It implies that every bounded set of initial
conditions approaches $\mathcal M$ at an exponential rate, so that the
off-manifold part of the dynamics becomes negligible after a transient period.
As a result, the long-time behavior of the system can be faithfully
described by the finite-dimensional dynamics induced on $\mathcal{M}$.

\subsubsection{Spectral gap condition and existence of inertial manifolds}

Almost all known existence results for inertial manifolds rely on the presence
of a \emph{spectral gap} in the linear operator $A$.
This condition reflects a strong separation between low and high modes and
provides a mechanism by which the fast-decaying components can be slaved to
the slow dynamics.

Recall that $A$ is a positive self-adjoint operator with compact inverse, with eigenvalues
$0<\lambda_1\le \lambda_2\le\cdots$ and corresponding spectral projections
$P_n$ and $Q_n = I-P_n$. The nonlinearity $f$ is globally Lipschitz from $D(A^\alpha)$ into
$D(A^\beta)$ for some $0\le \alpha-\beta \le \tfrac12$, with Lipschitz constant $C_1$.
A standard existence theorem then states that if, for some $n$,
\begin{equation}
\lambda_{n+1}-\lambda_n
>
2C_1\bigl(\lambda_n^{\alpha-\beta}+\lambda_{n+1}^{\alpha-\beta}\bigr),
\label{eq:spectral_gap}
\end{equation}
there exists an inertial manifold $\mathcal M$ of dimension $n$ \cite{robinson1996asymptotic}.
Moreover, $\mathcal M$ can be realized as the graph of a Lipschitz function
$\phi:P_nH_0\to Q_nH_0\cap D(A^\alpha)$ and attracts all trajectories exponentially.

Under this condition, inertial manifolds can be rigorously constructed
for several classes of dissipative PDEs, including some reaction--diffusion systems, viscous Burgers equation, and
Kuramoto--Siva\-shinsky equation.

\subsubsection{Reduced dynamics on the manifold}

Under the spectral-gap condition above, the inertial manifold has the graph form
\[
\mathcal{M}=\{p+\phi(p):\, p\in P_nH_0\},
\]
where $\phi:P_nH\to Q_nH_0\cap D(A^\alpha)$ is Lipschitz. Therefore, for each
$u\in\mathcal{M}$, the decomposition
\[
u=p+\phi(p),\qquad p=P_nu,
\]
is uniquely determined by the low-mode coordinate $p$. Applying $P_n$ to the full equation
$\partial_t u + Au = f(u)$ and substituting $u=p+\phi(p)$ gives the closed reduced system
on $P_nH_0$:
\begin{equation}\label{eq:reduced}
\frac{dp}{dt} + A p = P_n f\big(p+\phi(p)\big).
\end{equation}
Solutions of \eqref{eq:reduced} parameterize trajectories on $\mathcal{M}$, thus, once a trajectory has been restricted to the inertial manifold, its evolution is governed by a closed finite-dimensional ODE.

\subsubsection{Asymptotic completeness}
The defining properties of an inertial manifold give a distance-to-set form
of attraction: invariance and exponential attraction imply that every
trajectory approaches the manifold at a uniform exponential rate.
However, these properties alone do not specify how individual trajectories
approach the manifold, nor whether their long-time behavior can be related
to a single trajectory evolving on it.

Asymptotic completeness provides a refinement of exponential attraction by
upgrading it from a setwise statement to an orbitwise one.
Rather than asserting only that trajectories converge toward the manifold,
it guarantees that each trajectory in the full phase space is exponentially
tracked by a trajectory lying on the manifold \cite{langa1999determining,robinson1996asymptotic}.

\begin{definition}[Asymptotic completeness]
Let $\mathcal{M}\subset H$ be a positively invariant attracting set for the
semigroup $\{S(t)\}_{t\ge 0}$.
The set $\mathcal{M}$ is said to be \emph{asymptotically complete} if for every
initial condition $u\in H$ there exists a point $u^M\in\mathcal{M}$ such that
\begin{equation}
\|S(t)u - S(t)u^M\| \le C(u)\, e^{-\mu t},
\quad \forall t\ge 0,
\end{equation}
for some constant $\mu>0$.
\end{definition}

For such a $u^M$, the trajectory $S(t)u^M$ is referred to as the
\emph{tracking trajectory} associated with $u$. We will also refer to
$u^M$ as the \emph{asymptotic phase} of $u$.
Asymptotic completeness implies that every trajectory is
exponentially tracked by a corresponding tracking trajectory on $\mathcal{M}$ and hence
shares the same asymptotic behavior. Figure~\ref{fig:im_vis} schematically illustrates this idea.

\begin{figure}[H]
\centering
\includegraphics[width=0.72\textwidth]{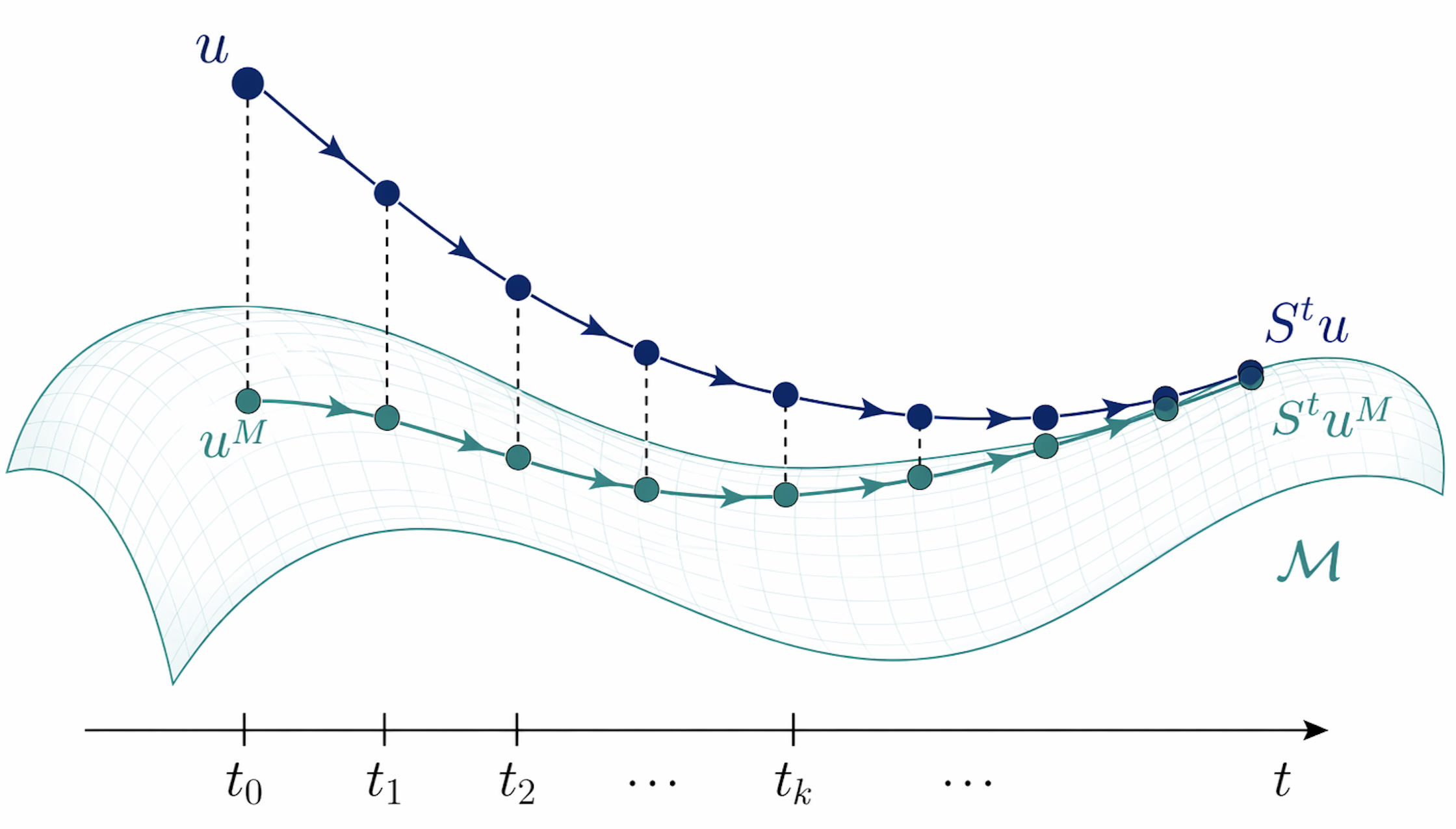}
\caption{Schematic of a trajectory in the full space and its tracking trajectory on the inertial manifold.}
\label{fig:im_vis}
\end{figure}

A sufficient condition for asymptotic completeness is
\emph{flow-normal hyperbolicity} \cite{robinson1996asymptotic,rosa1996inertial}.

\begin{definition}[Flow-normal hyperbolicity]
Let $\mathcal{M}\subset H$ be a positively invariant attracting manifold for
the semigroup $\{S(t)\}_{t\ge 0}$. Suppose that $\mathcal{M}$ attracts
trajectories at an exponential rate $\mu>0$, namely
\begin{equation}
\operatorname{dist}(S(t)u,\mathcal{M})
\le
Ce^{-\mu t},
\quad \forall t\ge 0.
\end{equation}
The manifold $\mathcal{M}$ is said to be \emph{flow-normally hyperbolic} if there
exist $\gamma<\mu$ such that
\begin{equation}
\|u_1(t)-u_2(t)\|
\le
C e^{-\gamma t}\|u_1(0)-u_2(0)\|, \quad \forall t\le 0.
\end{equation}
for two trajectories $u_1(t),u_2(t)$ lying on $\mathcal{M}$, 
\end{definition}

In other words, flow-normal hyperbolicity requires that the backward-in-time
separation of trajectories on $\mathcal{M}$ is dominated by the exponential
attraction toward $\mathcal{M}$. Such a domination property is not automatic from the existence of an inertial manifold and must be verified under additional assumptions in specific settings.

Flow-normal hyperbolicity implies asymptotic completeness \cite{robinson1996asymptotic}. Moreover, it also implies the uniqueness of the asymptotic phase
$u^M$. Indeed, if $u^M_1,u^M_2\in\mathcal{M}$ both track the same trajectory
$S(t)u$ at rate $\mu$, then
\[
\|S(t)u^M_1-S(t)u^M_2\|
\lesssim e^{-\mu t}.
\]
Applying the backward-time estimate on $\mathcal{M}$ gives
\[
\|u^M_1-u^M_2\|
\lesssim
e^{\gamma t}\|S(t)u^M_1-S(t)u^M_2\|
\lesssim
e^{-(\mu-\gamma)t},
\]
and letting $t\to\infty$ yields $u^M_1=u^M_2$. Thus the tracking trajectory
associated with each initial condition is uniquely determined.

The uniqueness of the asymptotic phase naturally induces an asymptotic phase map
\[
\Pi:H\to\mathcal{M},
\]
which assigns to each initial condition $(u\in H)$ its unique asymptotic phase
$(\Pi(u)=u^M\in\mathcal{M})$. Under suitable additional spectral-gap and
regularity assumptions, one can further obtain the continuity of this map \cite{chow1992smoothness,rosa1996inertial}.
When this holds, $\Pi$ may be viewed as a continuous asymptotic phase
retraction from the full phase space onto the inertial manifold.

% Under the spectral gap condition proposed in \cite{rosa1996inertial}, one further obtains an
% stable-foliation structure. More precisely, the inertial manifold
% $\mathcal{M}$ admits a family of stable leaves
% $\{W^s(\xi)\}_{\xi\in\mathcal{M}}$ such that
% \begin{enumerate}
% \item $H = \bigsqcup_{\xi\in\mathcal{M}} W^s(\xi)$;
% \item $W^s(\xi)\cap\mathcal{M}=\{\xi\}$;
% \item if $u\in W^s(\xi)$, then
% \begin{equation}
% \|S(t)u - S(t)\xi\| \le C e^{-\beta t}, \qquad t\ge 0,
% \end{equation}
% for suitable $\beta>0$.
% \end{enumerate}

% Thus, each stable leaf $W^s(\xi)$ consists of all initial conditions whose
% trajectories are exponentially tracked by the same trajectory $S(t)\xi$ on
% $\mathcal{M}$.

% Accordingly, one may introduce an continuous asymptotic phase retraction
% \[
% \Pi:H\to\mathcal{M},
% \]
% which assigns to each $u\in H$ the unique point
% $\Pi(u)\in\mathcal{M}$ such that
% \[
% u \in W^s(\Pi(u)).
% \]
% So $\Pi(u)$ is a point on the manifold whose trajectory
% $S(t)\Pi(u)$ exponentially tracks the trajectory $S(t)u$. 

This naturally induces a decomposition of $u \in H$ into a
manifold component and an off-manifold residual:
\[
u^M := \Pi(u), \qquad
u^R := u - \Pi(u).
\]
Thus,
\[
u = u^M + u^R,
\]
where $u^M$ represents the asymptotic-phase representative on
$\mathcal M$, while $u^R$ represents the deviation of $u$ from manifold.

% In this formulation, the dynamics along $\mathcal{M}$ govern the slow,
% long-time evolution of the system, while deviations in the transverse
% directions decay exponentially.
% This provides a precise mathematical expression of the slaving principle,
% in which all remaining degrees of freedom are exponentially slaved to the
% finite-dimensional dynamics on the inertial manifold.

\section{Inertial Manifold Neural Operators (IMNO)}

The above inertial-manifold picture suggests an orbitwise reduction:
each full trajectory admits an asymptotic phase on a finite-dimensional,
positively invariant set, and the remaining degrees of freedom decay
exponentially. Motivated by this, we introduce a neural operator architecture that explicitly decomposes the state into
\begin{enumerate}
    \item[(i)] a low-dimensional \emph{manifold coordinate} $h(t)$,
    \item[(ii)] a function-valued \emph{residual} $u^R(t)$,
\end{enumerate}
with the goal of approximating both the reduced dynamics on the manifold and the off-manifold correction needed to recover the full state solution.

The overall structure of the Inertial Manifold Neural Operator (IMNO) is summarized in the flowchart below. In the following subsections, we describe each part in detail.

\begin{figure}[htbp]
    \centering
    \includegraphics[width=0.95\textwidth]{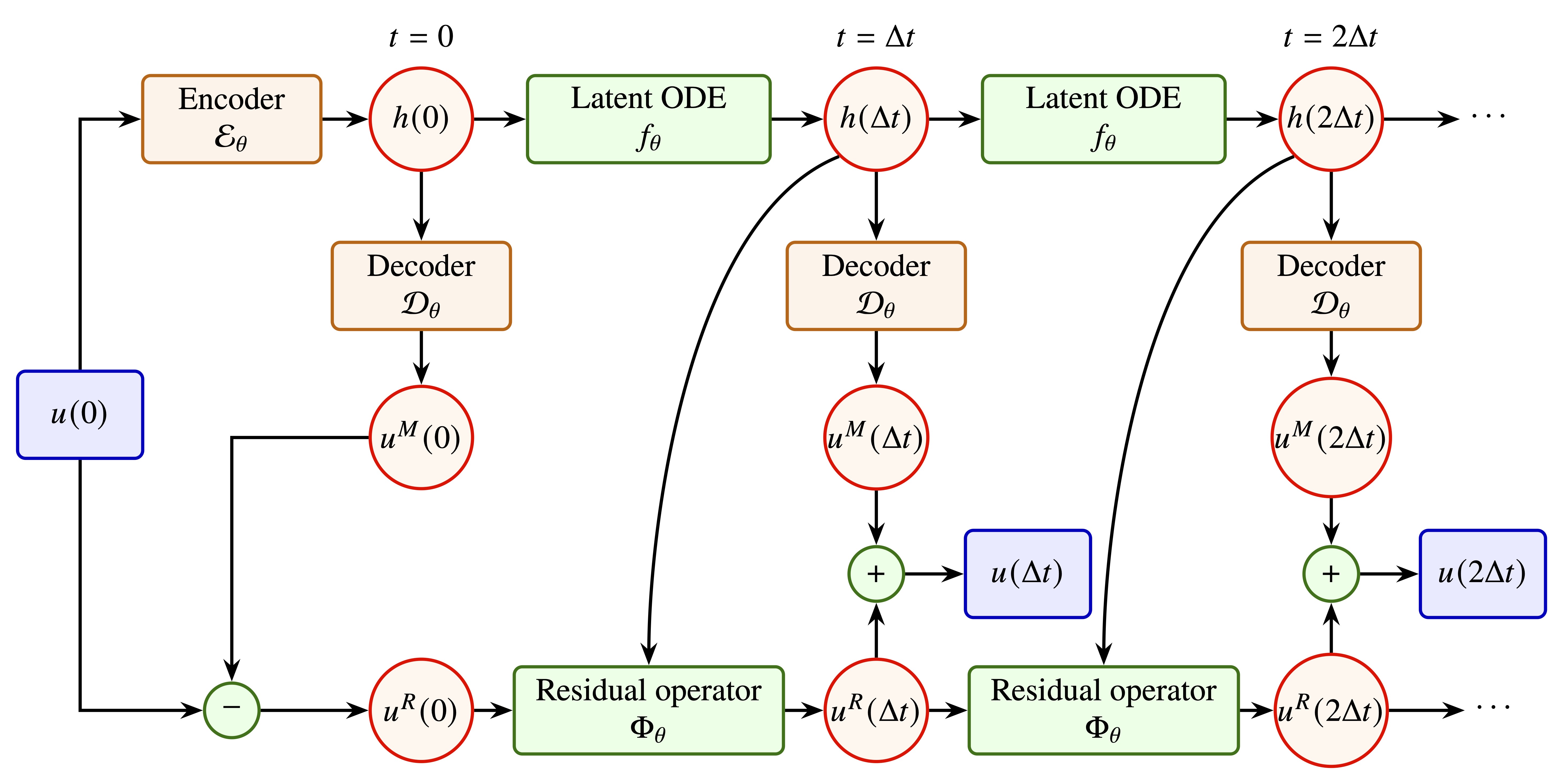}
    \caption{Overall architecture of IMNO.}
    \label{fig:imno-flowchart}
\end{figure}

\subsection{Manifold-Based Decomposition}

Let $u(t)\in H$ denote the solution of the PDE at time $t$.
IMNO is built on the assumption that the full state solution $u$ admits a decomposition
into a manifold component and a residual component,
\begin{equation}\label{eq:decomp}
u(0) \;=\; u^M(0) + u^R(0),
\end{equation}
where $u^M(t)$ lies on a learned manifold $\mathcal{M}_\theta\subset H$,
and $u^R(t)\in H$ represents the remaining degrees of freedom.

% The manifold component $u^M(t)$ is parameterized by a finite-dimensional
% intrinsic coordinate $h(t)\in\mathbb{R}^d$ through an embedding of the learned manifold decoder,
% \[
% u^M(t)=\mathcal{D}_\theta(h(t)).
% \]
Conceptually, $u^M(t)$ serves as a representative of the asymptotic phase,
while the full state $u(t)$ is viewed as lying on the stable leaf associated
with $u^M(t)$.
This decomposition is consistent with the inertial-manifold and
asymptotic-completeness perspective introduced in the previous section: the
long-time behavior is governed by evolution along the manifold, while
residual components decay rapidly.

To realize the decomposition \eqref{eq:decomp} in practice, we parameterize
the learned manifold $\mathcal{M}_\theta$ through an encoder--decoder pair
\[
\mathcal{E}_\theta:H\to \mathbb{R}^{d_h},
\qquad
\mathcal{D}_\theta:\mathbb{R}^{d_h}\to H,
\]
and define
\[
\mathcal{M}_\theta \;=\;\{\mathcal{D}_\theta(h): h\in\mathbb{R}^{d_h}\}.
\]
Both the encoder and decoder are designed to be resolution-independent,
they act on functions rather than fixed-dimensional vectors and remain
compatible with varying spatial resolutions.

% \paragraph{Manifold projection and state representation.}
Given an initial condition $u(0)\in H$, IMNO begins by projecting
$u(0)$ onto the learned manifold.
The encoder produces a latent coordinate
\[
h(0) = \mathcal{E}_\theta(u(0)),
\]
which serves as the coordinate induced by an embedding of the learned manifold into $\mathbb{R}^d$.
The corresponding manifold component is obtained by decoding,
\[
u^M(0) = \mathcal{D}_\theta(h(0)) \in \mathcal{M}_\theta,
\]
and the residual is defined as
\[
u^R(0) = u(0) - u^M(0).
\]
The state of the system $u(0)$ is thus represented by the pair $(h(0), u^R(0))$, where
$h(0)$ corresponding to the latent coordinate of $u$ on the learned manifold and $u^R(0)$ capturing its deviation from the manifold.

Starting from this representation, the latent variable $h(t)$ and the residual
$u^R(t)$ are evolved forward in time according to separate learned dynamics,
which are introduced in the following subsections.
At each step, the full state is reconstructed as
\[
u(t) = \mathcal{D}_\theta(h(t)) + u^R(t).
\]
% This separation between manifold evolution and residual correction reflects
% the inertial-manifold perspective we introduced above.

The encoder--decoder architecture is specified as follows.

\subsubsection*{Encoder Architecture}
The encoder $\mathcal{E}_\theta$ maps an input field $u\in H$ to a latent
coordinate $h\in\mathbb{R}^{d_h}$ by extracting global, low-frequency features.
The input field is first lifted to a higher-dimensional feature representation
through a pointwise linear layer that uses both the field value and the spatial
grid coordinate. The lifted feature field is then transformed to Fourier space,
and only a fixed number of low modes are retained. The real and imaginary parts
of these modes across all lifted channels are concatenated and projected into
$\mathbb{R}^{d_h}$ by a MLP, followed by a normalization layer.

\subsubsection*{Decoder Architecture}
The decoder $\mathcal{D}_\theta$ maps a latent coordinate $h\in\mathbb{R}^{d_h}$
back to a field in $H$ by generating low-frequency Fourier coefficients.
A MLP first maps $h$ to complex-valued coefficients
corresponding to a fixed set of Fourier modes.
These coefficients are then embedded into a Fourier spectrum and transformed back to physical space at a desired
spatial resolution via an inverse Fourier transform.
\subsection{Dynamics of the Manifold and Residual Components}
\subsubsection{Latent dynamics on the learned manifold}

The inertial-manifold theory indicates that the latent dynamics on the manifold can be written in the closed form \eqref{eq:reduced}
, the latent coordinate evolves autonomously and independently of the residual $u^R(t)$. Therefore, we model this closed dynamics using a Neural ODE \cite{chen2018neural}, where $f_\theta:\mathbb{R}^{d_h}\to\mathbb{R}^{d_h}$ is a learned vector field.

\begin{equation}\label{eq:latent_closed}
\dot h(t) = f_\theta(h(t)),
\end{equation}

In practice, we adopt a forward Euler discretization,
leading to the one-step update
\begin{equation}
h_{t+1} = h_t + \Delta t\, f_\theta(h_t),
\end{equation}
we seek $h(t)$ to capture the long-time dynamics of the system.

\subsubsection{Off-manifold residual dynamics}

The residual component $u^R(t)$ represents off-manifold corrections and is not expected to possess an autonomous, closed dynamics.
Instead, following the stable-foliation viewpoint from the inertial manifold theory,
the evolution of $u^R(t)$ also depends on the latent coordinate $h(t)$.

Specifically, we parameterize the dynamics of the residual component $u^R(t)$ by a neural operator of the form
% \begin{equation}
% u^{R}_{t+1} = \Phi_\theta\!\big(u^{R}_{t},h_{t+1}\big).
% \end{equation}
\begin{equation}
% u^{R}(t+\Delta t) = \Phi_\theta\!\big(u^{R}(t),h(t+\Delta t)\big).
u^R_{t+1} = \mathcal{R}_\theta(u^R_t, h_{t+1}).
\end{equation}
The operator $\mathcal{R}_\theta$ is a neural-operator,
mapping function-valued inputs to function-valued outputs, while being conditioned on the finite-dimensional latent coordinate $h$.

Building on the FNO-based neural-operator framework introduced in Section~\ref{sec:neural_operators}, we define $\mathcal{R}_\theta$ as a composition of lifting, hidden, and projection layers:
\[
\mathcal{R}_\theta(\cdot, h)
=
\mathcal{Q} \circ \mathcal{L}^{(h)}_L \circ \cdots \circ \mathcal{L}^{(h)}_1 \circ \mathcal{P}.
\]
Here $\mathcal{P}$ is a lifting operator from the residual channel space to an intermediate latent feature space, and $\mathcal{Q}$ is a projection operator that maps latent features back to the residual channel space. Following the notation in Section~\ref{sec:neural_operators}, these maps are written as
\[
\mathcal{P}:\mathcal{X}(\Omega;\mathbb{R}^{k})\to\mathcal{V}(\Omega;\mathbb{R}^{d_c}),
\qquad
u(x)\mapsto P\bigl(u(x),x\bigr),
\]
and
\[
\mathcal{Q}:\mathcal{V}(\Omega;\mathbb{R}^{d_c})\to\mathcal{X}(\Omega;\mathbb{R}^{k}),
\qquad
v(x)\mapsto Q(v(x)),
\]
where $P$ is a pointwise MLP from $\mathbb{R}^{k+d}$ to $\mathbb{R}^{d_c}$, $Q$ is a pointwise MLP from $\mathbb{R}^{d_c}$ to $\mathbb{R}^{k}$.

For a standard FNO, the hidden-layer update takes the form
\[
\mathcal{L}_\ell(v)(x)=\sigma\!\bigl(W_\ell v(x)+b_\ell+(\mathcal{K}_\ell v)(x)\bigr),
\]
where $W_\ell$ is a pointwise linear transformation and $\mathcal{K}_\ell$ is a spectral convolution operator. In our model, we take the kernel in $\mathcal{K}_\ell$ to be the same as in FNO, namely
\[
({\mathcal K}_\ell v)(x)
  =
  {\mathcal F}^{-1}\!\left(
       R_{\ell}(k)\,\widehat v(k)
  \right)(x).
\]

Furthermore, we modify this basic layer by replacing the fixed bias term with a latent coordinate-dependent. More precisely, at layer $\ell$, the update in the layer is 
\[
\mathcal{L}^{(h)}_\ell(v)(x)=\sigma\!\bigl(W_\ell v(x)+\beta_\ell(h)+(\mathcal{K}_\ell v)(x)\bigr),
\]

Here $\beta_\ell$ is a learnable affine map from $\mathbb{R}^{d_h}$ to $\mathbb{R}^{d_c}$, so $\beta_\ell(h_{t+1})$ is a learned vector in the feature-channel dimension obtained from the latent state $h_{t+1}$ and broadcast uniformly over the spatial domain. In this way, the residual dynamics are not modeled by a fixed operator acting on $u^R_t$ alone; instead, the operator is conditioned on the current location on the learned manifold through an additive correction.

Therefore, $\mathcal{R}_\theta$ may be interpreted as a conditioned neural operator for the residual dynamics: it updates the residual field through an operator
whose action depends on the current latent coordinate. This is consistent with
the inertial-manifold viewpoint, where the residual degrees of freedom are
not evolved independently but are coupled to the low-dimensional manifold
dynamics.

\begin{remark}
One may wonder whether the latent-dependent term \(\beta_\ell(h_{t+1})\), being
spatially constant, is too limited to inject sufficiently rich information from
the latent coordinate into the residual dynamics. From the viewpoint of
expressive power, this limitation is not fundamental. Universal approximation
results for the averaging neural operator (ANO) show that even remarkably simple nonlocal mechanisms can suffice for universal operator approximation
\cite{lanthaler2025nonlocality}.

Although $\beta_\ell(h_{t+1})$ does not depend on the spatial coordinate, it enters every hidden layer and interacts with the pointwise nonlinearities and the nonlocal Fourier convolution operators. In this sense, it provides a global mechanism through which the latent coordinate can influence the residual evolution operator. At the level of approximation theory, this mechanism is sufficient to establish a corresponding universal approximation theorem for the model considered here.

\end{remark}

% \begin{theorem}[Universal approximation of IMNO on systems with inertial manifold structure]
% \label{thm:ua}
% (Todo) Let $H$ be a Banach space and let $\mathcal{G}:H\to H$ be the one-step solution operator
% of a dissipative PDE. Assume that:

% \begin{enumerate}
%     \item the PDE admits an asymptotically complete inertial manifold $\mathcal M\subset H$
%     of dimension $m$;
%     \item the associated asymptotic phase map $\Pi:H\to\mathcal M$ exists and is continuous on a compact set
%     $K\subset H$, and is compatible with the one-step flow:
%     \[
%     \Pi\circ \mathcal{G} = \mathcal{G}|_{\mathcal M}\circ \Pi
%     \quad \text{on } K.
%     \]
%     \item The residual one-step map
%     \[
%     \mathcal R(u^R,h)
%     =
%     \mathcal{G}(D(h)+u^R)
%     -
%     \Pi(\mathcal{G}(D(h)+u^R))
%     \]
%     is continuous on the compact set induced by \(K\).
% \end{enumerate}

% Then for every $\varepsilon>0$, there exists an IMNO model such that for every $u\in K$,
% if
% \[
% u=\underbrace{u^M}_{\Pi(u)}+\underbrace{u^R}_{u-\Pi(u)},
% \qquad
% \mathcal{G}(u)=u^{M,ref}_{next} + u^{R,ref}_{next},
% \]
% and
% \[
% \hat G_\theta(u)=u^{M}_{next} + u^{R}_{next}
% \]
% is the corresponding one-step IMNO prediction, one has
% \[
% \sup_{u\in K}\Bigl(
% \|u^{M}_{next} - u^{M,ref}_{next}\|_H + \|u^{R}_{next} - u^{R,ref}_{next}\|_H
% \Bigr)<\varepsilon.
% \]
% In particular,
% \[
% \sup_{u\in K}\| G_\theta(u)-\mathcal{G}(u)\|_H<\varepsilon.
% \]
% \end{theorem}

\begin{theorem}[Universal approximation of IMNO on systems with inertial manifold structure]
\label{thm:imno-uat}
Let $H=H^s(\mathbb{T}^d;\mathbb{R}^c)$ for some $s\ge 0$, and let
$$
\mathcal{G}=S(\Delta t):H\to H
$$
be the continuous one-step solution operator of a dissipative PDE. Let $K\subset H$ be compact. Assume that:
\begin{enumerate}

    \item the PDE admits a finite-dimensional flow-normally hyperbolic inertial manifold
    $\mathcal M\subset H$;

    \item the associated asymptotic phase map
    $\Pi:H\to \mathcal M$ is continuous.
\end{enumerate}
Then, for every $\varepsilon>0$, there exists an IMNO model $\mathcal{G}_\theta:H\to H$
such that for every $u\in K$, if
\[
    u = u^M + u^R,
    \qquad
    u^M := \Pi(u),
    \qquad
    u^R := u-\Pi(u),
\]
and
\[
    \mathcal{G}(u)
    =
    u^{M,\mathrm{ref}}_{\mathrm{next}}
    +
    u^{R,\mathrm{ref}}_{\mathrm{next}},
    \qquad
    u^{M,\mathrm{ref}}_{\mathrm{next}}
    :=
    \Pi(\mathcal{G}(u)),
    \qquad
    u^{R,\mathrm{ref}}_{\mathrm{next}}
    :=
    \mathcal{G}(u)-\Pi(\mathcal{G}(u)),
\]
while the corresponding IMNO prediction is written as
\[
    \mathcal{G}_\theta(u)
    =
    u^M_{\mathrm{next}}
    +
    u^R_{\mathrm{next}},
\]
one has
\[
    \sup_{u\in K}
    \left(
    \|u^M_{\mathrm{next}}-u^{M,\mathrm{ref}}_{\mathrm{next}}\|_H
    +
    \|u^R_{\mathrm{next}}-u^{R,\mathrm{ref}}_{\mathrm{next}}\|_H
    \right)
    < \varepsilon .
\]
In particular,
\[
    \sup_{u\in K}
    \|\mathcal{G}_\theta(u)-\mathcal{G}(u)\|_H
    < \varepsilon .
\]
\end{theorem}

The proof is deferred to Appendix A.

\begin{remark}
The universal approximation result above only requires the additive conditioning
mechanism described in the main architecture. This shows that, at the level of
expressive power, additional conditioning mechanism is not necessary. However, empirically, we find it beneficial to enrich the conditioning by
adding a channel-wise multiplicative modulation. More precisely, let
\(\gamma_\ell\) be another learnable MLP mapping \(\mathbb{R}^{d_h}\) to
\(\mathbb{R}^{d_c}\). We then define
\[
\mathcal{L}_\ell(v)(x)
=
\sigma_\ell\!\Bigl(
\bigl(1+\tanh(\gamma_\ell(h_{t+1}))\bigr)\odot
\bigl(
W_\ell v(x)+(\mathcal{K}_\ell v)(x)
\bigr)
+
\beta_\ell(h_{t+1})
\Bigr).
\]
Here \(\odot\) denotes channel-wise multiplication, and both
\(\gamma_\ell(h_{t+1})\) and \(\beta_\ell(h_{t+1})\) are broadcast uniformly over
the spatial domain. In practice, this modification adds only a small number of
parameters while improving prediction accuracy across many test settings. Therefore, we adopt this mechanism in all
numerical experiments below.
\end{remark}

% Concretely, $\Phi_\theta$ follows the latent fourier neural operator (FLNO) paradigm:
% the residual field $u_{R,t}$ is first lifted to a higher-dimensional feature space,
% then processed through a sequence of nonlocal operator layers that act in
% Fourier space, and finally projected back to the physical space.
% The latent coordinate $h_{t+1}$ is injected into the operator layers as a
% global conditioning variable, allowing the residual dynamics to adapt to the
% current location on the learned manifold.

% This structure enables $\Phi_\theta$ to capture nonlocal interactions in the
% residual field while remaining explicitly conditioned on the slow manifold
% dynamics.
% From the perspective of inertial-manifold theory, $\Phi_\theta$ provides a
% data-driven approximation of the transverse adjustment along stable leaves,
% with the residual evolving in response to changes in the manifold
% state.

\subsection{Loss Design}\label{sec:loss}
Our loss is designed to encode the inertial-manifold tracking property
discussed above. When the dynamics admit an inertial manifold and are
asymptotically complete, each trajectory is exponentially tracked by a
(unique) trajectory on the manifold, and the residual decays
exponentially fast in time. In the notation of \eqref{eq:decomp}, this means
that the residual $u^R(t) = u(t) - u^M(t)$ should become smaller and smaller as
$t$ increases. Consequently, for sufficiently large $t$, the manifold component $u^M(t)$
is expected to be close to the reference solution $ u^{ref}(t)$.

To encourage this behavior during training, we augment the standard one-term
prediction loss with an auxiliary term that penalizes the mismatch between the
manifold component $u^M(t)$ and the reference solution $u^{ref}(t)$, and we assign
this term an exponentially increasing weight over the rollout time. So the loss function $\mathcal L$ is defined as:
% \begin{equation}\label{eq:loss_total}
% \mathcal L
% =
% \sum_{t=t_0+1}^{t_{\mathrm{train}}}\!\!\|\hat u(t) - u(t)\|_2^2
% \;+\;
% \alpha \sum_{t=t_0+1}^{t_{\mathrm{train}}}\!\! w_t\,\|\hat u(t) - u^M(t)\|_2^2,
% \end{equation}
\begin{equation}\label{eq:loss_total}
\mathcal L
=
\sum_{t=t_0+1}^{t_{\mathrm{train}}}\!\!\| u^{ref}_t - u_t\|_2^2
\;+\;
\alpha \sum_{t=t_0+1}^{t_{\mathrm{train}}}\!\! w_t\,\| u^{ref}_t - u^M_t\|_2^2,
\end{equation}
with
\begin{equation}\label{eq:time_weight}
w_t
=
\frac{\min\{\gamma^{\,t-(t_0+1)},\, w_{\max}\}}{w_{\max}} \in (0,1].
\end{equation}

The weight \(w_t\) is chosen to increase exponentially over time, so as to place greater emphasis on the auxiliary term at later rollout steps and reflect the expected decrease of the mismatch \(\| u^{ref}(t)-u^M(t)\|_2\); the cap \(w_{\max}\) prevents the weight from growing without bound.

\section{Shift-Equivariant Inertial Manifold Neural Operator}

\subsection{Translational symmetry in PDEs}

Many PDEs posed on homogeneous domains with periodic boundary conditions (or on $\mathbb{R}^d$) are equivariant under spatial translations, or shifts. In other words, if the initial condition is shifted in space, then the corresponding solution is shifted by exactly the same amount. We now formalize this symmetry structure rigorously:

Let $T_s$ denote the translation operator acting on functions by
\begin{equation}
(T_s u)(x) = u(x+s), \qquad s \in \mathbb{R}.
\end{equation}

If the governing PDE is translation-invariant, its solution operator (flow map) $S(t):H\to H$ is defined by
\[
S(t)u_0 := u(t;u_0),
\]
where $u(t;u_0)$ is the solution at time $t$ with initial data $u_0$. Then, it satisfies
\begin{equation}
S(t) \circ T_s = T_s \circ S(t),
\end{equation}
for all admissible shifts $s$.

Consequently, the global attractor of the system contains entire group orbits under the action of translations. For any trajectory $u(t)$ lying on the attractor, all of its translated versions $T_s u(t)$ also lie on the attractor.

% From the perspective of reduced-order modeling, translation symmetry introduces a meaningful prior structure of the system. States that differ only by a spatial shift are dynamically equivalent, yet they are represented as distinct points in the original function space. This redundancy complicates the learning of a low-dimensional parameterization. To construct a more compact representation, it is natural to separate spatial shift from the original dynamics.

From the perspective of reduced-order modeling, translation symmetry introduces a meaningful prior structure of the system. States that differ only by a spatial shift are dynamically equivalent, yet they are represented as distinct points in the original function space. This redundancy complicates the learning of a low-dimensional parameterization. So, separating group motion from intrinsic dynamics can lead to more compact and physically meaningful reduced representations~\cite{mendible2020dimensionality, rowley2000reconstruction,shuai2025symmetry}. Motivated by this idea, we separate the spatial shift from the intrinsic dynamics to construct a more compact model.

\begin{remark}
Standard FNO can naturally preserve translation equivariance when no explicit
coordinate embedding is used. This is because spectral convolutions and pointwise
channel mixing commute with spatial translations. For PDEs with translation symmetry, an FNO without coordinate embedding naturally preserves translation equivariance and often achieves better accuracy. However, the original IMNO does not automatically preserve this
symmetry, since the latent coordinate may entangle intrinsic shape
information with spatial phase. This motivates the shift-equivariant IMNO
developed below.
\end{remark}

\subsection{Symmetry reduction via phase alignment}

To factor out translation symmetry, we fix the phase of a reference Fourier mode and use it to define a canonical representative for each solution $u \in H$.

Let $u(x,t)$ be a periodic solution on $[0,L]$. Its canonical representative $\tilde{u}(x,t)$ is defined by
\begin{equation}
\tilde{u}(x,t) = T_{-{\frac{L\phi(t)}{2\pi}}} u(x,t),
\end{equation}
where the phase $\phi(t)$ is chosen so that a prescribed Fourier mode (e.g., the first mode) satisfies a fixed phase condition. This gives the decomposition
\begin{equation}
u(x,t) = T_{{\frac{L\phi(t)}{2\pi}}} \tilde{u}(x,t),
\end{equation}
where:
\begin{itemize}
    \item $\tilde{u}(x,t)$ represents the intrinsic shape dynamics in a canonical frame,
    \item $\phi(t)$ captures the spatial translation along the symmetry direction.
\end{itemize}

We can make the quotient–phase structure explicit. Denote $\mathcal{T}_{\theta}=T_{\frac{L\theta}{2\pi}}$. The translation group is
\[
G=\{\mathcal{T}_{\theta}:\theta\in[0,2\pi)\}\cong \mathbb{S}^1,
\]
acting on $H$. If
\[
u(x)=\sum_{n\in\mathbb{Z}} \hat u_n e^{in\frac{2\pi x}{L}},
\]
then
\[
(\mathcal{T}_{\theta} u)(x)=\sum_{n\in\mathbb{Z}} e^{in\theta}\hat u_n e^{in\frac{2\pi x}{L}},
\qquad \theta\in[0,2\pi).
\]
% In this case, $H$ can be written as a direct sum of the group direction and the quotient space, i.e.,
% \[
% H \cong G \oplus H/G.
% \]
% The quotient space is
% \[
% H/G=\{[u]:u\in H,\ \hat u_1\in\mathbb{R}\}.
% \]

Therefore, if the first Fourier coefficient of the state is nonzero, after fixing a phase condition, the state can be represented by a
translation phase together with a quotient representative, i.e.,
\[
H_* \cong \mathbb{S}^1 \times \Sigma,
\qquad \Sigma \cong H_*/G,
\]
where \(H_*=\{u\in H:\widehat u_1\neq 0\}\) and \(\Sigma=\{u\in H:\widehat u_1\in \mathbb{R^+}\}\).

Define the phase functional $\Theta:H_* \to \mathbb{S}^1$ by
\[
\Theta(u)=\arg(\hat u_1)=:\theta_0.
\]
Then the canonical projection operator $\pi:H_* \to H/G$ is
\[
\pi(u)=\tilde u=\mathcal{T}_{-\Theta(u)}u.
\]
Then, the reduced dynamics of $\tilde u$ evolves on $H/G$, while the phase $\phi(t)=\Theta(u(x,t))\in\mathbb{S}^1$ tracks the spatial translation phase.

In this quotient space, the inertial-manifold structure of the original dynamics is preserved under the symmetry reduction. If the governing equation is shift equivariant and a $G$-invariant inertial manifold $\mathcal{M}\subset H^*$ exists for the original dynamics, then an inertial manifold for the induced dynamics on $H/G$ is given by its projection
\[
\mathcal{M}_{H/G} := \pi(\mathcal{M}) \subset H/G.
\]

\subsection{Inertial manifold decomposition}

The classical IMNO formulation models the solution as
\begin{equation}
u(t) = \mathcal{D}_\theta(h(t)) + u^R(t),
\end{equation}
where $D_\theta$ parameterizes an approximate inertial manifold $\mathcal{M_\theta}$, $h(t)$ denotes the latent coordinate, and $u^R(t)$ represents the residual component. 

To leverage translation symmetry, we use $D_\theta$ to parameterize the inertial manifold $\mathcal{M}_{H/G}$ in the quotient space, with $h$ as the corresponding latent coordinate. We then extend this decomposition to an equivariant form:
\begin{equation}
u(t) = \mathcal{T}_{\phi(t)} \big( \mathcal{D}_\theta(h(t)) \big) +  u^R(t),
\end{equation}
where:
\begin{itemize}
    \item $h(t) \in \mathbb{R}^{d_h}$ is the latent coordinate that parameterizes the manifold,
    \item $\phi(t) \in \mathbb{R}$ is the spatial translation phase of $u(t)$,
    \item $ u^R(t)$ represents the residual component.
\end{itemize}

So, given an initial condition $u(0)$, we first apply phase alignment to decompose it into a canonical representative and a phase,
\[
u(0) = T_{\phi(0)} \tilde{u}(0),
\]
where $\tilde{u}(0)$ represents the canonical representative, and $\phi(0)$ records the spatial shift. The encoder is then applied to $\tilde{u}$ to obtain the initial latent coordinate,
\[
h(0) = \mathcal{E}_\theta(\tilde{u}(0)).
\]

Then, the latent dynamics is modeled as
\begin{equation}
\dot{h} = f_\theta(h),
\end{equation}
while the phase evolution is governed by
\begin{equation}
\dot{\phi} = g_\theta(h),
\end{equation}
where both $f_\theta$ and $g_\theta$ are parameterized by neural networks.

This construction can be viewed as learning an equivariant inertial manifold in the space $H/G$. By explicitly separating the spatial translation, the model avoids representing entire group orbits within the latent coordinate and thus achieves a more compact representation.

% \subsection{Architectural realization}

% The shift-equivariant IMNO is implemented through the following components:

% \paragraph{Canonical encoding.}
% At the first time step, the solution is aligned to the canonical frame via phase alignment. The encoder $\mathcal{E}_\theta$ first applies an FFT to the aligned field $\hat{u}$ and then maps the Fourier coefficients through an MLP to produce the latent coordinate $h$. Because the encoder operates in Fourier space, it is resolution free and can accept inputs at arbitrary grid resolutions.

% \paragraph{Latent evolution.}
% The intrinsic manifold dynamics is modeled through a neural ODE for $h(t)$, while the phase increment is predicted as a function of the latent state.

\subsection{Residual dynamics}
The evolution of the residual component $u^R$ is challenging, as its update depends not only on the current residual component but also on the latent coordinate$h$ and the translation phase $\phi$.
In our model, the residual dynamics are modeled by
% \begin{equation}
%  u^R_{t+1} = \Phi_\theta(u^R_t, h_t, \phi_t),
% \end{equation}
\begin{equation}
 u^R(t+\Delta t) = \mathcal{R}_\theta(u^R(t), h(t+\Delta t), \phi(t+\Delta t)),
\end{equation}
where $\mathcal{R}_\theta$ denotes a learned residual update operator.

To incorporate the phase information into the residual update while preserving shift equivariance, we modify the positional input embedding used in the lifting layer. This differs from the standard coordinate embedding $[u^R(x),x]$, where the absolute coordinate $x$ is provided directly and may break translation equivariance. Instead, we replace the absolute coordinate by a phase-dependent spatial feature that transforms consistently with the solution under shifts. Specifically, for each grid point $x$, we define
\begin{equation}
\Psi(\phi)(x)
=
\cos\!\left(\phi+\frac{2\pi}{L}x\right),
\end{equation}
and concatenate this feature pointwise with the residual field:
\begin{equation}
[u^R_t(x),\ \Psi(\phi_{t+1})(x)].
\end{equation}
This augmented field is then passed through the same conditional Fourier-layer residual operator introduced in Section~3.2.2, with the Fourier layers conditioned on the latent coordinate $h_{t+1}$.

The key point is that the phase feature transforms consistently with spatial shifts. Under a translation $\mathcal{T}_\delta$, the phase changes from $\phi$ to $\phi+\delta$, and the embedding satisfies
\begin{equation}
\Psi(\phi+\delta)=\mathcal{T}_\delta\Psi(\phi).
\end{equation}
Since the remaining Fourier-layer operations commute with spatial translations, the residual operator preserves shift equivariance:
\begin{equation}
\mathcal{R}_\theta\bigl(\mathcal{T}_\delta u^R_t,h_{t+1},\phi_{t+1}+\delta\bigr)
=
\mathcal{T}_\delta\mathcal{R}_\theta(u^R_t,h_{t+1},\phi_{t+1}).
\end{equation}

\paragraph{Reconstruction.}
Once $u^R_t$, $\phi_t$, and $h_t$ have been obtained, the full solution is reconstructed as
\begin{equation}
u_t = u^R_t + \mathcal{T}_{\phi_t} \mathcal{D}_\theta(h_t).
\end{equation}
The decoder $\mathcal{D}_\theta$ is parameterized by a neural network that maps the latent coordinate $h$ to a set of low-frequency Fourier coefficients. Applying an inverse FFT then gives the corresponding field in physical space. To ensure that the decoded manifold component is represented in the canonical frame, we impose a gauge condition on the first Fourier mode: its imaginary part is set to zero and its real part is replaced by its absolute value so that it is nonnegative. So, the first Fourier coefficient of $\mathcal{D}_\theta(h)$ is real and nonnegative. Consequently, $\mathcal{D}_\theta(h_t)$ lies in the chosen canonical frame, while the translation operator $\mathcal{T}_{\phi_t}$ restores the spatial phase of the full solution.

Overall, by construction, the overall architecture is strictly shift-equivariant. For any spatial shift applied to the input, the predicted output shifts accordingly. This shift-equivariant IMNO reduces redundancy in the latent representation and improves the compactness of the learned model.

\section{Numerical Experiments}

In this section, we demonstrate the performance of IMNO and its shift-equivariant variant (hereafter abbreviated as IMNO-SE) on the Burgers, nonlocal Burgers, Kuramoto--Sivashinsky, and Navier--Stokes equations, and compare them with FNO and RNO baselines.

For a fair comparison, we use the same Fourier modes and channel width
across all Fourier layers. In particular, FNO uses four Fourier layers with
GeLU activations, whereas RNO uses one layer in each block; both FNO
and RNO retain $k_{\max}=24$ Fourier modes and use channel width $d_c=64$ in 1D cases, $k_{\max}=12$ and $d_c=20$ in 2D cases.

For IMNO, we adopt the same Fourier modes $k_{\max}$ and width
$d_c$ in the Fourier layers. Unless otherwise specified, the latent dimension
is $d_h=32$. In one-dimensional experiments, the encoder retains the lowest $k_E=6$
Fourier modes to extract global features, while the decoder generates the
manifold component using the lowest $k_D=8$ Fourier modes. In two-dimensional
experiments, we use $k_E=3$ and $k_D=4$. For the loss function, we set $\gamma = 1.1$, and $w_{\max}=10$
in \eqref{eq:loss_total}--\eqref{eq:time_weight}. 

% Unless otherwise specified, all remaining IMNO hyperparameters (e.g., encoder
% widths) are kept fixed across experiments.

To exploit the symmetry structure of the underlying PDEs, we use shift-equivariant versions of
IMNO, FNO, and RNO whenever the governing equation is shift equivariant. For PDEs without
translation symmetry, we instead use the original versions of these
models.

For the parameter $\alpha$, we evaluate multiple values in our experiments.
Intuitively, $\alpha$ controls how strongly the manifold component is encouraged to approximate the output function.
When $\alpha=0$, the latent coordinate $h$ only serves as a time-evolving variable that encodes the global information of the solution, and the reconstructed manifold component $u^M$ is not required to form
an explicit part of the solution.
As $\alpha$ increases, the model is increasingly encouraged to represent the reconstructed solution $u = u^M + u^R$ primarily through the manifold component $u^M$.

For the Burgers, 1D Navier--Stokes and 2D compressible Navier--Stokes equations, all training and testing data are taken from the
publicly available PDEBench repository \cite{takamoto2022pdebench}.
For other equations, we generate the dataset using our
own simulation code, following the same discretization and data format as
PDEBench to ensure a consistent training and evaluation pipeline. For all experiments, the dataset is split into training, validation, and test sets with a ratio of 8:1:1. The validation set is used for model selection during training. For each model and experimental configuration, we repeat the training with three fixed random seeds. For each seed, we select the model checkpoint that achieves the lowest validation error and record its error on the test set. The results reported below are the mean and standard deviation of the test errors over the three seeds. We train and evaluate the models under the autoregressive training setup. Starting from the initial condition $u_0$,
the model predicts $ u(t_1)$, and then uses $u(t_1)$ to predict $u(t_2)$. This procedure is
iterated so that each future state $u(t_{n})$ is generated solely from
previous predictions, yielding a full autoregressive rollout up to a prescribed final
time $t_{\mathrm{end}}$. During training, the predictions at all rollout steps are
included in the loss function.

To quantitatively evaluate accuracy over the rollout interval
$[1,\, t_{\mathrm{end}}]$, we employ a normalized relative $L^2$ error over the spatial domain:

$$
Error
=
\sqrt{
\frac{
\sum_{n=1}^{t_{\mathrm{end}}}
\sum_{j=1}^{N_x}
\big( u(x_j,t_n) - u^{ref}(x_j,t_n) \big)^2
}{
\sum_{n=1}^{t_{\mathrm{end}}}
\sum_{j=1}^{N_x}
u^{ref}(x_j,t_n)^2
}
}.
$$

In practice, the error is computed for each sample and then averaged over the evaluation batch.

\subsection{Burgers Equation}

The one-dimensional Burgers equation is a classical nonlinear PDE occurring in many fields, such as fluid mechanics and traffic flow. It takes the form

$$
\begin{aligned}
\partial_t u(x, t)+\partial_x\left(u^2(x, t) / 2\right) & =\nu \partial_{x x} u(x, t), & & x \in(0,1) \\
u(x, 0) & =u_0(x), & & x \in(0,1)
\end{aligned}
$$

with periodic boundary conditions. 

For the training procedure, 
we employ a dataset consisting of
$N_{\mathrm{data}} = 10000$ samples, each resolved on a spatial grid of
$N_x = 256$ points with spacing $\Delta x = \tfrac{1}{256}$.

In the temporal direction, the solution is sampled every $\Delta t = 0.05$,
yielding $N_t = 41$ snapshots from $t = 0$ to $t= 2$.  
During training and evaluation, the model is provided only the initial snapshot
$u(\cdot,0)$ and must autoregressively predict all subsequent time steps
$u(\cdot,t_1), u(\cdot,t_2), \dots, u(\cdot,t_{\mathrm{end}})$.
We train and evaluate the models under two rollout horizons
($t_{\mathrm{end}} = 0.5$ and $t_{\mathrm{end}} = 2$),
and report the results in Table~\ref{Burgers2}. 

% \begin{table}[H]
% \centering
% \caption{Relative $L^2$ error and number of parameters of different neural operators at two rollout horizons}
% \label{Burgers2}
% \begin{tabular}{c|ccccccc}
% \toprule
%  & LNO & FNO & LFNO & RLNO & RLFNO & IMNO & IMNO-SE\\
% \midrule
% Params 
% & 132097 & 418497 & 525313 & 138433 & 531649 & 482673 & 529610\\
% \midrule
% \multicolumn{8}{c}{$t_{\mathrm{end}} = 0.5$} \\
% \midrule
% $\nu = 0.002$ & 16.45\% & 3.21\% & 2.55\% & 15.27\% & 2.29\% & 2.20\% &2.41\% \\
% $\nu = 0.01$  & 9.29\% & 1.72\% & 1.06\% & 9.47\% & 1.05\% &0.97\%  &1.31\% \\
% \midrule
% \multicolumn{8}{c}{$t_{\mathrm{end}} = 2$} \\
% \midrule
% $\nu = 0.002$ & - & 9.26\% & - & 14.77\% & 3.02\% & 2.18\% &2.65\%\\
% $\nu = 0.01$  & - & - & - & 8.66\% & 1.07\% &0.85\% &1.49\%\\
% \bottomrule
% \end{tabular}
% \end{table}

\begin{table}[H]
\centering
\caption{Relative $L^2$ error of FNO, RNO and IMNO-SE for the Burgers Equation}
\label{Burgers2}
\begin{tabular}{c|cccc}
\toprule
 & FNO  &RNO  & IMNO-SE($\alpha =0.1$) & IMNO-SE($\alpha =0$)\\
\midrule
Params 
& 418433 &623533& 538314 & 538314\\
\midrule
\multicolumn{5}{c}{$t_{\mathrm{end}} = 0.5$} \\
\midrule
$\nu = 0.002$ & \(2.91 \pm 0.07\)\% & \(3.55 \pm 0.05\)\%  &\(2.25 \pm 0.03\)\% &\(1.94 \pm 0.04\)\%\\
$\nu = 0.01$  & \(1.35 \pm 0.08\)\%  &\(1.79 \pm 0.03\)\%  &\(1.16 \pm 0.03\)\% &\(0.78 \pm 0.03\)\%\\
\midrule
\multicolumn{5}{c}{$t_{\mathrm{end}} = 2$} \\
\midrule
$\nu = 0.002$ & \(53.3 \pm 43.8\)$\%^*$ & \(4.12 \pm 0.92\)\% &\(2.45 \pm 0.12\)\% &\(1.95 \pm 0.12\)\%\\
$\nu = 0.01$  & \(67.4 \pm 25.8\)$\%^*$ & \(1.80 \pm 0.11\)\%  &\(1.46 \pm 0.10\)\% &\(0.85 \pm 0.02\)\%\\
\bottomrule
\end{tabular}

\vspace{0.5em}
\begin{minipage}{0.78\textwidth}
\footnotesize
\textit{Note.} ``$^{*}$'' indicates unstable training leading to degraded performance. 
\end{minipage}
\end{table}

% \noindent\textit{Note.} The symbol $*$ indicates unstable training, where the loss does not decrease steadily but suddenly increases at some epoch, leading to degraded final performance. The symbol $-$ indicates that training is highly unstable, resulting in either NaN values or a relative error larger than $1$.

From Table~\ref{Burgers2}, we observe that IMNO-SE achieves the best overall
accuracy across both short- and long-horizon rollouts. It remains stable in the
long-horizon setting, unlike FNO, and achieves better performance with fewer
parameters than RNO.

Figure~\ref{fig:burgers_imnosi_examples} shows representative IMNO-SE predictions of $u$ and $u^M$ for two representative initial conditions, compared with the ground truth. We observe that $u^M$ can capture the long-term dynamics of the system.
\begin{figure}[H]
    \centering
    \begin{subfigure}{0.9\textwidth}
        \centering
        \includegraphics[width=\textwidth]{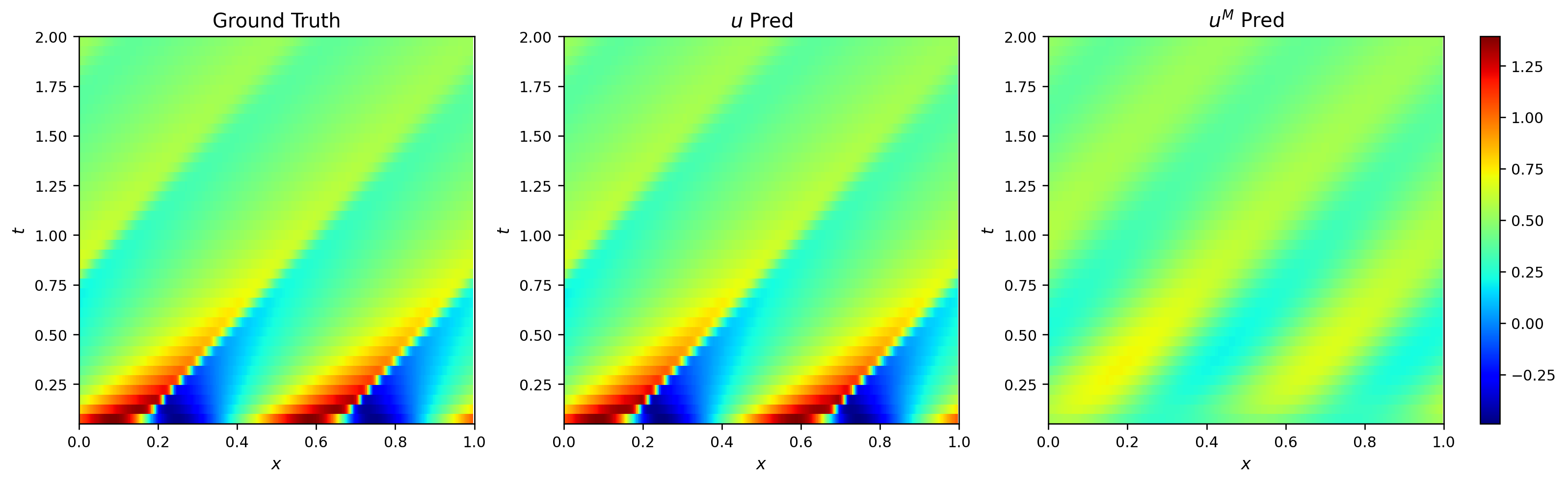}
        \caption{Example 1.}
        \label{fig:burgers_imnosi_case1}
    \end{subfigure}

    \hfill

    \begin{subfigure}{0.9\textwidth}
        \centering
        \includegraphics[width=\textwidth]{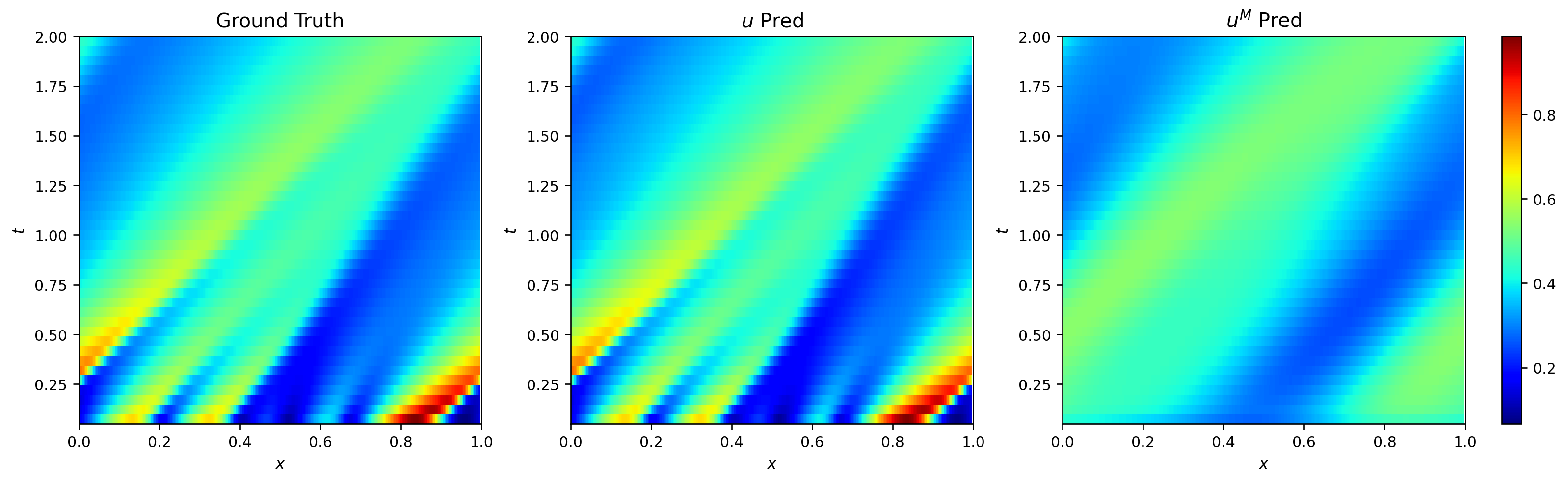}
        \caption{Example 2.}
        \label{fig:burgers_imnosi_case2}
    \end{subfigure}
    \caption{IMNO-SE ($\alpha=0.1$) predictions for 1D Burgers equation ($\nu=0.01$).}
    \label{fig:burgers_imnosi_examples}
\end{figure}

\subsection{Nonlocal Burgers Equation}

We next consider a variant of the Burgers equation, in which the advection speed at each point depends on the whole solution. The equation takes the form
\[
\begin{aligned}
\partial_t u(x, t) + a(t)\,\partial_x u(x,t) &= \nu \partial_{xx} u(x,t) +f(x), 
& & x \in (0,1)\;  \\
u(x,0) &= u_0(x), 
& & x \in (0,1)
\end{aligned}
\]
with periodic boundary conditions. Here, the transport velocity is nonlocal and is defined through the inner product
\[
a(t)=\langle \omega, u(t)\rangle=\int_0^1 \omega(x)\,u(x,t)\,dx,
\]
where $\omega$ is a prescribed kernel. The right-hand side term $f(x)$ represents external forcing. Unlike Burgers equation, the advection at each spatial location depends not only on the local value of $u$ but also on global spatial information. According to \cite{constantin2012integral},  when the parameters satisfy suitable conditions, this equation admits an inertial manifold, providing a rigorous finite-dimensional structure for its long-time dynamics.

For the training procedure, we use the same dataset configuration as in the previous Burgers case. The dataset consists of $N_{\mathrm{data}} = 10000$ samples, each discretized on a spatial grid with $N_x = 256$ points and spacing $\Delta x = \tfrac{1}{256}$. The solution is also sampled with step size $\Delta t = 0.05$, yielding $N_t = 41$ snapshots from $t=0$ to $t=2$. As before, the model is provided only with the initial condition $u(\cdot,0)$ and must autoregressively predict all subsequent steps. We train and evaluate both short- and long-horizon rollouts ($t_{\mathrm{end}} = 0.5$ and $t_{\mathrm{end}} = 2$). 

Compared to the standard Burgers equation, this nonlocal system breaks translation equivariance due to the global coupling through $a(t)$, so we do not report IMNO-SE results here, and both FNO and RNO are evaluated using non-translation-equivariant versions.

In this experiment, we fix the parameters as
\[
\omega(x)=\cos(2\pi x)+0.5\cos(4\pi x), \qquad \nu=0.01.
\]

For the forcing term \(f\), we consider two representative cases.
The first case is the unforced setting \(f=0\), which is relatively simple. In this case, although the equation contains the nonlocal transport term \(\langle \omega,u(t) \rangle u_x\), the advection speed depends only on time through the scalar quantity \(\langle \omega,u(t) \rangle\). Consequently, after a time-dependent spatial shift, the equation is equivalent to a heat equation, and the dynamics are essentially simple without nontrivial long-time structures.

The second case considers a nonzero forcing term. In particular, we take
\[
f(x) = -2\sin(2\pi x) - \sin(4\pi x).
\]
For this choice, by examining the finite-dimensional reduced system associated with the leading Fourier modes, we can verify that the reduced dynamics admit an equilibrium with a two-dimensional unstable manifold \cite{rosa1996inertial}. Since the inertial manifold is also parameterized by the leading Fourier modes in this equation, the induced dynamics on the inertial manifold are already nontrivial. Therefore, the forced case provides a more challenging test setting to examine our model.

The quantitative results are reported in Table~\ref{NonlocalBurgers}. Although FNO and IMNO achieve comparable performance at the short-horizon prediction time \(t=0.5\), FNO exhibits training instability over the longer horizon \(t=2\), which leads to noticeably worse results. By contrast, RNO is also stable over the longer horizon, but it has a larger number of parameters and still yields lower accuracy than IMNO.

\begin{table}[H]
\centering
\caption{Relative $L^2$ error of FNO, RNO and IMNO for the nonlocal Burgers equation}
\label{NonlocalBurgers}
\begin{tabular}{c|cccc}
\toprule
 & FNO & RNO  & IMNO($\alpha=0.1$) & IMNO($\alpha=0$) \\
\midrule
Params 
& 418497 &623617 & 491377 & 491377\\
\midrule
\multicolumn{5}{c}{$t_{\mathrm{end}} = 0.5$} \\
\midrule
$f=0$ & \(0.59 \pm 0.09\)\% & \(0.89 \pm 0.05\)\% & \(0.51 \pm 0.004\)\% &  \(0.43 \pm 0.02\)\% \\
$f\neq 0$  & \(0.77 \pm 0.06\)\% & \(1.10 \pm 0.03\)\% & \(0.54 \pm 0.004\)\% &  \(0.48 \pm 0.03\)\% \\
\midrule
\multicolumn{5}{c}{$t_{\mathrm{end}} = 2$} \\
\midrule
$f=0$ & \(16.60 \pm 20.34\)$\%^*$ & \(1.14 \pm 0.03\)\% & \(0.58 \pm 0.0003\)\% & \(0.41 \pm 0.006\)\%  \\
$f\neq 0$  & \(26.71 \pm 12.89\)$\%^*$ & \(2.36 \pm 0.08\)\% & \(1.16 \pm 0.03\)\% &  \(1.00 \pm 0.06\)\%  \\
\bottomrule
\end{tabular}
\end{table}

Figure~\ref{fig:nonlocal_burgers_examples} shows representative IMNO predictions for both $u$ and $u^M$, compared with the ground truth. We observe that \(u^M\) captures the dominant long-time dynamics, while the residual correction \(u^R\) refines the prediction so that the whole prediction \(u\) remains accurate even through the short-time transient regime. 
\begin{figure}
    \begin{subfigure}{0.9\textwidth}
        \centering
        \includegraphics[width=\textwidth]{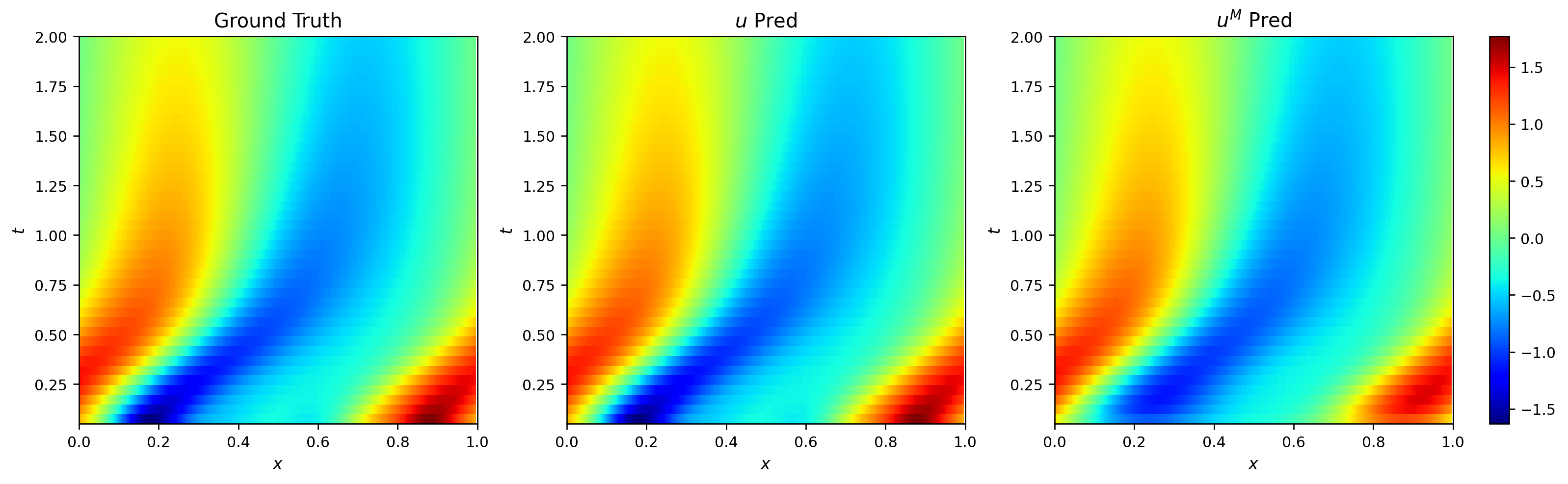}
        \caption{case 1.}
    \end{subfigure}

    \vspace{1.0em}

    \begin{subfigure}{0.9\textwidth}
        \centering
        \includegraphics[width=\textwidth]{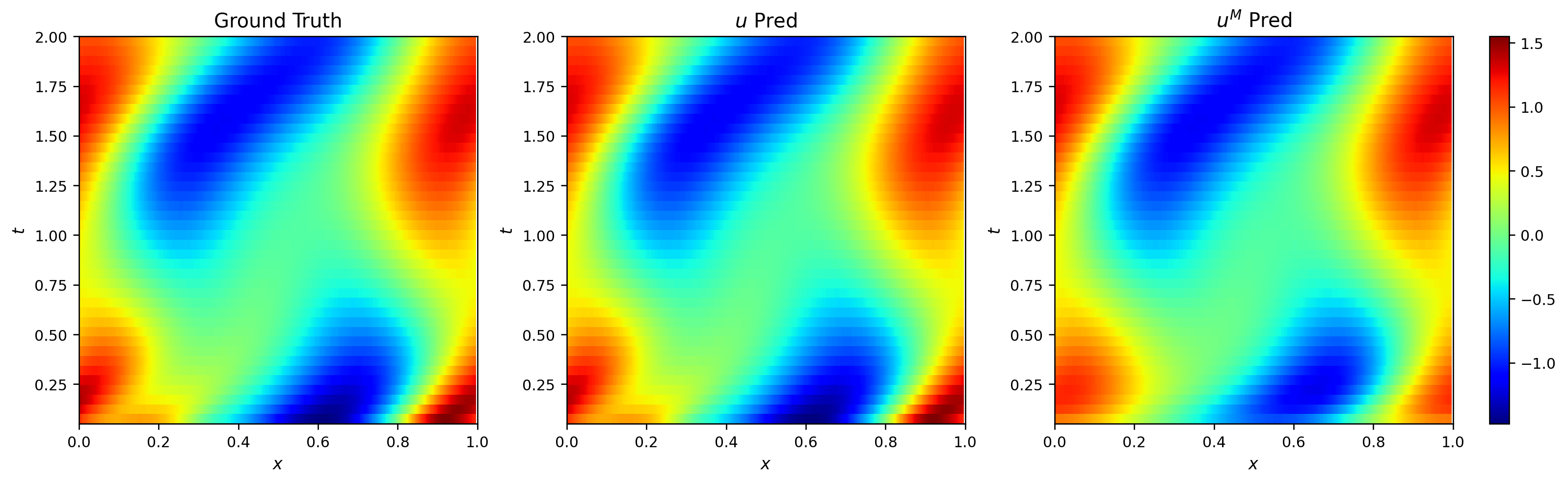}
        \caption{case 2.}
    \end{subfigure}
    \caption{IMNO ($\alpha=0.1$) predictions for the nonlocal Burgers equation.}
    \label{fig:nonlocal_burgers_examples}
\end{figure}

To further demonstrate the structure of the learned latent dynamics, we project the trajectory of the latent coordinate $h$ onto the first three principal components obtained from PCA. This provides a low-dimensional visualization of the learned latent coordinate's evolution and helps illustrate whether it captures the reduced dynamics on the inertial manifold. Figure~\ref{fig:nonlocal_burgers_hidden_pca} shows projected latent trajectories of the two cases.

\begin{figure}[H]
    \centering
    \begin{subfigure}{0.48\textwidth}
        \centering
        \includegraphics[width=\textwidth]{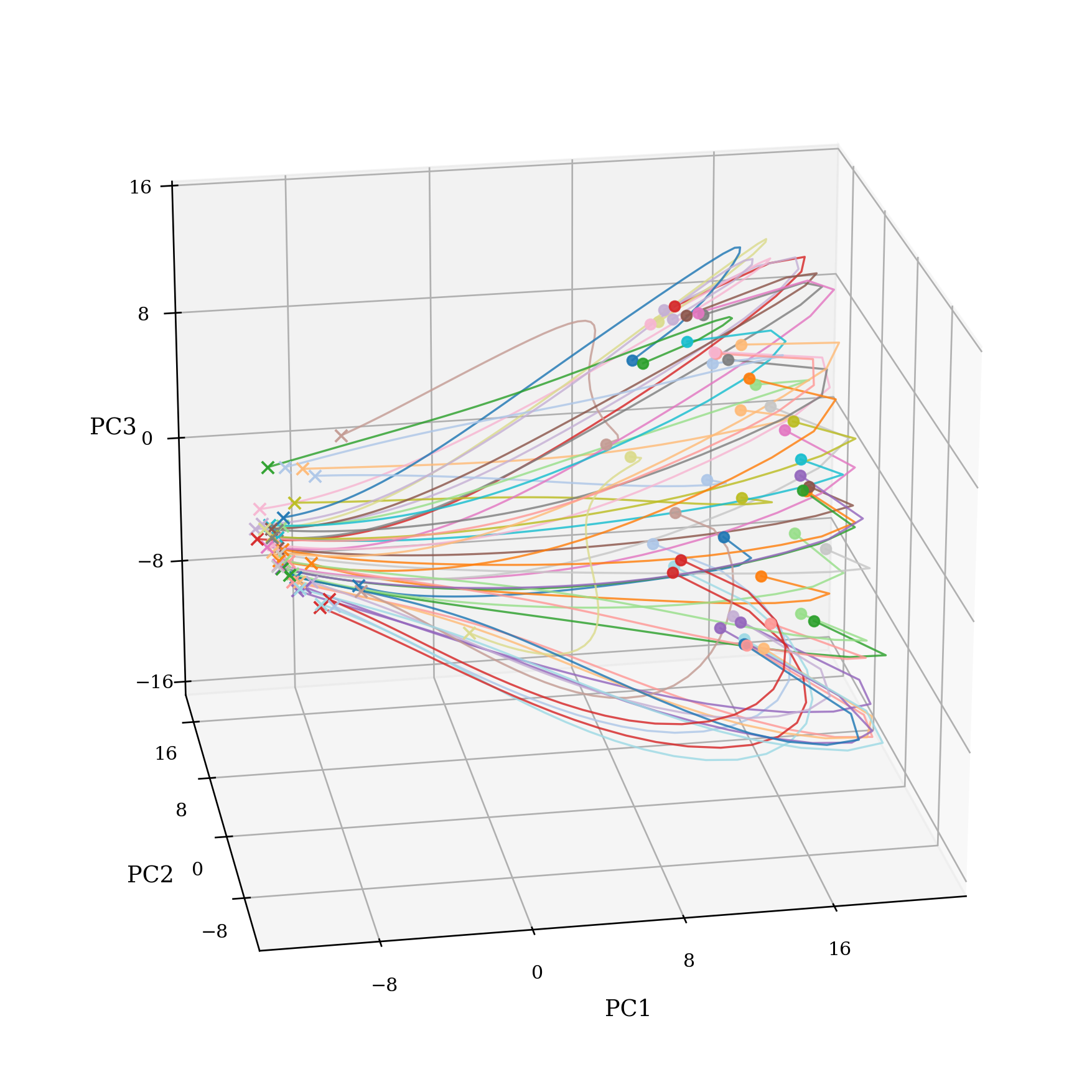}
        \caption{Case 1.}
        \label{fig:nonlocal_burgers_hidden_pca_case1}
    \end{subfigure}
    \hfill
    \begin{subfigure}{0.48\textwidth}
        \centering
        \includegraphics[width=\textwidth]{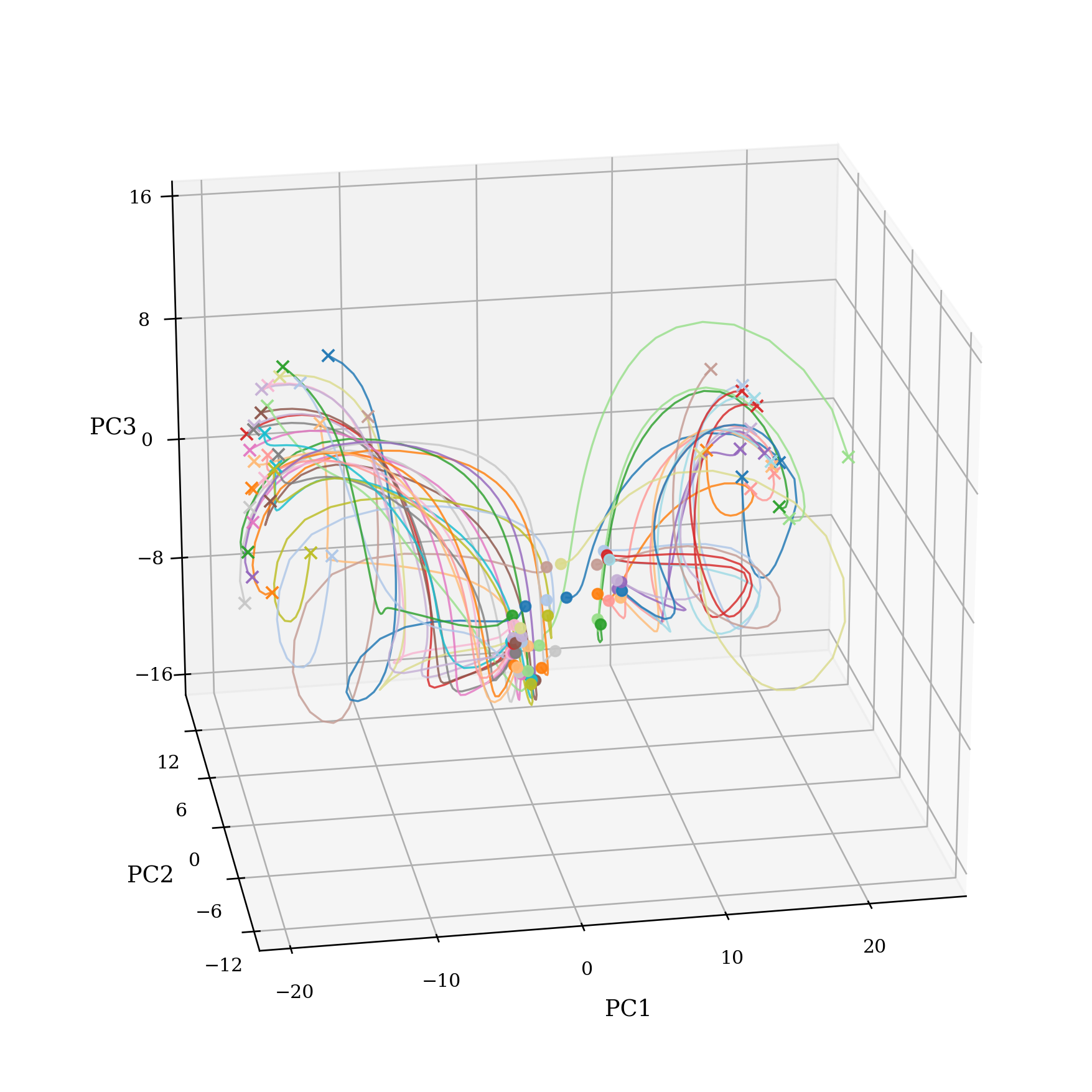}
        \caption{Case 2.}
        \label{fig:nonlocal_burgers_hidden_pca_case2}
    \end{subfigure}
    \caption{Representative hidden-state trajectories projected onto the top three PCA modes for the nonlocal Burgers equation. (Dots indicate the initial states of different trajectories, while crosses mark their corresponding final states at $t=2$.)}
    \label{fig:nonlocal_burgers_hidden_pca}
\end{figure}

For case 1, where the forcing is zero, one can rigorously show that the reduced dynamics on the inertial manifold admits a unique stable fixed point. The projected latent trajectory is consistent with this picture: it evolves toward a single attractor, indicating that the learned latent dynamics correctly captures the simple asymptotic behavior of the system.

For case 2, where the forcing is nonzero, one can likewise show that the reduced dynamics on the inertial manifold admits one saddle point and two stable fixed points. The projected latent trajectories again agree well with this phase-space structure: they approach two different attractors, reflecting the nontrivial dynamics induced by the forcing.

\subsection{Kuramoto--Sivashinsky equation}

We consider the one-dimensional Kuramoto--Sivashinsky (KS) equation
% $$
% \begin{aligned}\label{eq:KS}
% \partial_t u(x,t)
% &= -u \partial_x u(x,t)
% - \partial_{xx} u(x,t)
% - \partial_{xxxx} u(x,t), &&
% \qquad x \in (0,L)
% \\
% u(x,0) &= u_0(x), 
%   && x \in (0,L)
% \end{aligned}
% $$
\begin{equation}
\begin{alignedat}{2}\label{eq:KS}
\partial_t u(x,t)
&= -u(x,t) \partial_x u(x,t)
- \partial_{xx} u(x,t)
- \partial_{xxxx} u(x,t),
&\qquad & x \in (0,L) \\
u(x,0) &= u_0(x),
&\qquad & x \in (0,L)
\end{alignedat}
\end{equation}
with periodic boundary conditions on $[0,L]$.
% We take $L=4\pi$ and $L=8\pi$ and assume
% $u_0 \in \dot L^2_{\mathrm{per}}([0,L];\mathbb{R})$, the space of mean-zero,
% square-integrable periodic functions.

% For any such initial condition, the KS equation generates a global semigroup
% $S_t : \dot L^2_{\mathrm{per}} \to \dot L^2_{\mathrm{per}}$.
In this experiment, data are generated by numerically solving \eqref{eq:KS} using an exponential
time-differencing fourth-order Runge--Kutta (ETDRK4) scheme.
Initial conditions are sampled from a mean-zero Gaussian measure with
covariance
\[
L^{-2}\
\bigl(-\Delta + (49/L^2) I\bigr)^{-2}.
\]
(Equivalently, in Fourier space, the zero mode is set to zero,
\(\widehat u_0=0\). For each positive Fourier mode \(k>0\), the complex
coefficient \(\widehat u_k\) is sampled as a mean-zero Gaussian random
variable with variance
\[
\mathbb E|\widehat u_k|^2
=
L^2\left((2\pi k)^2+49\right)^{-2},
\]
and the negative modes are determined by
\[
\widehat u_{-k}=\overline{\widehat u_k}.
\]
Thus the initial condition is a real-valued, mean-zero Gaussian random field whose Fourier variance scales like \(k^{-4}\) at high wavenumbers.)

For the training procedure, like the previous experiment, we still employ a dataset consisting of
$N_{\mathrm{data}} = 10000$ samples, each resolved on a spatial grid of
$N_x = 256$ points with spacing $\Delta x = \tfrac{L}{256}$.

In the temporal direction, the solution is sampled every $\Delta t = 1$,
yielding $N_t = 41$ snapshots from $t = 0$ to $t= 40$.  
During training and evaluation, the model is provided only the initial condition
$u(\cdot,0)$ and autoregressively predicts all subsequent steps.
We then train and evaluate the models under two horizons ($t_{end} = 20$ and $t_{end} = 40$) for both
$L=4\pi$ and $L=8\pi$. The Table~\ref{KS2} compares the autoregressive prediction
performance of FNO, RNO and IMNO.
\begin{table}[H]
\centering
\caption{Relative $L^2$ error of FNO, RNO and IMNO-SE for the KS equation}
\label{KS2}
\begin{tabular}{c|cccc}
\toprule
 & FNO & RNO & IMNO-SE($\alpha=0.1$) & IMNO-SE($\alpha=0$) \\
\midrule
Params 
& 418433 & 623533 & 538314 & 538314 \\
\midrule
\multicolumn{5}{c}{$t_{\mathrm{end}} = 20$} \\
\midrule
$L = 4\pi$ & \(1.04 \pm 0.20\)\% & \(0.42 \pm 0.01\)\% & \(0.52 \pm 0.02\)\% & \(0.46 \pm 0.003\)\% \\
$L = 8\pi$ & \(1.74 \pm 0.05\)\% & \(1.19 \pm 0.01\)\% & \(4.25 \pm 0.17\)\% & \(2.99 \pm 0.42\)\% \\
\midrule
\multicolumn{5}{c}{$t_{\mathrm{end}} = 40$} \\
\midrule
$L = 4\pi$ & \(87.67 \pm 11.85\)$\%^*$ & \(2.19 \pm 0.23\)\% & \(0.89 \pm 0.05\)\% & \(0.84 \pm 0.02\)\% \\
$L = 8\pi$ & \(94.17 \pm 2.72\)$\%^*$ & \(13.99 \pm 0.71\)\% & \(41.21 \pm 5.63\)\% & \(35.58 \pm 3.63\)\% \\
\bottomrule
\end{tabular}
\end{table}

When $L=4\pi$, the system is not chaotic, and the manifold prediction $u^M$ reconstructs most of the dominant features. Figure~\ref{fig:ks_4pi_examples} shows IMNO-SE predictions for two representative initial conditions, compared with the ground truth. We observe that both the full prediction $u$ and the manifold component $u^M$ are highly accurate in this experiment, indicating that in this setting, solutions evolve towards the inertial manifold very quickly. IMNO-SE learned the dynamics on the inertial manifold accurately.
\begin{figure}[H]
    \centering
    \begin{subfigure}{0.9\textwidth}
        \centering
        \includegraphics[width=\textwidth]{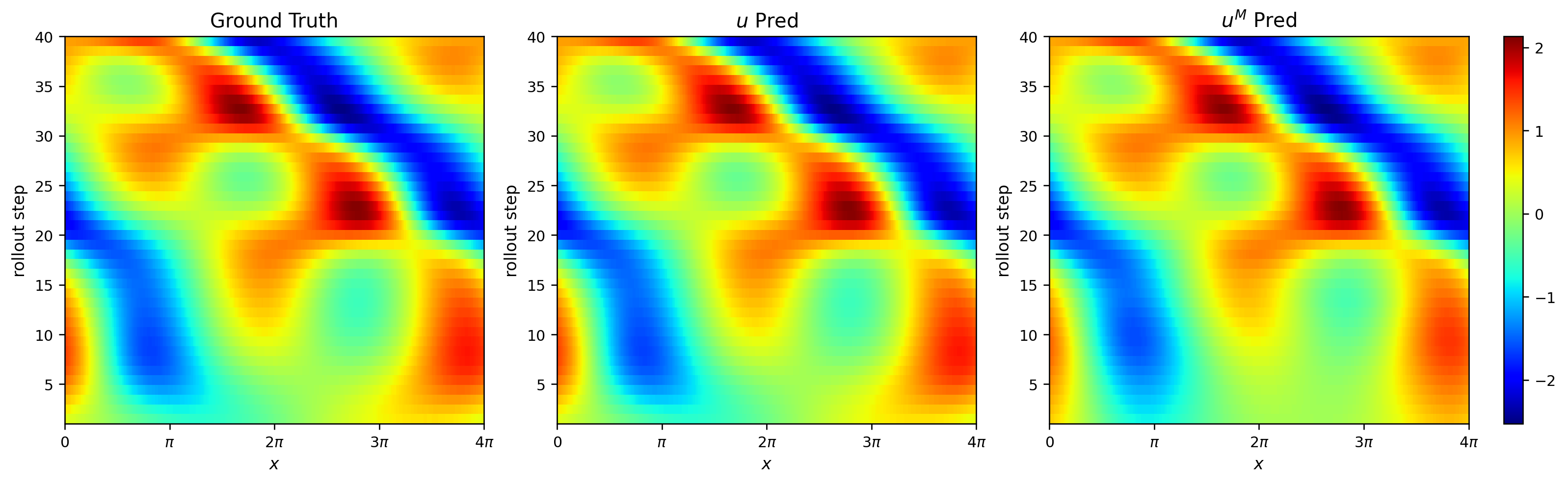}
        \caption{Example 1.}
        \label{fig:ks_4pi_case1}
    \end{subfigure}

    \vspace{1.0em}

    \begin{subfigure}{0.9\textwidth}
        \centering
        \includegraphics[width=\textwidth]{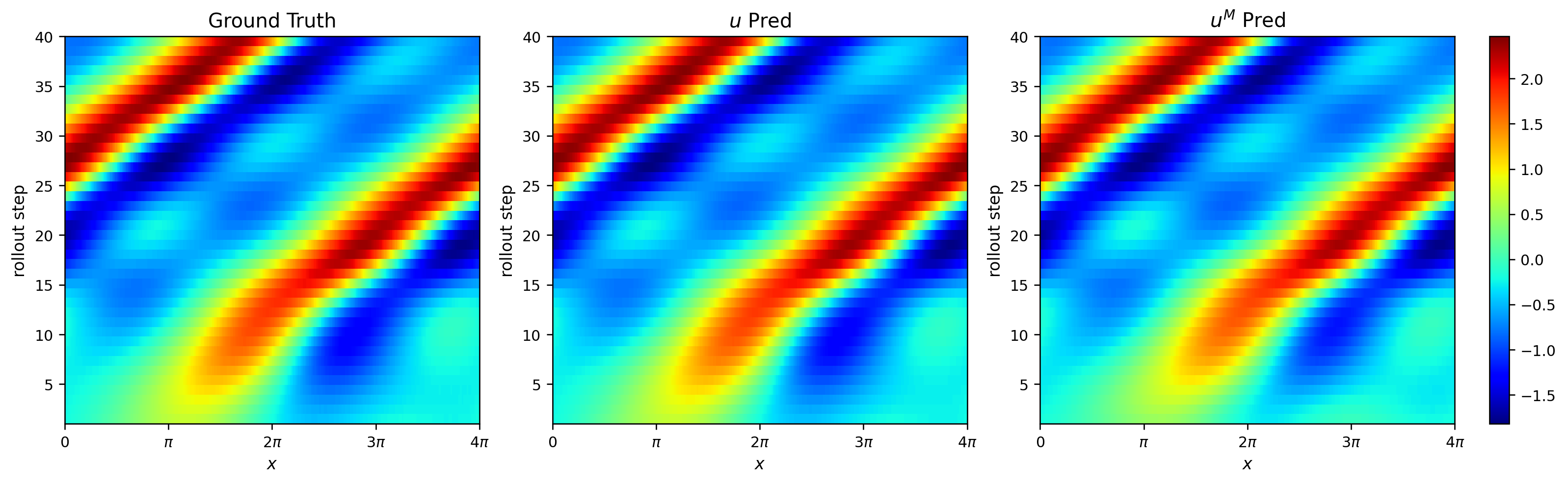}
        \caption{Example 2.}
        \label{fig:ks_4pi_case2}
    \end{subfigure}
    \caption{IMNO-SE ($\alpha=0.1$) predictions for the KS equation at $L=4\pi$}
    \label{fig:ks_4pi_examples}
\end{figure}

Moreover, in this setting, the Kuramoto--Sivashinsky equation lies in the regime of a low-dimensional $1:2$ interaction between the first two Fourier modes and admits two symmetry-related stable traveling waves. These waves are relative equilibria in the full state space, but become stable fixed points after quotienting out translation symmetry. Consistent with this picture, the PCA visualization of the trajectories in the learned latent coordinate $h$ in Figure~\ref{fig:ks_4pi_hidden_pca} clearly approach two different attractors, with oscillatory transients before convergence. This suggests that the latent trajectories have captured the underlying dynamics.

\begin{figure}[H]
    \centering
    \includegraphics[width=0.48\textwidth]{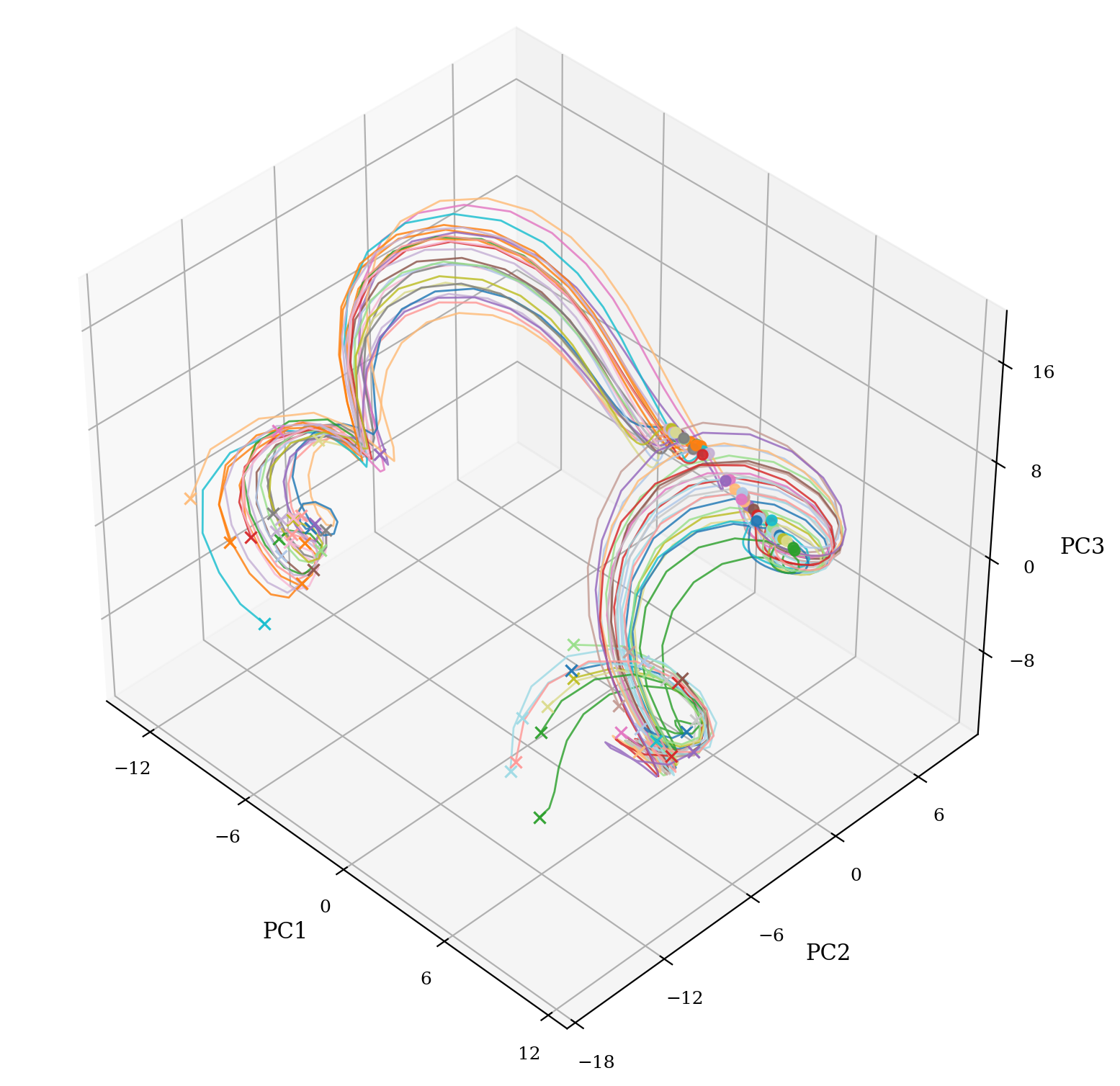}
    \caption{Representative hidden-state trajectories projected onto the top three PCA modes for the KS equation at $L=4\pi$.}
    \label{fig:ks_4pi_hidden_pca}
\end{figure}

However, when $L=8\pi$, the dynamics become chaotic. IMNO-SE's predictive accuracy is lower than that of FNO and RNO. 
In fact, we find that IMNO is less accurate for PDEs with strongly chaotic dynamics, although it still captures some low-dimensional structure effectively. This is consistent with the intrinsic difficulty of predicting chaotic systems: small errors grow rapidly, and the additional approximation error introduced by the reduced manifold dynamics can accumulate during long autoregressive rollout and degrade the overall prediction quality. We will look further into this issue into the experiment of Kolmogorov flow.

\subsection{One-dimensional NS Equation}

We consider the one-dimensional compressible Navier--Stokes equations on the torus
$\Omega = (0,1)$:
% \begin{equation}\label{eq:ns_mass}
% \partial_t \rho + \partial_x(\rho v) = 0,
% \end{equation}
% \begin{equation}\label{eq:ns_momentum}
% \rho(\partial_t v + v\,\partial_x v) = -\partial_x p + \left(\zeta + \frac{4\eta}{3}\ \right)\partial_{xx} v,
% \end{equation}
% \begin{equation}\label{eq:ns_energy}
% \partial_t\!\left[e + \frac{\rho v^2}{2}\right]
%  + \partial_x\!\left[\left(e + p + \frac{\rho v^2}{2}\right) v - v\,\left(\zeta +\frac{4\eta}{3}\right)\partial_{x} v \right] = 0,
% \end{equation}
\begin{align}
\partial_t \rho + \partial_x(\rho v)
&= 0, \\
\rho(\partial_t v + v\,\partial_x v)
+\partial_x p&=
\left(\zeta + \frac{4\eta}{3}\right)\partial_{xx} v,
 \\
\partial_t\!\left(e + \frac{\rho v^2}{2}\right)
+\partial_x\!\biggl[
\left(e + p + \frac{\rho v^2}{2}\right)v
\biggr]
&= \partial_x\!\biggl[
v\left(\zeta + \frac{4\eta}{3}\right)\partial_x v
\biggr].
\end{align}

where $\rho$ is the density, $v$ is the velocity, $p$ is the pressure,
and $e = p/(\Gamma-1)$ is the internal energy density (Here we set $\Gamma=5/3$, corresponding to a monatomic ideal gas). The viscosities $\eta$ and
$\zeta$ denote shear and bulk viscosity, respectively, and are both set to 0.1 in our experiments.

For data generation and model training, we use a dataset from PDEBench \cite{takamoto2022pdebench} consisting of
$N_{\mathrm{data}} = 10000$ samples, each resolved on a spatial grid of
$N_x = 128$ points with spacing $\Delta x=\frac{1}{128}$.
In the temporal direction, the solution is sampled every $\Delta t = 0.02$,
yielding $N_t = 51$ snapshots.
Before training, we normalize the data channel-wise via a linear transformation,
so that the data distribution of different physical variables have the same means and variances.

In this experiment, we train and evaluate the models under two horizons ($t_{\mathrm{end}}=0.5$ and $t_{\mathrm{end}}=1.0$). The results of FNO, RNO, and IMNO-SE are reported in Table~\ref{tab:ns_results_cfd}.

\begin{table}[H]
\centering
\caption{Relative $L^2$ error of FNO, RNO, and IMNO-SE for the 1D NS Equation}
\label{tab:ns_results_cfd}
\begin{tabular}{c|ccccc}
\toprule
& FNO & RNO & IMNO-SE($\alpha=1$)  & IMNO-SE($\alpha=0.1$) & IMNO-SE($\alpha=0$) \\

\midrule
Params & 418819 & 623939 & 619340 & 619340 &  619340\\
\midrule

$t_{\mathrm{end}} = 0.5$ & \(12.05 \pm 13.44\)$\%^*$ & \(3.24 \pm 0.06\)\% & \(5.12 \pm 0.05\)\% & \(3.54 \pm 0.29\)\% & \(2.69 \pm 0.10\)\% \\

$t_{\mathrm{end}} = 1.0$ & \(84.47 \pm 5.14\)$\%^*$ & \(4.79 \pm 0.64\)\% & \(4.49 \pm 0.41\)\% & \(4.33 \pm 0.15\)\% & \(3.53 \pm 0.17\)\% \\

\bottomrule
\end{tabular}
\end{table}

Figure ~\ref{fig:ns_imnose_predictions} shows the IMNO-SE predictions for both $\alpha=0.1$ and $\alpha=1$. We observe a clear trade-off between these two choices of $\alpha$: when $\alpha=0.1$ the prediction of $u^M$ is not desirable, but for $\alpha=1$ the prediction of $u^M$ becomes much better and captures the long-time dynamics. On the other hand, increasing $\alpha$ also leads to a degradation in the overall accuracy of the final prediction $u$. This highlights the importance of the choice of parameter $\alpha$.
\begin{figure}[H]
    \centering
    \begin{subfigure}{0.8\textwidth}
        \centering
        \includegraphics[width=\textwidth]{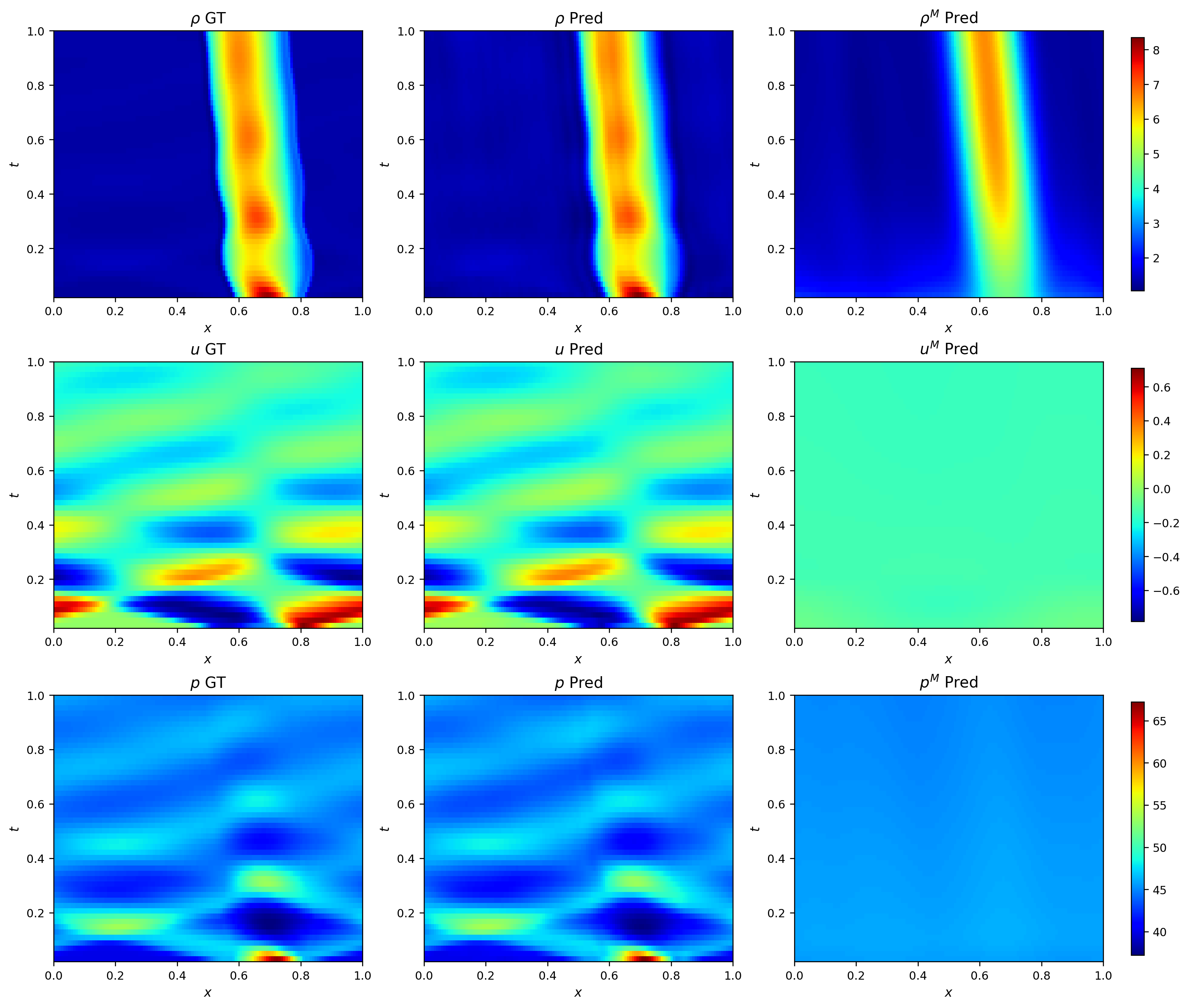}
        \caption{Model prediction($\alpha=0.1$).}
        \label{fig:ns_imnose_u}
    \end{subfigure}

    \hfill

    \begin{subfigure}{0.8\textwidth}
        \centering
        \includegraphics[width=\textwidth]{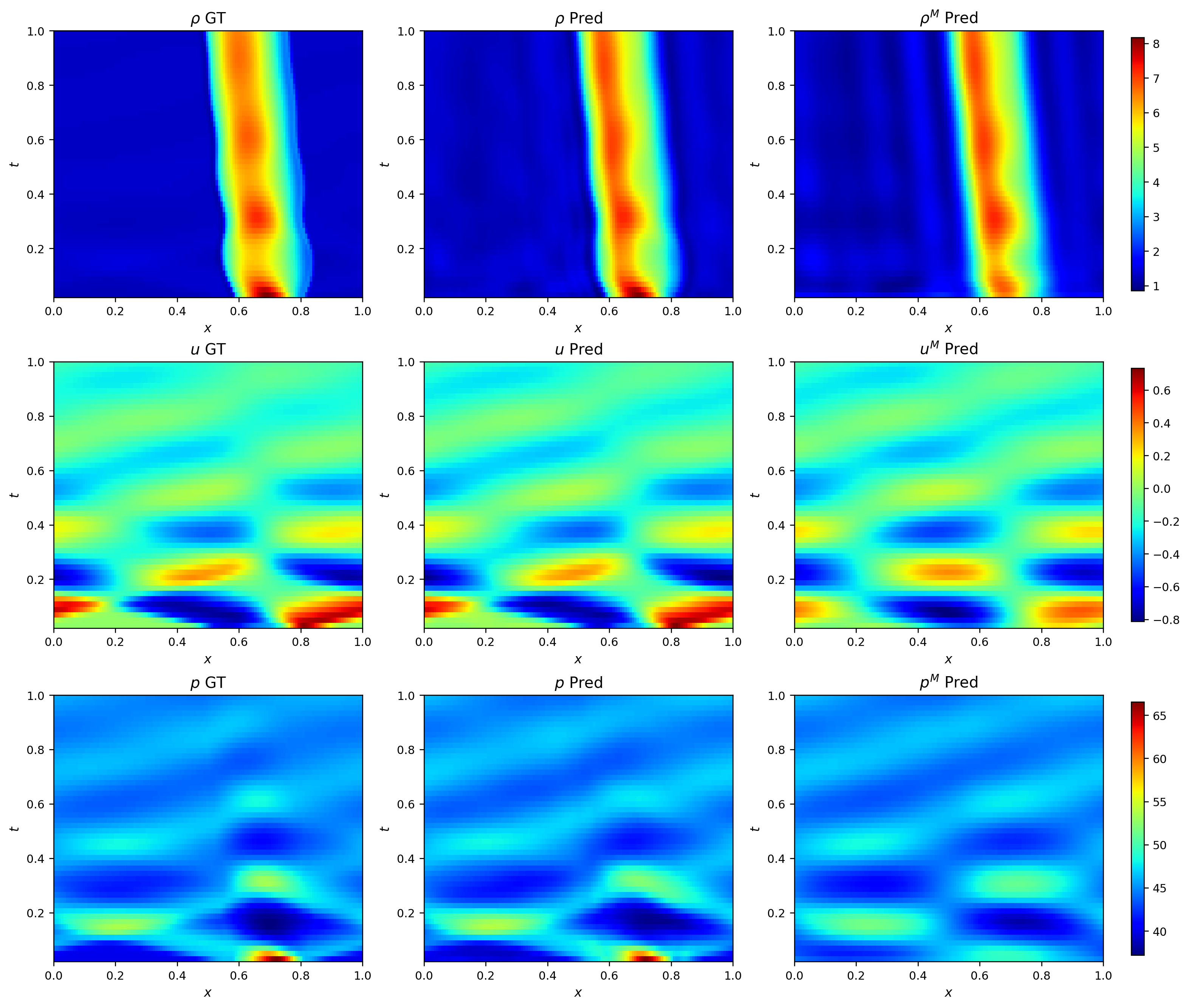}
        \caption{Model prediction ($\alpha=1$).}
        \label{fig:ns_imnose_um}
    \end{subfigure}
    \caption{IMNO-SE ($\alpha=0.1,1$) predictions for the 1D NS equation.}
    \label{fig:ns_imnose_predictions}
\end{figure}

\subsection{Two-dimensional Compressible Navier--Stokes Equation}
We consider the two-dimensional compressible Navier--Stokes equations on the torus
$\Omega = (0,1)\times(0,1)$:
\begin{align}
\partial_t \rho + \nabla \cdot (\rho \mathbf{v}) &= 0, \label{eq:ns2d_mass} \\
\partial_t (\rho \mathbf{v}) + \nabla \cdot (\rho \mathbf{v}\otimes \mathbf{v}) + \nabla p &= \nabla \cdot \boldsymbol{\tau}, \label{eq:ns2d_momentum} \\
\partial_t \!\left( e + \frac{1}{2}\rho |\mathbf{v}|^2 \right)
+ \nabla \cdot \!\left[\left(e+p+\frac{1}{2}\rho |\mathbf{v}|^2\right)\mathbf{v}\right]
&= \nabla \cdot (\boldsymbol{\tau}\mathbf{v}), \label{eq:ns2d_energy}
\end{align}
where $\rho$ is the density, $\mathbf{v}=(v_1,v_2)$ is the velocity, $p$ is the pressure,
and
$e = p/(\Gamma-1)$ is the internal energy density (Here we set $\Gamma=5/3$, corresponding to a monatomic ideal gas).
The viscous stress tensor is given by
\[
\boldsymbol{\tau}
= \eta \bigl(\nabla \mathbf{v} + \nabla \mathbf{v}^{\top}\bigr)
+ \left(\zeta - \frac{2}{3}\eta\right) (\nabla \cdot \mathbf{v}) I,
\]
where $\eta$ and $\zeta$ denote the shear and bulk viscosities, respectively. In this experiment, we set $\eta=\zeta=0.01$.

As in the 1D case, we use the PDEBench dataset \cite{takamoto2022pdebench}, which contains
$N_{\mathrm{data}} = 1000$ samples, each resolved on a spatial grid of
$N_x = N_y =64$ points with spacing $\Delta x=\Delta y=\frac{1}{64}$.
In the temporal direction, the solution is sampled every $\Delta t = 0.05$,
yielding $N_t = 21$ snapshots. The initial velocity field is scaled so that the resulting initial characteristic Mach number is a prescribed value $M=0.1$.
Again, before training, we normalize the data channel-wise via a linear transformation,
so that the data distribution of different physical variables have the same means and variances.

In this experiment, we train and evaluate the models under two horizons ($t_{\mathrm{end}}=0.5,1.0$). The results of FNO, RNO, and IMNO-SE are reported in Table~\ref{tab:ns_results_hyper}.

\begin{table}[H]
\centering
\caption{Relative $L^2$ error of FNO, RNO and IMNO-SE for the 2D NS equation}
\label{tab:ns_results_hyper}
\begin{tabular}{c|cccc}
\toprule
& FNO & RNO  & IMNO-SE($\alpha=0.1$) & IMNO-SE($\alpha=0$) \\

\midrule
Params & 465784 & 697084 & 616294 &  616294\\
\midrule

$t_{\mathrm{end}} = 0.5$ & \(25.64 \pm 1.22\)\% & \(20.71 \pm 0.12\)\%  & \(10.46 \pm 0.67\)\% & \(10.48 \pm 0.39\)\% \\

$t_{\mathrm{end}} = 1.0$ & \(52.46 \pm 0.99\)$\%^*$ & \(24.73 \pm 1.58\)\%  & \(14.73 \pm 1.15\)\% & \(14.68 \pm 1.43\)\% \\

\bottomrule
\end{tabular}
\label{tab:NS2d}
\end{table}

To further compare the behavior of different models, we visualize the predicted
solutions for the density $\rho$ and pressure $p$ at four representative times,
$t=0.05, 0.25, 0.5$, and $1.0$, using models trained with the rollout horizon
$t_{\mathrm{end}}=1.0$. The corresponding results are shown in
Figure~\ref{fig:ns2d_rho_p_comparison}. We can clearly
observe that IMNO-SE produces significantly more accurate predictions than both RNO and
FNO across all displayed times. In particular, when $t_{\mathrm{end}}=1.0$, the FNO baseline exhibits
instability during training, which leads to substantially poorer rollout predictions in
the final model.

\begin{figure}[H]
    \centering
    \begin{subfigure}[b]{0.7\textwidth}
        \centering
        \includegraphics[width=\textwidth]{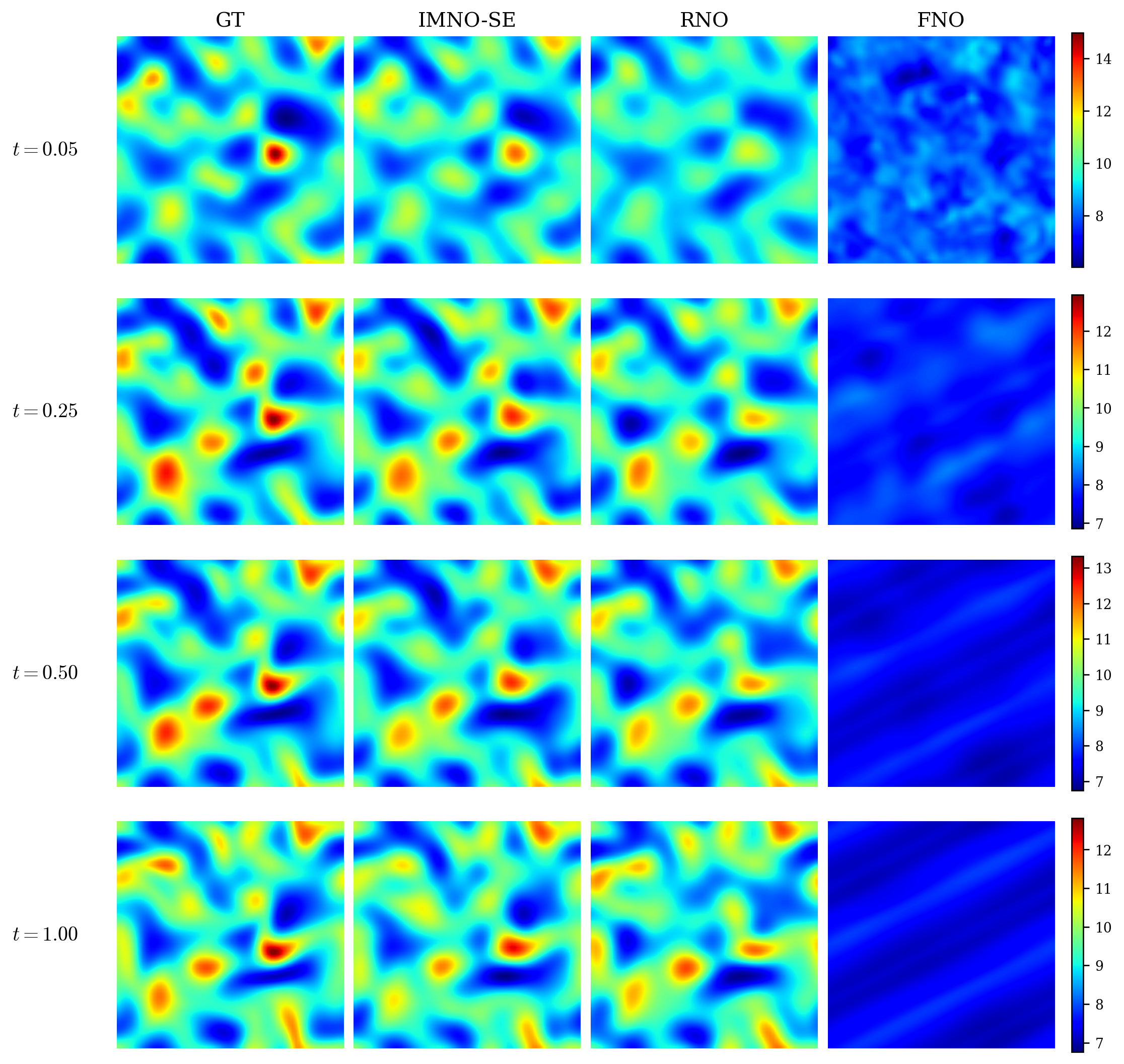}
        \caption{density $\rho$}
    \end{subfigure}
    
    \begin{subfigure}[b]{0.7\textwidth}
        \centering
        \includegraphics[width=\textwidth]{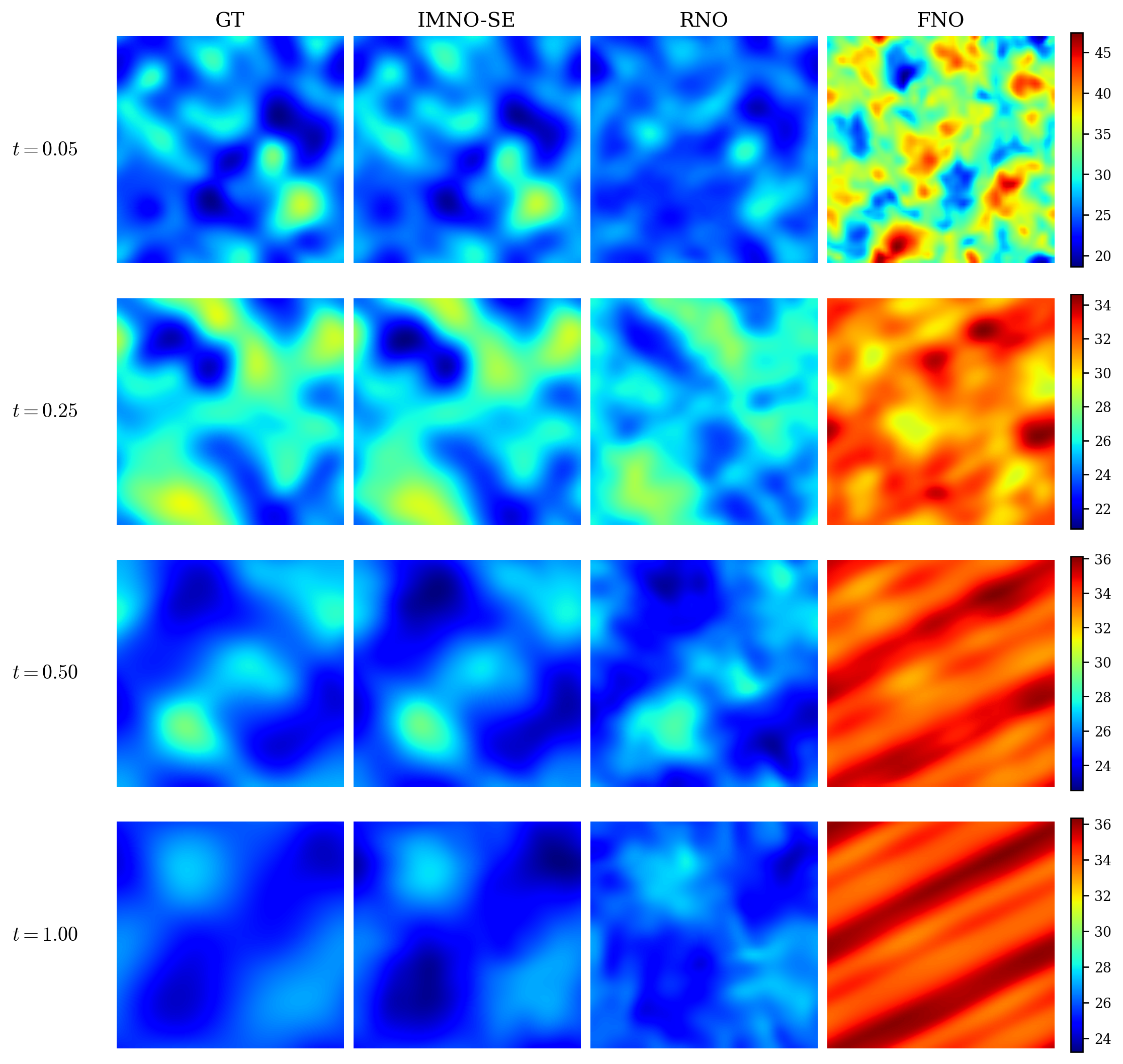}
        \caption{pressure $p$}
    \end{subfigure}
    \caption{Comparison of FNO, RNO and IMNO-SE ($\alpha=0.1$) for the 2D compressible NS equation at $t=0.05, 0.25, 0.50$, and $1.00$.}
    \label{fig:ns2d_rho_p_comparison}
\end{figure}

Although IMNO-SE achieves better predictive accuracy in this example, the learned
manifold component \(u^M\) does not clearly capture the long-time solution structure by
itself. One possible reason is that compressible Navier--Stokes dynamics involve several
strongly coupled physical variables, including velocity, density, pressure, and energy. Density and pressure variations influence the
velocity field, while the velocity field in turn transports and reshapes the density,
pressure, and energy distributions. These degrees of freedom may prevent the dynamics from being
well described by a simple, strongly attracting low-dimensional manifold. In fact, the existence of a finite-dimensional inertial manifold for the Navier--Stokes equations is still an open problem. Consequently, the learned
manifold component \(u^M\) is not expected to provide a clean standalone reduced-order
description of the full solution.

Nevertheless, the prediction of IMNO-SE is still accurate. This suggests that, even when the learned manifold component cannot by itself serve as an accurate reduced-order model of the long-time dynamics, the latent representation can still provide useful global information for the residual neural operator. Therefore, in examples where a clear
low-dimensional structure is not visually apparent, IMNO-SE can still function as an
effective neural operator with good accuracy.

\subsection{Kolmogorov Flow}

We next consider the two-dimensional incompressible Navier--Stokes equation in vorticity on the torus $\Omega = (0,2\pi)\times(0,2\pi)$:
% \begin{align}
% \partial_t w(x,t) + u(x,t)\cdot \nabla w(x,t) &= \nu \Delta w(x,t) + f(x),
% &\qquad x\in(0,2\pi)^2,\ t\in(0,T], \\
% \nabla \cdot u(x,t) &= 0, &\qquad x\in(0,2\pi)^2,\ t\in[0,T], \\
% w(x,0) &= w_0(x), &\qquad x\in(0,2\pi)^2,
% \end{align}
\begin{align}
\partial_t w(x,t) + u(x,t)\cdot \nabla w(x,t) &= \nu \Delta w(x,t) + f(x), \\
\nabla \cdot u(x,t) &= 0, \\
w(x,0) &= w_0(x),
\end{align}
where $u$ is the velocity field, $w=\nabla \times u$ is the scalar vorticity, $\nu>0$ is the viscosity, and $f$ is the forcing term. 
We formulate the learning problem in terms of vorticity rather than velocity. The vorticity formulation often provides a more informative and physically compact description of the flow dynamics. Given $w$, the velocity field $u$ can be recovered through the Biot--Savart law.

In this example, we consider a two-dimensional Kolmogorov flow. 
Specifically, the initial vorticity field $w_0(x)$ is sampled from the Gaussian measure
\[
w_0 \sim \mu, \qquad \mu = \mathcal{N}\bigl(0, (-\Delta +I)^{-2.5}\bigr),
\]
with periodic boundary conditions on the torus. We consider viscosity $\nu=0.1$,
the external forcing is fixed as
\[
f(x) = -0.1\cos x_2.
\]

For the training procedure, we employ a dataset consisting of
$N_{\mathrm{data}} = 1000$ samples, each resolved on a spatial grid of
$N_x = N_y= 64$ points with spacing $\Delta x=\Delta y=\frac{\pi}{32}$.
In the temporal direction, the solution is sampled every $\Delta t = 1.0$,
yielding $N_t = 21$ snapshots.  

% As the viscosity decreases, the dynamics become increasingly complex and less dissipative over long horizons,
In the Kolmogorov flow setting considered here, the equation is shift-equivariant in the $x$-direction but not in the $y$-direction. Accordingly, IMNO-SE is implemented with shift equivariance imposed only along the $x$-direction, and the FNO and RNO baselines are also taken to be their $x$-shift-equivariant versions.
We train and evaluate the models under two rollout horizons ($t_{\mathrm{end}}=10$ and $t_{\mathrm{end}}=20$), and report the
results in Table~\ref{tab:ns_results_kolmogorov}.

\begin{table}[H]
\centering
\caption{Relative $L^2$ error of FNO, RNO and IMNO-SE for the 2D incompressible NS equation}
\label{tab:ns_results_kolmogorov}
\begin{tabular}{c|cccc}
\toprule
& FNO & RNO & IMNO-SE($\alpha=0.1$) & IMNO-SE($\alpha=0$)\\ 
\midrule
Params & 465357 & 696657 & 578146 & 578146\\
\midrule
$t_{\mathrm{end}}=10$ & \(3.43 \pm 0.08\)\% & \(6.13 \pm 0.11\)\% &  \(3.76 \pm 0.12\)\% & \(2.96 \pm 0.02\)\%\\
$t_{\mathrm{end}}=20$ & \(3.10 \pm 0.11\)\% &  \(5.16 \pm 0.06\)\% & \(3.94 \pm 0.60\)\% & \(2.75 \pm 0.11\)\%\\ 
\bottomrule
\end{tabular}
\end{table}

% ToDO: We observe that this benchmark exhibits a clear dependence on both the viscosity and the amount of training data. 
% When sufficient training data are available, the prediction error is substantially reduced; when the data are limited and the viscosity is small, all methods suffer noticeable degradation. 
% In this regime, the main advantage of IMNO is not necessarily that it reconstructs an obvious manifold component, but rather that the latent variable $h$ provides a compact global summary of the large-scale flow structure, which improves the conditioning of the residual operator and can lead to better long-horizon stability.

\begin{figure}[H]
\centering
\begin{subfigure}[b]{0.49\textwidth}
    \centering
    \includegraphics[width=\textwidth]{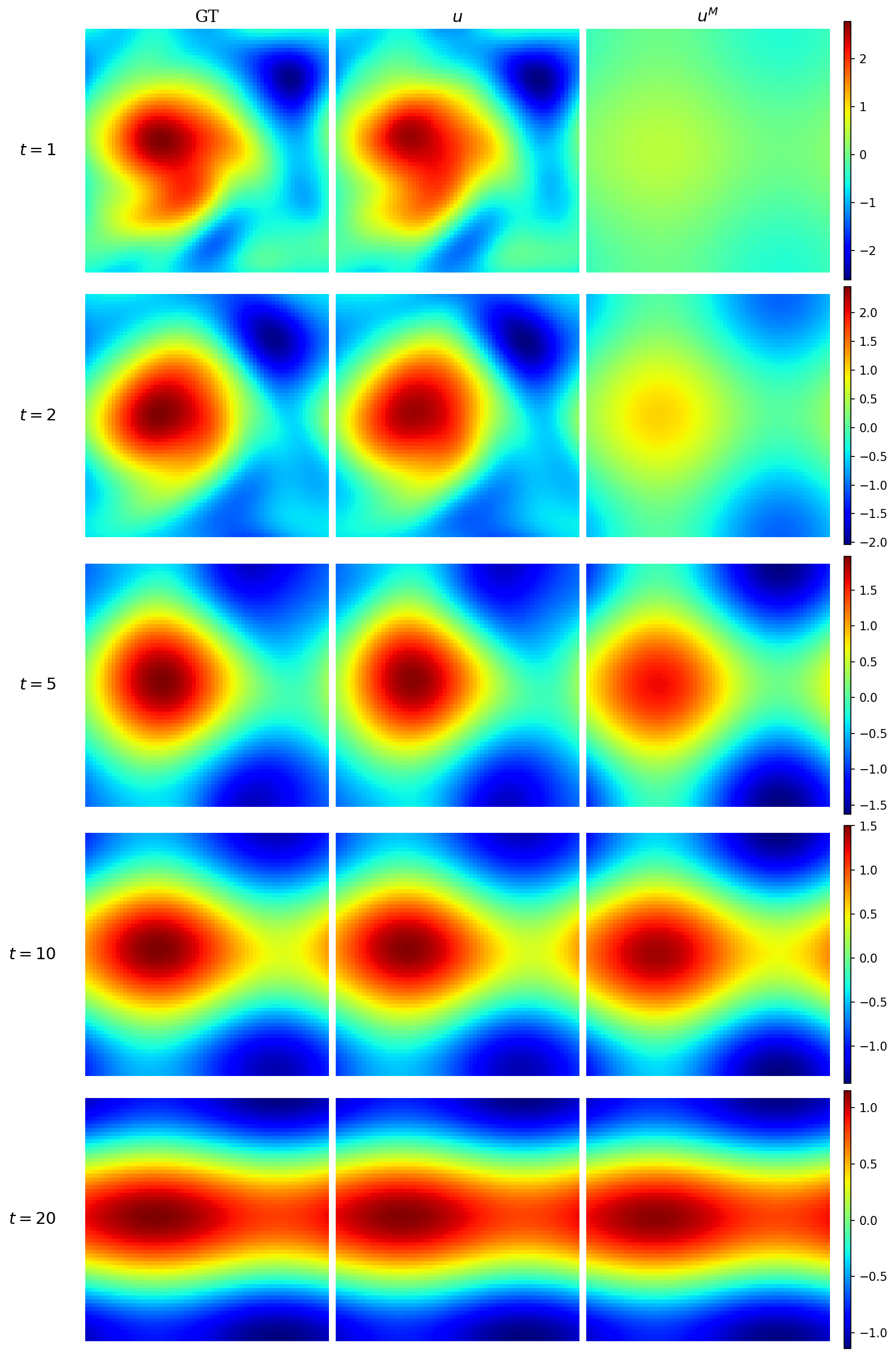}
    \caption{Example 1.}
\end{subfigure}
\hfill
\begin{subfigure}[b]{0.49\textwidth}
    \centering
    \includegraphics[width=\textwidth]{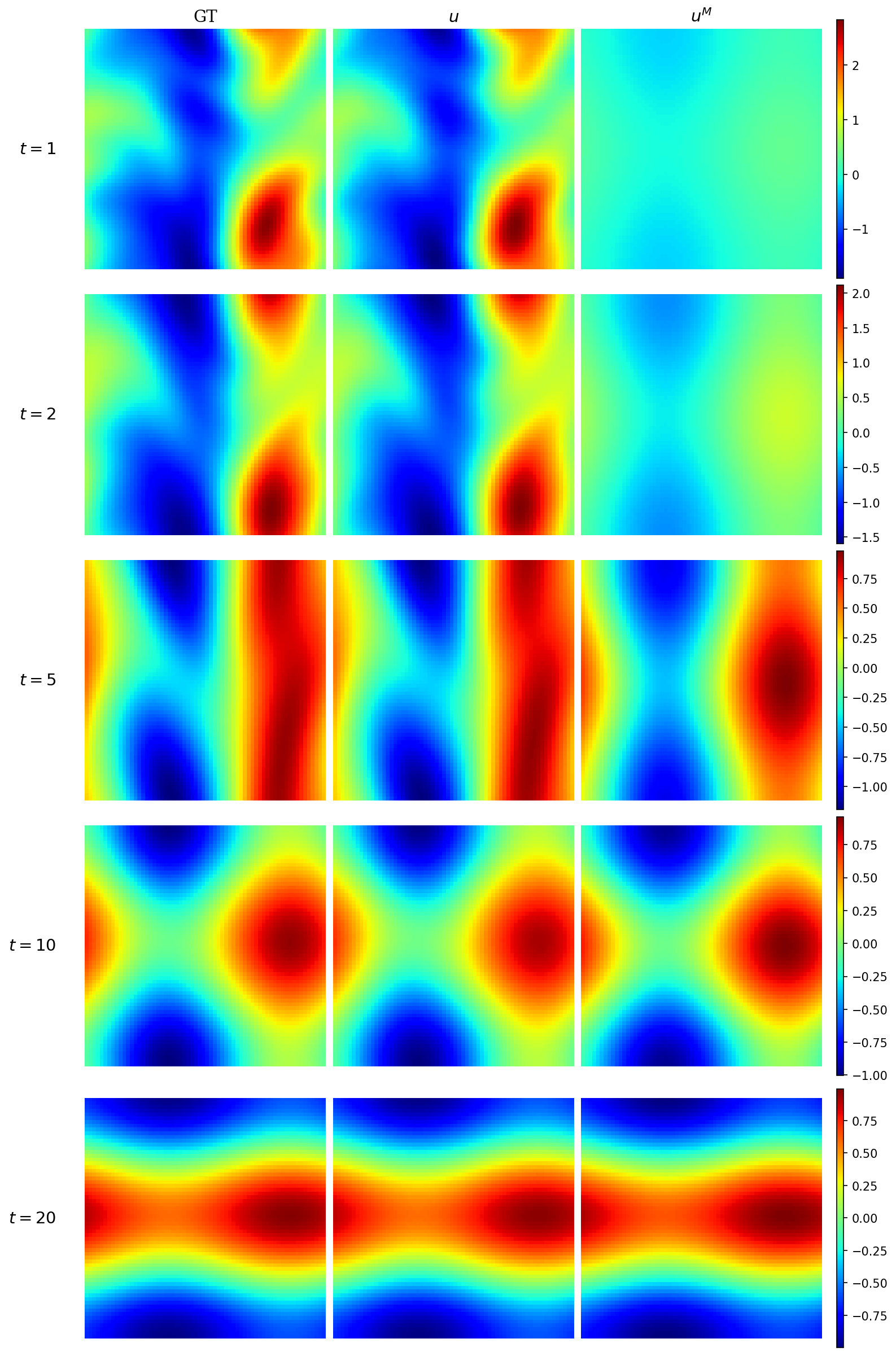}
    \caption{Example 2.}
\end{subfigure}
\caption{IMNO-SE ($\alpha=0.1$) predictions for the incompressible 2D NS equation at $t = 1, 2, 5, 10, 20$.}
\label{fig:2DNS_examples}
\end{figure}

Figure~\ref{fig:2DNS_examples} shows representative IMNO predictions. Consistent with the earlier examples, these results again show that IMNO captures a low-dimensional long-time structure consistent with the inertial-manifold picture. In particular, the manifold component \(u^M\) captures the dominant long-time dynamics, while the full field \(u\) remains accurate throughout the rollout interval. 
Taken together, these results indicate that IMNO remains effective in capturing the long-time dynamics in the two-dimensional setting.

We finally examine a more challenging low-viscosity regime with
\(\nu=10^{-3}\). As in the KS example, this case illustrates a limitation of
IMNO for strongly chaotic dynamics. Figure~\ref{fig:2DNS_chaotic_examples}
shows representative IMNO-SE predictions in this regime. At early times, the
full prediction \(u\) still captures the large-scale patterns of the reference
solution, although some small-scale vortical structures are already smoothed
out. At later times, however, the prediction exhibits a substantial
structural mismatch with the ground truth.

This degradation is consistent with the difficulty of learning reduced dynamics
in chaotic systems. In such regimes, the inertial manifold, if it exists, may have a substantially larger intrinsic dimension than in the previous examples. Therefore, representing the dynamics with a low-dimensional latent coordinate may discard important information and limit the accuracy of the approximation. Moreover, the learned
latent dynamics are highly sensitive to small approximation errors. These errors
can accumulate during autoregressive rollout and cause the decoded manifold
component \(u^M\) to drift away from the reference trajectory. For example, in Figure~\ref{fig:2DNS_chaotic_examples}, the early predictions still
retain the correct large-scale flow pattern, but at later times the manifold
component becomes increasingly dominant in the reconstructed solution while
being misaligned with the true solution. This misalignment then contaminates the residual
correction and leads to the overall distortion of the full prediction. However, although IMNO performs less well in strongly chaotic regimes, where an inertial manifold may be very high-dimensional or may not exist, this behavior is consistent with the theory behind the model: IMNO works best precisely when the dynamics admit a meaningful low-dimensional structure.

\begin{figure}[H]
\centering
\begin{subfigure}[b]{0.49\textwidth}
    \centering
    \includegraphics[width=\textwidth]{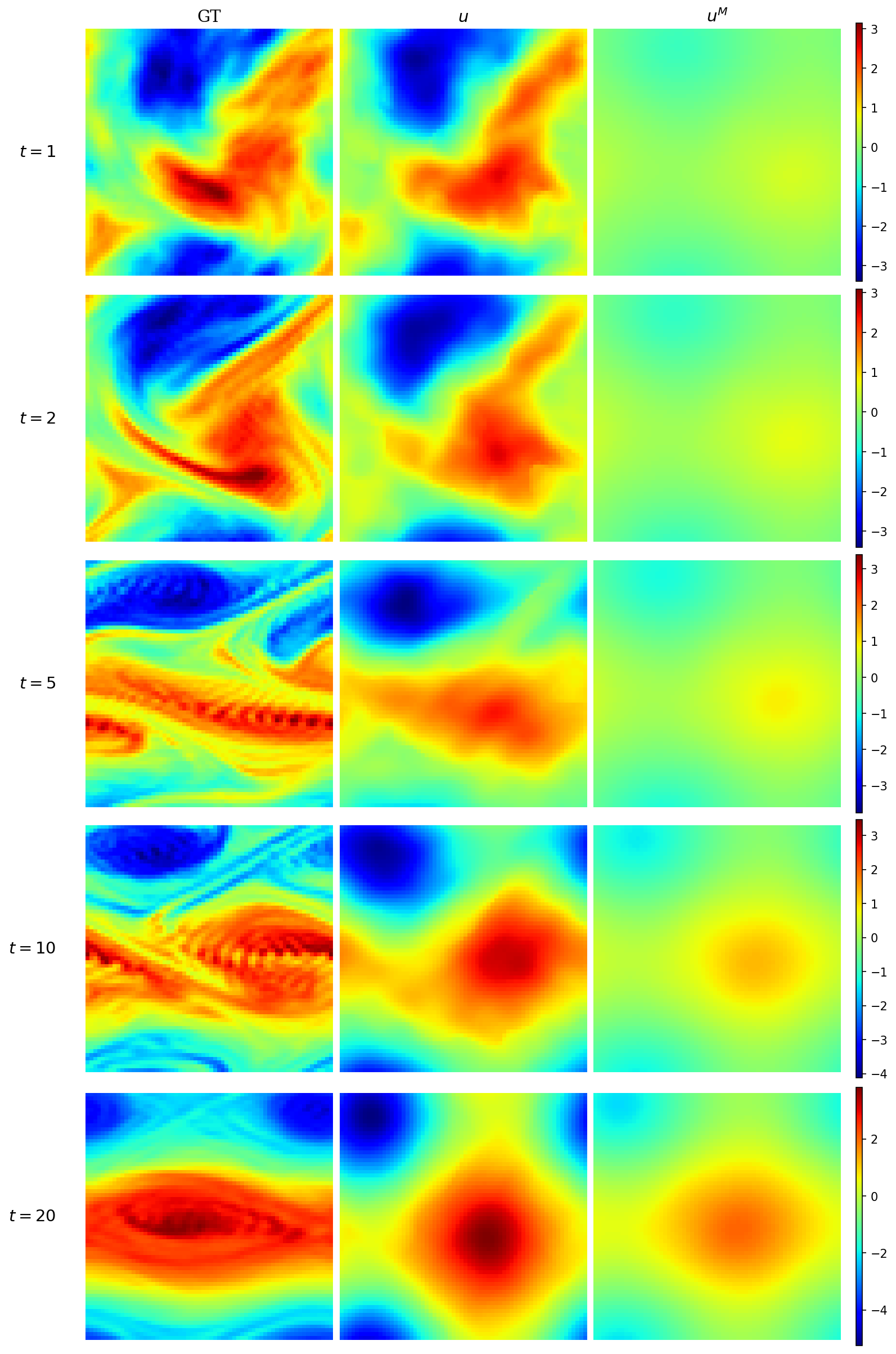}
    \caption{Example 1.}
\end{subfigure}
\hfill
\begin{subfigure}[b]{0.49\textwidth}
    \centering
    \includegraphics[width=\textwidth]{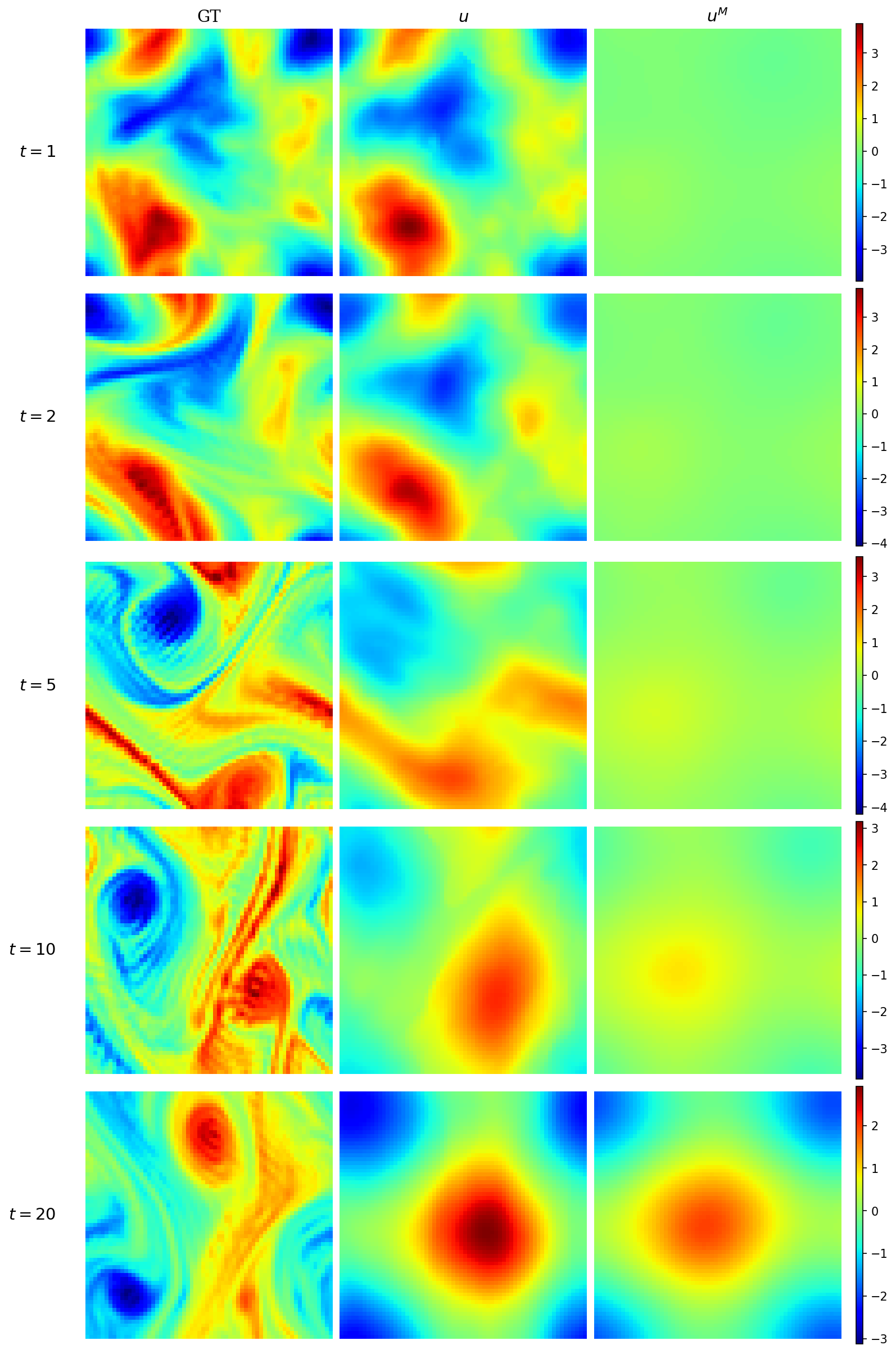}
    \caption{Example 2.}
\end{subfigure}
\caption{IMNO-SE ($\alpha=0.1$) predictions for the incompressible 2D Navier--Stokes equation at $t = 1, 2,5, 10, 20$ in the turbulence regime \(\nu=10^{-3}\).}
\label{fig:2DNS_chaotic_examples}
\end{figure}

\section{Conclusion}
In this paper, we developed the Inertial Manifold Neural Operator (IMNO), a novel neural operator framework for dissipative time-dependent PDEs. Such PDEs' long-time dynamics can often be described by some low-dimensional structure due to dissipation. Inertial manifold theory provides a natural framework for describing such dynamics, under which the solution can be decomposed into a low-dimensional manifold component governing the long-time evolution and a rapidly decaying residual. Guided by this theoretical picture, IMNO is developed to discover and exploit such latent low-dimensional structures in PDE dynamics, thereby learning a more physically interpretable representation while better preserving the intrinsic structure of the underlying PDE.

Various benchmark problems have been used to evaluate the performance of IMNO. The results show that IMNO can effectively capture the low-dimensional long-time dynamics of the underlying PDEs in many cases, while also exhibiting good training stability and accuracy. Moreover, for some more complex PDE systems, although such low-dimensional dynamics may be too difficult to capture accurately, IMNO still remains effective as a neural operator and achieves good accuracy.

An interesting observation from our experiments is that increasing $\alpha$ in the loss function, thus encouraging a larger portion of the solution to be represented by the manifold component, leads to worse overall predictive performance in many cases. One possible reason is that, when the manifold component becomes more dominant, the remaining residual may contain relatively more high-frequency content, which is empirically harder for FNO-type architectures to approximate accurately. Another possible reason is that the residual dynamics may become more tightly coupled with the latent manifold dynamics when $\alpha$ increases, while the current neural operator structure may not be expressive enough to capture this coupling efficiently. These observations suggest that future work should develop a deeper understanding of how the manifold and residual components are jointly trained and how they interact, thereby enabling the development of better architectures and loss functions that better exploit the manifold structure while maintaining accurate full-state prediction.

Beyond this issue, several other directions are also worth further investigation. For example, the performance of IMNO remains limited for systems with chaotic dynamics. In such cases, the manifold component $u^M$ may still capture some meaningful low-dimensional patterns, but errors grow rapidly due to the system's chaotic behavior. Therefore, the manifold prediction may drift during long-time rollout, which in turn affects the residual prediction and degrades the overall performance. A possible remedy is to periodically reproject the current full predicted solution onto the learned manifold, thereby updating the latent coordinate and mitigating its drift during long-time rollout.

IMNO could also be further applied to larger-scale simulations, where computational efficiency becomes a more central concern. In this setting, the learned low-dimensional manifold may serve as a cheaper surrogate for long-time evolution. For example, once the residual part has decayed to a sufficiently small level, one may continue the rollout mainly through the manifold dynamics rather than evaluating the full residual neural operator at every step. This could potentially lead to a significant speedup for large-scale autoregressive prediction.

In a broader context, we believe that this work, together with the questions it raises, highlights the potential connection between operator learning and reduced-order modeling. By bridging these two fields, it may help us better understand their respective advantages and limitations, with the potential to enrich both fields and to inspire further research in scientific machine learning.
% To address central challenge is that the residual dynamics are not autonomous, but depend on the latent dynamics; this is addressed by the proposed architecture. Under suitable assumptions, we further establish a universal approximation property for IMNO.

\section*{Data availability}
Code and data for reproducing the results in this paper will be available at \url{https://github.com/Xiaoyang-Xie/IMNO}. All numerical experiments were conducted using NVIDIA A100 and L40 GPUs.

\bibliographystyle{abbrv}
\bibliography{references}

\newpage

\section*{Appendix}\label{sec:appendix_proof_conditioned_operator}

\subsection*{A. Proof of the Universal Approximation Theorem}
This appendix proves Theorem~\ref{thm:imno-uat}. Section~A.1 fixes notation and collects the preliminary properties used throughout the proof; Section~A.2 establishes six preliminary lemmas;
Section~A.3 assembles the proof.
 
\subsubsection*{A.1 Basic settings and properties}

Throughout Appendix~A, we work under the assumptions of
Theorem~\ref{thm:imno-uat}. We first fix the notation and collect several
basic properties that will be used throughout the proof.

Recall that
\[
H=H^s(\mathbb T^d;\mathbb R^c)
\]
for some $s\ge 0$. For each $N\in\mathbb N$, let $P_N:H\to H$ denote the
orthogonal projection onto the Fourier modes satisfying $|k|_\infty\le N$,
and set
\[
Q_N:=I-P_N.
\]
We also denote by $\mathcal F_N$ the associated finite-dimensional Fourier
feature map
$$
\mathcal F_N:H\to\mathbb R^{D_N},
$$
which collects the real and imaginary parts of the Fourier coefficients of $u$ with $|k|_\infty\le N$. Here, $D_N = c(2N+1)^d$ denotes the total dimension of this truncated feature space. Conversely, let
$$
\mathcal S_N:\mathbb R^{D_N}\to P_NH
$$
denote the associated band-limited Fourier synthesis map. These maps are
chosen so that
$$
\mathcal S_N(\mathcal F_N(u))=P_Nu,
\qquad u\in H.
$$

The projections satisfy $\|P_N\|\le 1$ for all $N$, and
$P_Nu\to u$ in $H$ as $N\to\infty$ for every $u\in H$. Moreover, this convergence is uniform on compact subsets of
$H$. Indeed, if $C\subset H$ is compact and $\eta>0$, choose a finite
$\eta/3$-net $\{u_j\}_{j=1}^J$ of $C$. Since $P_Nu_j\to u_j$ for each
fixed $j$, for all sufficiently large $N$ we have
$\max_{1\le j\le J} \norm{P_Nu_j-u_j}_H<\eta/3$. For any $u\in C$, choosing
$j$ such that $\norm{u-u_j}_H<\eta/3$ gives
$$
\norm{P_Nu-u}_H
\le
\norm{P_N(u-u_j)}_H
+
\norm{P_Nu_j-u_j}_H
+
\norm{u_j-u}_H
<\eta ,
$$
using $\|P_N\|\le 1$. Hence
$$
\sup_{u\in C}\norm{P_Nu-u}_H\to 0 .
$$

Let $\mathcal{G}=S(\Delta t):H\to H$ be the one-step solution operator.
By assumption, $\mathcal{G}$ is continuous on $H$. Let
$\mathcal M\subset H$ be the finite-dimensional flow-normally hyperbolic
inertial manifold in Theorem~\ref{thm:imno-uat}, and let
$\Pi:H\to\mathcal M$ be the associated continuous asymptotic phase map.

We next state two consequences of flow-normal hyperbolicity. First, the
asymptotic phase is unique, which implies
$$
\Pi\circ\mathcal G
=
\mathcal G|_{\mathcal M}\circ\Pi.
$$
Indeed, according to the definition of $\Pi$, $S(t)\Pi(u)$ exponentially tracks $S(t)u$, then
$$
S(t)\mathcal G\Pi(u)
=
S(t+\Delta t)\Pi(u)
$$
exponentially tracks
$$
S(t)\mathcal G u
=
S(t+\Delta t)u.
$$
Thus $\mathcal G\Pi(u)$ is an asymptotic phase of $\mathcal G u$, and gives
$$
\Pi(\mathcal G u)=\mathcal G\Pi(u).
$$

Second, the restricted one-step map $\mathcal G|_{\mathcal M}$ is injective.
Indeed, suppose that $u_1^M,u_2^M\in\mathcal M$ satisfy
$\mathcal G u_1^M=\mathcal G u_2^M$. Applying the backward separation
estimate over a time interval of length $\Delta t$ gives
$$
\|u_1^M-u_2^M\|_H
\le
C e^{\gamma\Delta t}
\|\mathcal G u_1^M-\mathcal G u_2^M\|_H
=
0.
$$
Hence $u_1^M=u_2^M$, and therefore
$\mathcal G|_{\mathcal M}$ is injective.

For the compact set $K\subset H$ appearing in Theorem~\ref{thm:imno-uat}, define
$$
\mathcal M_0:=\Pi(K),
\qquad
\mathcal M_1:=\mathcal M_0\cup \mathcal{G}(\mathcal M_0).
$$
Since $K$ is compact and $\Pi$ is continuous, $\mathcal M_0$ is compact.
Since $\mathcal{G}$ is continuous, $\mathcal{G}(\mathcal M_0)$ is also
compact. Moreover, by the positive invariance of the inertial manifold,
$\mathcal{G}(\mathcal M_0)\subset\mathcal M$. Therefore
$\mathcal M_1\subset\mathcal M$ is compact. These compact sets collect the asymptotic phases of the inputs in $K$ and their images after one time step on the manifold.

Since $\mathcal M_1$ is a compact subset of the $m$-dimensional inertial manifold
$\mathcal M$, its topological dimension is at most $m$. By the
Menger--Nöbeling embedding theorem, $\mathcal M_1$ admits a topological embedding
into $\mathbb R^{2m+1}$. Therefore, if the latent dimension satisfies
$d_h\ge 2m+1$, we can fix a continuous injective map
\[
    E:\mathcal M_1\to \mathbb R^{d_h}.
\]
Since $\mathcal M_1$ is compact and $\mathbb R^{d_h}$ is Hausdorff, $E$ is
a homeomorphism from $\mathcal M_1$ onto its image. Define
\[
    \Lambda_0:=E(M_0),
    \qquad
    \Lambda_1:=E(\mathcal{G}(M_0)).
\]
Then
\[
    E(M_1)=\Lambda_0\cup\Lambda_1.
\]
We write
\[
    D:=E^{-1}:
    \Lambda_0\cup\Lambda_1\to M_1
\]
for the inverse map, which is continuous.

For $u\in K$, define
$$
h(u):=E(\Pi(u)),
\qquad
r(u):=u-\Pi(u).
$$
Since $\Pi$, $E$, and the identity map on $H$ are continuous, both
$h:K\to\mathbb R^{d_h}$ and $r:K\to H$ are continuous. Therefore
$$
C_K
:=
{(r(u),h(u)):u\in K}
\subset H\times\mathbb R^{d_h}
$$
is compact.

We next define the exact one-step map in latent coordinates:
\[
    F:\Lambda_0\to \Lambda_1,
    \qquad
    F(h)
    :=
    E\bigl(\mathcal{G}(D(h))\bigr).
\]
Indeed, if $h\in\Lambda_0$, then $D(h)\in M_0$, and hence
$\mathcal{G}(D(h))\in \mathcal{G}(M_0)$. Therefore
$F(h)\in\Lambda_1$.

The map $F$ is continuous. Moreover, since $D$ and
$E$ are injective and $\mathcal{G}|_{\mathcal M}$ is injective,
$F$ is injective on $\Lambda_0$. Hence $F$ is a continuous
bijection from the compact set $\Lambda_0$ onto the compact set
$\Lambda_1$. Since $\Lambda_1\subset\mathbb R^{d_h}$ is Hausdorff,
$F$ is a homeomorphism from $\Lambda_0$ to $\Lambda_1$.

Next, define the exact one-step residual map
$\mathcal R_0:H\times\Lambda_0\to H$ by
$$
\mathcal R_0(u^R,h)
:=
\mathcal{G}(D(h)+u^R)
-
\Pi\bigl(\mathcal{G}(D(h)+u^R)\bigr).
$$
Since $D$, $\mathcal G$, and
$\Pi$ are continuous, and addition is continuous on $H$,
the map $\mathcal R_0$ is continuous.

For every $u\in K$, we have
$$
D(h(u))
=
D(E(\Pi(u)))
=
\Pi(u),
$$
and hence
$$
\mathcal R_0(r(u),h(u))
=
\mathcal{G}(u)-\Pi(\mathcal{G}(u)).
$$
Thus $\mathcal R_0$ gives the exact next-step residual when evaluated at the current residual and the current latent coordinate.

We now specify the IMNO architecture class used in the approximation theorem. In the proof below, when we refer to an IMNO model, we mean a model consisting of the following components:
\begin{itemize}
\item \textbf{Encoder.}
The encoder has the form
$$
E_\theta(u)=\varphi_\theta(\mathcal F_{k_E}(u)),
$$
where $\mathcal F_{k_E}(u)\in\mathbb R^{D_{k_E}}$ collects the real and
imaginary parts of the Fourier coefficients of $u$ with
$|k|_\infty\le k_E$, and $\varphi_\theta$ is a MLP.
The truncation level $k_E$ is a free hyperparameter.

\item \textbf{Decoder.}
The decoder has the form
$$
    D_\theta(h)=\mathcal S_{k_D}(\psi_\theta(h)),
$$
where $\psi_\theta$ is an MLP producing the real and imaginary parts of
the Fourier coefficients for modes $|k|_\infty\le k_D$, and
$\mathcal S_{k_D}$ denotes the corresponding band-limited Fourier
synthesis map. The truncation level $k_D$ is also free.

\item \textbf{Latent dynamics.}
The latent dynamics are represented by an explicit one-step update
$$
    h_{t+1}=h_t+\Delta t\, f_\theta(h_t),
$$
where $f_\theta:\mathbb R^{d_h}\to\mathbb R^{d_h}$ is an MLP.

\item \textbf{Residual operator.}
The residual dynamics are represented by a conditioned neural operator
$$
    \mathcal{R}_\theta(\cdot,h)
    =
    \mathcal Q\circ\mathcal L_L^{(h)}\circ\cdots\circ\mathcal L_1^{(h)}\circ \mathcal P ,
$$
where each hidden layer has the form
$$
   \mathcal L_\ell^{(h)}(v)(x)
    =
    \sigma\bigl(
        W_\ell v(x)
        +
        \beta_\ell(h)
        +
        (\mathcal K_\ell v)(x)
    \bigr).
$$
Here $\mathcal P$ and $\mathcal Q$ are pointwise lifting and projection maps, $W_\ell$ is a
pointwise linear transformation, $\mathcal K_\ell$ is a Fourier integral
operator, and $\beta_\ell$ is a learnable affine map
from the latent coordinate to the channel space. The width, depth, and
number of retained Fourier modes are free parameters.

\end{itemize}

Throughout the proof, the activation function $\sigma$ is chosen so that
the standard finite-dimensional universal approximation theorem for MLPs
applies.
% Whenever identity channels are needed in the construction,
% we use the fact that the identity function can be represented exactly or
% approximated uniformly on compact intervals by such MLPs.

\begin{remark}[Encoder used in the experiments]
The theoretical encoder class in the proof applies a nonlinear MLP after Fourier truncation. This generality is needed for the universal approximation theorem, since the target phase-coordinate map
$u\mapsto E(\Pi(u))$ is a continuous map on $K$ and need not be linear.

In the implementation, we use a lighter encoder. The input is first lifted
pointwise to a higher-dimensional channel space, low Fourier modes are then
extracted, and a final linear map produces the latent coordinate. The resulting
encoder is therefore linear and is less expressive than the nonlinear MLP class
used in the theorem. Nevertheless, in our experiments, this simpler variant
often achieves accuracy comparable to, or even better than, variants with an
additional nonlinear layer, while also improving training stability. It may be
viewed as learning a data-adapted linear approximation to the target
phase-coordinate map over the region sampled by the training data.
% This can be partially
% explained by the structure of graph-based inertial-manifold constructions. In
% the standard spectral-gap setting, the inertial manifold is often obtained as
% a graph over finitely many low Fourier modes,
% $$
% \mathcal M={p+\phi(p):p\in P_nH}.
% $$
% In this case, the manifold coordinate map may be chosen as
% $$
% E=P_n.
% $$
% Thus the exact encoder target is
% $$
% h(u)=E(\Pi(u))=P_n\Pi(u).
% $$
% Although this phase-coordinate map is generally nonlinear, the retraction
% property gives $\Pi(u)=u$ on the inertial manifold. Therefore, when $\Pi$ is sufficiently regular near the manifold, it is close to the identity
% in the neighborhood of the manifold, and the target coordinate map
% $h(u)=E(\Pi(u))$ behaves approximately linearly in it. 
However, this compact structure
may still be limited in some regimes, and incorporating appropriate
nonlinearity into the encoder is an interesting direction for future work.

\end{remark}
 
\subsubsection*{A.2 Preliminary lemmas}

We first recall the universal approximation theorem of Fourier neural
operators, which will be used later. The proof
can be found in~\cite{kovachki2021universal}.

\begin{lemmaA}[Universal approximation by FNOs]\label{lem}
Let $s,s'\ge 0$. Let
$$
\mathcal G:H^s(\mathbb T^d;\mathbb R^{d_a})
\to
H^{s'}(\mathbb T^d;\mathbb R^{d_u})
$$
be a continuous operator. Let
$K\subset H^s(\mathbb T^d;\mathbb R^{d_a})$ be a compact subset.
Then, for any $\varepsilon>0$, there exists a Fourier neural operator
$$
\mathcal N:H^s(\mathbb T^d;\mathbb R^{d_a})
\to
H^{s'}(\mathbb T^d;\mathbb R^{d_u}),
$$
of the standard FNO form, continuous as an operator from
$H^s(\mathbb T^d;\mathbb R^{d_a})$ to
$H^{s'}(\mathbb T^d;\mathbb R^{d_u})$, such that
$$
\sup_{a\in K}
\norm{\mathcal G(a)-\mathcal N(a)}_{H^{s'}}
\le
\varepsilon .
$$
\end{lemmaA}

In the main proof, we will use two elementary technical tools repeatedly. The first is a continuous extension result.

\begin{lemmaA}[Continuous extension]\label{lem:extension}
Let $X$ be a metric space, let $A\subset X$ be closed, and let $Y$ be a
Banach space. If $f:A\to Y$ is continuous, then there exists a continuous
map $\bar f:X\to Y$ such that
$$
\bar f|_A=f .
$$
\end{lemmaA}

This is the Dugundji extension theorem \cite{dugundji1951extension}. In what follows, we use Lemma~\ref{lem:extension} to extend
continuous maps originally defined only on compact subsets, such as $K$,
$\Lambda$, or $C_K$, to continuous maps on the ambient spaces.

The second tool is a relative form of uniform continuity. In an infinite-dimensional
space, a closed neighborhood of a compact set need not be compact. The
following lemma allows us to control a continuous map near a compact set
without requiring compactness of such neighborhoods.

\begin{lemmaA}[Uniform continuity relative to a compact set]\label{lem:ruc}
Let $(X,d_X)$ and $(Y,d_Y)$ be metric spaces, let $T:X\to Y$ be continuous,
and let $C\subset X$ be compact. Then, for every $\eta>0$, there exists
$\delta>0$ such that, for all $x\in C$ and all $y\in X$ satisfying
$d_X(x,y)<\delta$, one has
$$
d_Y(Tx,Ty)<\eta .
$$
\end{lemmaA}

\begin{proof}
Suppose the claim fails. Then there exist $\eta>0$, sequences
$x_n\in C$ and $y_n\in X$, such that
$$
d_X(x_n,y_n)<\frac1n,
\qquad
d_Y(Tx_n,Ty_n)\ge \eta .
$$
Since $C$ is compact, after passing to a subsequence we may assume that
$x_n\to x$ for some $x\in C$. Since $d_X(x_n,y_n)\to0$, we also have
$y_n\to x$. By continuity of $T$ at $x$, both $Tx_n$ and $Ty_n$ converge to
$Tx$, and hence $d_Y(Tx_n,Ty_n)\to0$, a contradiction.
\end{proof}

We next prove that the encoder class used in the theoretical IMNO architecture
can approximate any continuous target map on the compact
input set $K$. The key point is that Fourier truncation converges uniformly on
compact subsets of $H$, so the infinite-dimensional approximation problem can
be reduced to a finite-dimensional MLP approximation problem.

\begin{lemmaA}[Encoder approximation]\label{lem:enc}
Let $g:K\to\mathbb R^{d_h}$ be continuous on the compact set
$K\subset H$. Then, for every $\eta>0$, there exist $k_E\in\mathbb N$
and an MLP $\varphi:\mathbb R^{D_{k_E}}\to\mathbb R^{d_h}$ such that
$$
\sup_{u\in K}
\norm{
\varphi(\mathcal F_{k_E}(u))-g(u)
}
<\eta .
$$
\end{lemmaA}

\begin{proof}
By Lemma~\ref{lem:extension}, extend $g$ to a continuous map
$$
\bar g:H\to\mathbb R^{d_h}.
$$

We first choose a Fourier truncation level. Since $P_k u\to u$ in $H$ for
each $u\in H$, and this convergence is uniform on compact subsets of $H$, we
have
$$
\sup_{u\in K}\norm{P_ku-u}_H\to0
\qquad \text{as } k\to\infty .
$$
Applying Lemma~\ref{lem:ruc} to the continuous map $\bar g$ and the compact
set $K$, we can choose $k_E$ sufficiently large so that
$$
\sup_{u\in K}
\norm{
\bar g(P_{k_E}u)-\bar g(u)
}
<\frac{\eta}{2}.
$$

It remains to approximate the finite-dimensional map induced by this
truncation. Define
$$
A:=\mathcal F_{k_E}(K)\subset\mathbb R^{D_{k_E}} .
$$
Since $K$ is compact and $\mathcal F_{k_E}$ is continuous, $A$ is compact.
Now define
$$
g_{k_E}:\mathbb R^{D_{k_E}}\to\mathbb R^{d_h},
\qquad
g_{k_E}(z):=\bar g(\mathcal S_{k_E}(z)).
$$
This map is continuous on the compact set $A$. By the finite-dimensional
universal approximation theorem for MLPs, there exists an MLP
$\varphi:\mathbb R^{D_{k_E}}\to\mathbb R^{d_h}$ such that
$$
\sup_{z\in A}
\norm{
\varphi(z)-g_{k_E}(z)
}
<\frac{\eta}{2}.
$$

For every $u\in K$, we have
$$
g_{k_E}(\mathcal F_{k_E}(u))
=
\bar g(\mathcal S_{k_E}(\mathcal F_{k_E}(u)))
=
\bar g(P_{k_E}u).
$$
Therefore,
$$
\begin{aligned}
\norm{
\varphi(\mathcal F_{k_E}(u))-g(u)
}
&\le
\norm{
\varphi(\mathcal F_{k_E}(u))
-
\bar g(P_{k_E}u)
}
+
\norm{
\bar g(P_{k_E}u)-\bar g(u)
} 
<
\eta .
\end{aligned}
$$
Taking the supremum over $u\in K$ proves the claim.
\end{proof}

The next lemma establishes a decoder approximation result. It shows that a continuous map from
the finite-dimensional latent space into the function space $H$ can be
approximated by first using an MLP to generate finitely many Fourier
coefficients and then applying band-limited Fourier synthesis.

\begin{lemmaA}[Decoder approximation]\label{lem:dec}
Let $\Lambda\subset\mathbb R^{d_h}$ be compact, and let
$D:\Lambda\to H$ be continuous. Then, for every $\eta>0$, there exist
$k_D\in\mathbb N$ and an MLP
$\psi:\mathbb R^{d_h}\to\mathbb R^{D_{k_D}}$ such that
$D_\theta=S_{k_D}\circ\psi$ satisfies
$$
\sup_{h\in\Lambda}
\|D_\theta(h)-D(h)\|_H<\eta.
$$
\end{lemmaA}

\begin{proof}
Since $D$ is continuous and $\Lambda$ is compact, the image
$D(\Lambda)$ is compact in $H$. By the uniform convergence of Fourier
truncations on compact subsets of $H$, we have
$$
\sup_{h\in\Lambda}
\norm{P_N D(h)-D(h)}_H
=
\sup_{h\in\Lambda}
\norm{Q_ND(h)}_H
\to 0
\qquad \text{as } N\to\infty .
$$
Choose $k_D$ sufficiently large so that
$$
\sup_{h\in\Lambda}
\norm{Q_{k_D} D(h)}_H
<\frac{\eta}{2}.
$$

We now approximate the retained Fourier coefficients. Define the coefficient
map
$$
a_{k_D}:\Lambda\to\mathbb R^{D_{k_D}},
\qquad
a_{k_D}(h):=\mathcal F_{k_D}(D(h)).
$$
This map is continuous because $D$ is continuous and
$\mathcal F_{k_D}$ is a continuous finite-dimensional Fourier feature map.
Since $\Lambda$ is compact, the finite-dimensional universal approximation
theorem for MLPs gives an MLP
$\psi:\mathbb R^{d_h}\to\mathbb R^{D_{k_D}}$ such that
$$
\sup_{h\in\Lambda}
\norm{\psi(h)-a_{k_D}(h)}
<
\frac{\eta}{2C_{k_D}},
$$
where
$$
C_{k_D}:=\norm{\mathcal S_{k_D}}_{\mathcal L(\mathbb R^{D_{k_D}},H)}<\infty .
$$
This norm is finite because $\mathcal S_{k_D}$ is a linear map defined on the
finite-dimensional normed space $\mathbb R^{D_{k_D}}$; every such linear map
into a normed space is bounded.

For any $h\in\Lambda$, using
$\mathcal S_{k_D}(\mathcal F_{k_D}(D(h)))=P_{k_D} D(h)$, we obtain
$$
\begin{aligned}
\norm{D_\theta(h)- D(h)}_H
&=
\norm{\mathcal S_{k_D}(\psi(h))- D(h)}_H  \\
&\le
\norm{\mathcal S_{k_D}(\psi(h)-a_{k_D}(h))}_H
+
\norm{P_{k_D} D(h)- D(h)}_H  \\
&\le
C_{k_D}\norm{\psi(h)-a_{k_D}(h)}
+
\norm{Q_{k_D} D(h)}_H  \\
&<
\eta .
\end{aligned}
$$
Taking the supremum over $h\in\Lambda$ proves the claim.
\end{proof}

The next lemma establishes the universality property of the conditioned residual operator.

\begin{lemmaA}[Residual operator universality]\label{lem:cond_fno}
Let $S\subset H\times \mathbb R^{d_h}$ be compact, and let
$\mathcal{R}:H\times \mathbb R^{d_h}\to H$ be continuous. Then, for every $\eta>0$, there
exists a conditioned residual operator $\mathcal{R}_\theta$ such that
\[
    \sup_{(v,h)\in S}
    \|\mathcal{R}_\theta(v,h)-\mathcal{R}(v,h)\|_H < \eta .
\]
\end{lemmaA}

\begin{proof}
We split the proof into two steps. First, we encode the conditioning variable $h$ as a constant function and apply the standard FNO universal approximation theorem to obtain an FNO that approximates $\mathcal{R}$ on a compact set. Second, we show that one can directly construct a residual operator $\mathcal{R}_\theta$ that exactly reproduces  this FNO on $S$.

\medskip
\noindent
\textbf{Step 1: Approximating $\mathcal{R}$ by a standard FNO on augmented inputs.}

For each $h\in \mathbb R^{d_h}$, let $\mathbf 1\otimes h$ denote the
constant $\mathbb R^{d_h}$-valued function on $\mathbb T^d$ with value
$h$. Define the augmentation map
\[
    \mathcal I:H\times \mathbb R^{d_h}
    \longrightarrow
    H^s(\mathbb T^d;\mathbb R^{c+d_h}),
    \qquad
    \mathcal I(v,h):=(v,\mathbf 1\otimes h).
\]
Since $\mathcal I$ is continuous, the set
\[
    \widetilde S:=\mathcal I(S)
    =
    \{(v,\mathbf 1\otimes h):(v,h)\in S\}
\]
is compact. Moreover, $\mathcal I$ is injective on $S$. Since $S$ is compact
and $H^s(\mathbb T^d;\mathbb R^{c+d_h})$ is Hausdorff,
$\mathcal I$ is a homeomorphism from $S$ onto $\widetilde S$.

Define
\[
    \widetilde{\mathcal{R}}:\widetilde S\to H,
    \qquad
    \widetilde{\mathcal{R}}(v,\mathbf 1\otimes h):=\mathcal{R}(v,h).
\]
This map is well-defined and continuous because $\mathcal I^{-1}$ is
continuous on $\widetilde S$ and $\mathcal{R}$ is continuous on
$H\times\mathbb R^{d_h}$.

By Lemma~\ref{lem:extension}, $\widetilde{\mathcal{R}}$ admits a continuous
extension, denoted by $\bar{\mathcal{R}}$, from
$H^s(\mathbb T^d;\mathbb R^{c+d_h})$ to $H$. Applying the standard FNO
universal approximation theorem to this extended map on the compact set
$\widetilde S$, we obtain a standard FNO $\mathcal N$ such that
\[
    \sup_{(v,\mathbf 1\otimes h)\in \widetilde S}
    \bigl\|
        \mathcal N(v,\mathbf 1\otimes h)-\bar{\mathcal{R}}(v,\mathbf 1\otimes h)
    \bigr\|_H = \sup_{(v,h)\in S}
    \bigl\|
        \mathcal N(v,\mathbf 1\otimes h)-{\mathcal{R}}(v,h)
    \bigr\|_H
    <
    \eta.
\]
It remains to show that the conditioned residual-operator class can
reproduce the map
\[
    (v,h)\mapsto \mathcal N(v,\mathbf 1\otimes h)
\]
on $S$.

\medskip
\noindent
\textbf{Step 2: Reproducing the augmented FNO with a residual operator.}

In Step 1, we constructed a standard FNO $\mathcal N$ on the augmented
input $(v,\mathbf 1\otimes h)$ such that
$\mathcal N(v,\mathbf 1\otimes h)$
approximates $\mathcal{R} (v,h)$ uniformly on $S$. We now show that, with a suitable choice of parameters, the residual operator $\mathcal{R}_\theta$ can exactly reproduce this augmented-input FNO on $S$.
The only difference is how the latent variable $h$ enters the model. In
$\mathcal N$, $h$ is provided as a spatially constant input channel. In
$\mathcal{R}_\theta$, $h$ enters through the additive conditioning terms
$\beta_\ell(h)$. 

To construct such a residual operator, we first write the standard FNO
$\mathcal N$ explicitly.  Since $\mathcal N$
acts on the augmented input
\[
    w(x)=(v(x),h),
    \qquad
    w=(v,\mathbf 1\otimes h),
\]
we may write
\[
    \mathcal N
    =
    Q^{\mathcal N}
    \circ
    \mathcal L^{\mathcal N}_L
    \circ
    \cdots
    \circ
    \mathcal L^{\mathcal N}_1
    \circ
    \mathcal P^{\mathcal N}.
\]
Here $\mathcal P^{\mathcal N}$ is a pointwise lifting layer implemented by an MLP,
\[
    \mathcal P^{\mathcal N}(w)(x)
    =
    P^{\mathcal N}(v(x),h,x)
    \in \mathbb R^{d_{\mathcal N}},
\]
where $d_{\mathcal N}$ is the hidden channel width of $\mathcal N$. 

Each hidden Fourier layer $\mathcal L^{\mathcal N}_l$ has the form
\[
    \mathcal L^{\mathcal N}_\ell(z)(x)
    =
    \sigma\bigl(
        W^{\mathcal N}_\ell z(x)
        +
        b^{\mathcal N}_\ell
        +
        (K^{\mathcal N}_\ell z)(x)
    \bigr),
    \qquad z(x)\in \mathbb R^{d_{\mathcal N}},
\]
and $Q^{\mathcal N}$ is the final pointwise projection layer.

We construct the conditioned residual operator $\mathcal{R}_\theta$ in the form
\[
    \mathcal{R}_\theta
    =
   \mathcal Q
    \circ
   \mathcal L_L^{(h)}
    \circ
    \cdots
    \circ
   \mathcal L_1^{(h)}
    \circ
   \mathcal L_0^{(h)}
    \circ
   \mathcal P.
\]

Here, $\mathcal P$ is the lifting layer, and
$\mathcal L_0^{(h)}$ is a preparation block consisting of one or more
conditioned residual layers. Its role is to transform the lifted input
into the hidden representation required by the target FNO
$\mathcal N$. For each $\ell=1,\ldots,L$,
$\mathcal L_\ell^{(h)}$ is a single conditioned Fourier layer chosen to
reproduce the $\ell$th Fourier layer of $\mathcal N$. Finally,
$\mathcal Q$ reproduces the projection layer of $\mathcal N$.

% Choose the hidden width $d_R$ of $\mathcal{R}_\theta$ appropriately so that we can
% decompose the channel space as
% \[
%     \mathbb R^{d_R}
%     =
%     \mathcal A
%     \oplus
%     \mathcal B .
% \]
% The subspace $\mathcal A$ stores the initial input features,
% namely $v(x)$ together with the spatial coordinate features $x$. These channels
% provide the data needed by the preparation block. The subspace
% $\mathcal B\simeq \mathbb R^{d_{\mathcal N}}$ is initialized to zero and reserved for carrying the hidden feature representation of the augmented FNO $\mathcal N$.
% After the preparation block, $\mathcal B$ will contain the hidden feature
% $y_J(x,h)$ of the lifting MLP, and the subsequent blocks will update this feature
% to reproduce the computation of the Fourier blocks of $\mathcal N$.

Choose the hidden width $d_R$ of $\mathcal R_\theta$ so that its
channel space admits the coordinate decomposition
$$
\mathbb R^{d_R}
=
\mathcal A\oplus\mathcal B,
\qquad
\mathcal A\simeq\mathbb R^{c+d},
\qquad
\mathcal B\simeq\mathbb R^{d_{\mathcal N}}.
$$
Let $\pi_{\mathcal A}$ and $\pi_{\mathcal B}$ denote the corresponding
coordinate projections. The auxiliary channels $\mathcal A$ store the
local input features $(v(x),x)$, while the target channels $\mathcal B$
are used to reproduce the hidden features of the augmented FNO
$\mathcal N$.

Choose the lifting layer $\mathcal P$ so that
$$
\mathcal P(v)(x)
=
\bigl(v(x),x,0_{d_{\mathcal N}}\bigr)
\in
\mathcal A\oplus\mathcal B.
$$
Thus,
$$
\pi_{\mathcal A}\mathcal P(v)(x)
=
\bigl(v(x),x\bigr),
\qquad
\pi_{\mathcal B}\mathcal P(v)(x)
=
0.
$$

Write the lifting MLP of $\mathcal N$ in the standard form
\[
    y_1(x,h)=\bigl(v(x),x,h\bigr),
\]
\[
    y_{j+1}(x,h)
    =
    \sigma\bigl(
        A_j y_{j}(x,h)+c_j
    \bigr),
    \qquad j=1,\ldots J-1,
\]
and
\[
    P^{\mathcal N}(v(x),h,x)
    =
    A_{J}y_J(x,h)+c_{J}.
\]
By increasing the width and padding unused channels with zeros if necessary, we may assume that
$$
y_2,\ldots,y_J
\in
\mathbb R^{d_{\mathcal N}}
\simeq
\mathcal B.
$$

We now construct the preparation block $\mathcal L_0 ^{(h)}$. It consists of $J-1$ conditioned layers with $K_\ell=0$. Writing
$$
z_1(x)
=
\mathcal P(v)(x),
$$ 
the layers inside $\mathcal L_0^{(h)}$ take the form
$$
z_{j+1}(x)
=
\sigma\!\left(
W_{0,j}z_j(x)+\beta_{0,j}(h)
\right),
\qquad
j=1,\ldots,J-1.
$$
We choose these layers so that
$$
\pi_{\mathcal B}z_j(x)
=
y_j(x,h),
\qquad
j=2,\ldots,J.
$$
% The role of $\mathcal L_0$ is to reproduce the hidden part of the lifting
% MLP of the augmented FNO. We show that for the same input $(v(x),x,h)$ used by the lifting MLP of $\mathcal N$, the preparation block $\mathcal L_0$ can be chosen so that its output on the subspace $\mathcal B$ agrees exactly with $y_J(x,h)$.

To construct the first layer, decompose the first affine map of the
lifting MLP according to its local and latent inputs:
$$
A_1y_1(x,h)+c_1
=
A_{1,\mathrm{loc}}
\bigl(v(x),x\bigr)
+
A_{1,h}h+c_1.
$$

Choose the $\mathcal A\to\mathcal B$ and
$\mathcal B\to\mathcal B$ blocks of $W_{0,1}$ so that
$$
\pi_{\mathcal B}W_{0,1}|_{\mathcal A}
=
A_{1,\mathrm{loc}},
\qquad
\pi_{\mathcal B}W_{0,1}|_{\mathcal B}
=
0,
$$
and choose the conditioning term so that
$$
\pi_{\mathcal B}\beta_{0,1}(h)
=
A_{1,h}h+c_1.
$$
It follows that
$$
\begin{aligned}
\pi_{\mathcal B}z_2(x)
&=
\sigma\!\left(
A_{1,\mathrm{loc}}\bigl(v(x),x\bigr)
+
A_{1,h}h+c_1
\right)=
y_2(x,h).
\end{aligned}
$$

For $j=2,\ldots,J-1$, suppose inductively that
$$
\pi_{\mathcal B}z_j(x)
=
y_j(x,h).
$$
Choose the $j$th conditioned layer so that
$$
\pi_{\mathcal B}W_{0,j}|_{\mathcal B}
=
A_j,
\qquad
\pi_{\mathcal B}W_{0,j}|_{\mathcal A}
=
0,
$$
and
$$
\pi_{\mathcal B}\beta_{0,j}(h)
=
c_j.
$$
Then
$$
\begin{aligned}
\pi_{\mathcal B}z_{j+1}(x)
&=
\sigma\!\left(
A_j\pi_{\mathcal B}z_j(x)+c_j
\right)
\\
&=
\sigma\!\left(
A_jy_j(x,h)+c_j
\right)
\\
&=
y_{j+1}(x,h).
\end{aligned}
$$
Therefore, by induction,
$$
\pi_{\mathcal B}z_J(x)
=
y_J(x,h).
$$

After the preparation block, all subsequent couplings from
$\mathcal A$ into $\mathcal B$ are set to zero. Hence the remaining
computation in the $\mathcal B$-channels depends only on the feature
$y_J(x,h)$.

The remaining affine output layer of the lifting MLP is
$$
P^{\mathcal N}(x)=P^{\mathcal N}\bigl(v(x),h,x\bigr)
=
A_Jy_J(x,h)+c_J.
$$
Rather than implementing this map as an additional preparation layer,
we absorb it into the first conditioned Fourier layer $\mathcal L_1^{(h)}$.

The pre-activation of the first Fourier layer of $\mathcal N$ is
$$
W_1^{\mathcal N}P^{\mathcal N}(x)
+
\bigl(\mathcal K_1^{\mathcal N}P^{\mathcal N}\bigr)(x)
+
b_1^{\mathcal N}.
$$
Substituting
$P^{\mathcal N}(x)=A_Jy_J(x,h)+c_J$ gives
$$
\begin{aligned}
&
W_1^{\mathcal N}A_Jy_J(x,h)
+
\bigl(\mathcal K_1^{\mathcal N}(A_Jy_J)\bigr)(x)
+
W_1^{\mathcal N}c_J
+
\bigl(\mathcal K_1^{\mathcal N}c_J\bigr)(x)
+
b_1^{\mathcal N}.
\end{aligned}
$$

On the other hand, the $\mathcal B$-component of the pre-activation
of the first conditioned Fourier layer of $\mathcal R_\theta$ is
$$
\pi_{\mathcal B}
\left[
W_1z_J(x)
+
(\mathcal K_1z_J)(x)
+
\beta_1(h)
\right].
$$
Choose the relevant blocks of the pointwise and Fourier operators so
that
$$
\pi_{\mathcal B}W_1|_{\mathcal B}
=
W_1^{\mathcal N}A_J,
\qquad
\pi_{\mathcal B}W_1|_{\mathcal A}
=
0,
$$
and, for every retained Fourier mode $k$,
$$
\pi_{\mathcal B}R_1(k)|_{\mathcal B}
=
R_1^{\mathcal N}(k)A_J,
\qquad
\pi_{\mathcal B}R_1(k)|_{\mathcal A}
=
0.
$$
Finally, choose
$$
\pi_{\mathcal B}\beta_1(h)
=
W_1^{\mathcal N}c_J
+
\mathcal K_1^{\mathcal N}c_J
+
b_1^{\mathcal N}.
$$
With these choices, the $\mathcal B$-component of the first
conditioned Fourier layer of $\mathcal R_\theta$ exactly reproduces
the first Fourier layer of $\mathcal N$ applied to $P^{\mathcal N}(x)$.

We next reproduce the remaining Fourier layers of $\mathcal N$. For
each $\ell=2,\ldots,L$, choose the corresponding conditioned layer of
$\mathcal R_\theta$ so that
$$
\pi_{\mathcal B}W_\ell|_{\mathcal B}
=
W_\ell^{\mathcal N},
\qquad
\pi_{\mathcal B}W_\ell|_{\mathcal A}
=
0,
$$
$$
\pi_{\mathcal B}\mathcal K_\ell|_{\mathcal B}
=
\mathcal K_\ell^{\mathcal N},
\qquad
\pi_{\mathcal B}\mathcal K_\ell|_{\mathcal A}
=
0,
$$
and
$$
\pi_{\mathcal B}\beta_\ell(h)
=
b_\ell^{\mathcal N}.
$$
Since the first conditioned Fourier layer reproduces the output of
$\mathcal L_1^{\mathcal N}$ in the $\mathcal B$-channels, and each
subsequent conditioned layer reproduces the corresponding layer
$\mathcal L_\ell^{\mathcal N}$, induction over $\ell$ shows that the
hidden Fourier computation of $\mathcal R_\theta$ agrees exactly with
that of $\mathcal N$ in the $\mathcal B$-channels.

It remains to match the output projection. Choose $\mathcal Q$ to
depend only on the $\mathcal B$-component and to apply the same
pointwise projection as $\mathcal Q^{\mathcal N}$. Then, for every
$(v,h)\in S$,
$$
\mathcal R_\theta(v,h)
=
\mathcal N(v,\mathbf 1\otimes h).
$$
Combining this identity with the approximation estimate from Step 1
gives
$$
\sup_{(v,h)\in S}
\left\|
\mathcal R_\theta(v,h)
-
\mathcal R(v,h)
\right\|_H
<
\eta.
$$
This proves the lemma.
\end{proof}

\subsubsection*{A.3 Main Proof}

\begin{proof}[Proof of Theorem 1]
We now prove Theorem~\ref{thm:imno-uat} by combining the approximation results above.

\paragraph{Step 0: Target maps and extensions.}
We first recall the notations introduced in Section A.1. For $u\in K$, define
\[
    h(u)=E(\Pi(u)),
    \qquad
    r(u)=u-\Pi(u),
\]
and
\[
    \Lambda_0=h(K),
    \qquad
    F(h)
    =
    E\bigl(\mathcal{G}(D(h))\bigr).
\]
We also write
\[
    h_1(u):=F(h(u)),
    \qquad
    \Lambda_1:=F(\Lambda_0).
\]
By the discussion in Section A.1, $F:\Lambda_0\to\Lambda_1$ is a
homeomorphism. Moreover,
\[
    D(h_1(u))
    =
    \Pi(\mathcal{G}(u)).
\]

Because the residual operator in our architecture is conditioned
on the next latent coordinate, we first introduce its target domain
\[
    C_K'
    :=
    \{(r(u),h_1(u)):u\in K\}
    \subset H\times\mathbb R^{d_h}.
\]
We then define the exact residual target $\mathcal R:C'_K\to H$ by
\[
    \mathcal R(v,h_1)
    :=
    \mathcal R_0\bigl(v,(F)^{-1}(h_1)\bigr).
\]
Then $\mathcal R$ is continuous on $C_K'$, and for every $u\in K$,
\[
    \mathcal R(r(u),h_1(u))
    =
    \mathcal{G}(u)-\Pi(\mathcal{G}(u)).
\]
Thus $\mathcal R$ gives the exact next-step residual when evaluated at the
exact current residual and exact next-step latent coordinate.

Finally, by lemma~\ref{lem:extension}, choose continuous extensions
\[
    \bar D:\mathbb R^{d_h}\to H,
    \qquad
    \bar {\mathcal R}:H\times\mathbb R^{d_h}\to H,
\]
of $D$ and $\mathcal R$, respectively. Also extend
\[
    f(h)
    :=
    \frac{F(h)-h}{\Delta t},
    \qquad h\in\Lambda_0,
\]
to a continuous map
\[
    \bar f:\mathbb R^{d_h}\to\mathbb R^{d_h}.
\]
Let
\[
    \Lambda_* :=
    \{h\in\mathbb R^{d_h}:
    \operatorname{dist}(h,\Lambda_0\cup\Lambda_1)\le 1\}.
\]
This set is compact. Let $\omega_D$ and $\omega_f$ be moduli of uniform
continuity of $\bar D$ and $\bar f$ on $\Lambda_*$.

\paragraph{Step 1: Choosing the error budgets.}
We now choose the small constants used in the approximation of the encoder,
latent dynamics, decoder, and residual operator.

Apply lemma~\ref{lem:ruc} to $\bar{\mathcal R}$ on the compact set
$C_K'$. With tolerance $\varepsilon/4$, there exists $\delta_T\in(0,1]$
such that if $(v_0,h_0)\in C_K'$ and
\[
    \|v-v_0\|_H+\norm{h-h_0}<\delta_T,
\]
then
\[
    \|\bar{\mathcal R}(v,h)-\bar{\mathcal R}(v_0,h_0)\|_H<\frac{\varepsilon}{4}.
\]

Choose $\tau\in(0,\frac{\delta_T}{2}]$ such that
\[
    \omega_D(\tau)
    \le
    \min\left\{\frac{\varepsilon}{4},\frac{\delta_T}{4}\right\}.
\]
Then choose $\delta_1\in(0,\tau]$ such that
\[
    \delta_1+\Delta t\,\omega_f(\delta_1)\le \frac{\tau}{2}.
\]
Finally choose $\delta_2>0$ such that
\[
    \Delta t\,\delta_2\le \frac{\tau}{2},
\]
and set
\[
    \delta_4
    :=
    \min\left\{\frac{\varepsilon}{4},\frac{\delta_T}{4}\right\}.
\]

These choices ensure that small errors in the latent coordinate lead to
small errors after decoding, and that the approximate residual input
remains within the $\delta_T$-neighborhood where $\bar{\mathcal R}$ is controlled.

\paragraph{Step 2: Encoder and latent dynamics.}
We first approximate the exact latent coordinate map
$h:K\to\mathbb R^{d_h}$. By lemma~\ref{lem:enc},
there exist a truncation level $k_E$ and an encoder MLP such that
\[
    \hat h_0(u):=E_\theta(u)
\]
satisfies
\[
    \sup_{u\in K}
        \norm{\hat h_0(u)-h(u)}
    <
    \delta_1.
\]
Since $h(u)\in\Lambda_0$ and $\delta_1\le 1$, we have
$\hat h_0(u)\in\Lambda_*$ for every $u\in K$.

Next, approximate the continuous latent vector field $\bar f$ on
$\Lambda_*$. By MLP universality, choose $f_\theta$ so
that
\[
    \sup_{h\in\Lambda_*}
    \norm{f_\theta(h)-\bar f(h)}    <
    \delta_2.
\]
The learned one-step latent update is
\[
    \hat h_1(u)
    :=
    \hat h_0(u)+\Delta t\,f_\theta(\hat h_0(u)).
\]
We compare it with the exact next latent coordinate
\[
    h_1(u)=F(h(u)).
\]
Using the definition of $\bar f$ on $\Lambda_0$,
\[
    h_1(u)
    =
    h(u)+\Delta t\,\bar f(h(u)).
\]
Therefore,
\[
\begin{aligned}
    \norm{\hat h_1(u)-h_1(u)}
    &\le
    \norm{\hat h_0(u)-h(u)} \\
    &\quad
    +\Delta t\,\norm{f_\theta(\hat h_0(u))-\bar f(\hat h_0(u))} \\
    &\quad
    +\Delta t\,\norm{\bar f(\hat h_0(u))-\bar f(h(u))} \\
    &\le
    \delta_1+\Delta t\,\delta_2+\Delta t\,\omega_f(\delta_1).
\end{aligned}
\]
Define
\[
    \delta_3
    :=
    \delta_1+\Delta t\,\delta_2+\Delta t\,\omega_f(\delta_1).
\]
By the choices in Step 1, $\delta_3\le \tau$. Hence
\[
    \norm{\hat h_1(u)-h_1(u)}\le \delta_3\le \tau.
\]
Since $h_1(u)\in\Lambda_1$ and $\tau\le 1$, we also have
$\hat h_1(u)\in\Lambda_*$.

\paragraph{Step 3: Decoder and manifold prediction.}
We now apply Lemma~\ref{lem:dec} to the continuous map
$$
\bar D|_{\Lambda_*}:\Lambda_*\to H.
$$
Thus, we can choose a decoder $D_\theta$ such that
\[
    \sup_{h\in\Lambda_*}
    \|D_\theta(h)-\bar D(h)\|_H
    <
    \delta_4.
\]

The learned and reference next-step manifold components are given by
\[
    u^M_{\mathrm{next}}(u)
    :=
    D_\theta(\hat h_1(u)),
    \qquad
    u^{M,\mathrm{ref}}_{\mathrm{next}}(u)
    :=
    \Pi(\mathcal{G}(u))
    =
    \bar D(h_1(u)).
\]
Therefore,
\[
\begin{aligned}
    \|u^M_{\mathrm{next}}(u)
    -
    u^{M,\mathrm{ref}}_{\mathrm{next}}(u)\|_H
    &=
    \|D_\theta(\hat h_1(u))-\bar D(h_1(u))\|_H \\
    &\le
    \|D_\theta(\hat h_1(u))-\bar D(\hat h_1(u))\|_H \\
    &\quad
    +
    \|\bar D(\hat h_1(u))-\bar D(h_1(u))\|_H \\
    &<
    \delta_4+\omega_D(\delta_3) \\
    &\le
    \frac{\varepsilon}{4}+\frac{\varepsilon}{4}
    =
    \frac{\varepsilon}{2}.
\end{aligned}
\]
Thus
\[
    \sup_{u\in K}
    \|u^M_{\mathrm{next}}(u)
    -
    u^{M,\mathrm{ref}}_{\mathrm{next}}(u)\|_H
    <
    \frac{\varepsilon}{2}.
\]

Using the decoder estimate above, we can also control the error in the
residual input passed to the learned residual operator. The learned and reference initial residuals are given by
\[
    \hat u^R(u):=u-D_\theta(\hat h_0(u)),
    \qquad
    r(u)=u-\Pi(u)
    =
    u-D(h(u))
    =
    u-\bar D(h(u)).
\]
Therefore,
\[
\begin{aligned}
    \|\hat u^R(u)-r(u)\|_H
    &=
    \|D_\theta(\hat h_0(u))-\bar D(h(u))\|_H \\
    &\le
    \|D_\theta(\hat h_0(u))-\bar D(\hat h_0(u))\|_H \\
    &\quad
    +
    \|\bar D(\hat h_0(u))-\bar D(h(u))\|_H \\
    &<
    \delta_4+\omega_D(\delta_1) \\
    &\le
    \frac{\delta_T}{4}+\frac{\delta_T}{4}
    =
    \frac{\delta_T}{2}.
\end{aligned}
\]

\paragraph{Step 4: Residual prediction.}
The encoder, latent dynamics, and decoder have now been fixed. Hence the
map
\[
    u\mapsto (\hat u^R(u),\hat h_1(u))
\]
is continuous on $K$. Since $K$ is compact, the set
\[
    S_\theta
    :=
    \{(\hat u^R(u),\hat h_1(u)):u\in K\}
    \subset H\times\mathbb R^{d_h}
\]
is compact.

By lemma~\ref{lem:cond_fno}, there exists a conditioned residual operator
$\mathcal R_\theta$ such that
\[
\sup_{(v,h)\in S_\theta}
\|\mathcal R_\theta(v,h)-\bar{\mathcal R}(v,h)\|_H
<
\frac{\varepsilon}{4}.
\]

Moreover, by the estimates established in Steps 2 and 3,
\[
\begin{aligned}
    \|\hat u^R(u)-r(u)\|_H
    +
    \norm{\hat h_1(u)-h_1(u)}
    &<
    \frac{\delta_T}{2}+\delta_3 \\
    &\le
    \frac{\delta_T}{2}+\frac{\delta_T}{2}
    =
    \delta_T.
\end{aligned}
\]
% Here we used $\delta_3\le \tau \le \frac{\delta_T}{2}$. 
Hence, for every $u\in K$, the pair
$(\hat u^R(u),\hat h_1(u))$ lies within the $\delta_T$-neighborhood of
$(r(u),h_1(u))\in C_K'$.
Therefore, by the choice of $\delta_T$,
\[
    \|\bar{\mathcal R}(\hat u^R(u),\hat h_1(u))
    -
    \bar{\mathcal R}(r(u),h_1(u))\|_H
    <
    \frac{\varepsilon}{4}.
\]

The learned and reference next-step residuals are given by
\[
u^R_{\mathrm{next}}(u)
:=
\mathcal R_\theta\bigl(\hat u^R(u),\hat h_1(u)\bigr),
\qquad
u^{R,\mathrm{ref}}_{\mathrm{next}}(u)
:=
\bar{\mathcal R}\bigl(r(u),h_1(u)\bigr).
\]
Using the estimate above, we obtain
\[
\begin{aligned}
    \|u^R_{\mathrm{next}}(u)
    -
    u^{R,\mathrm{ref}}_{\mathrm{next}}(u)\|_H
    &=
    \|\mathcal R_\theta(\hat u^R(u),\hat h_1(u))
    -
    \bar{\mathcal R}(r(u),h_1(u))\|_H \\
    &\le
    \|\mathcal R_\theta(\hat u^R(u),\hat h_1(u))
    -
    \bar{\mathcal R}(\hat u^R(u),\hat h_1(u))\|_H \\
    &\quad
    +
    \|\bar{\mathcal R}(\hat u^R(u),\hat h_1(u))
    -
    \bar{\mathcal R}(r(u),h_1(u))\|_H \\
    &<
    \frac{\varepsilon}{4}+\frac{\varepsilon}{4}
    =
    \frac{\varepsilon}{2}.
\end{aligned}
\]
Thus
\[
    \sup_{u\in K}
    \|u^R_{\mathrm{next}}(u)
    -
    u^{R,\mathrm{ref}}_{\mathrm{next}}(u)\|_H
    <
    \frac{\varepsilon}{2}.
\]

\paragraph{Step 5: Combining the estimates.}
For every $u\in K$, the IMNO prediction is
\[
    \mathcal G_\theta(u)
    =
    u^M_{\mathrm{next}}(u)+u^R_{\mathrm{next}}(u),
\]
while the exact one-step solution is
\[
    \mathcal{G}(u)
    =
    u^{M,\mathrm{ref}}_{\mathrm{next}}(u)
    +
    u^{R,\mathrm{ref}}_{\mathrm{next}}(u).
\]
From Steps 3 and 4,
\[
\begin{aligned}
    &\|u^M_{\mathrm{next}}(u)
    -
    u^{M,\mathrm{ref}}_{\mathrm{next}}(u)\|_H
    +
    \|u^R_{\mathrm{next}}(u)
    -
    u^{R,\mathrm{ref}}_{\mathrm{next}}(u)\|_H \\
    &\qquad
    <
    \frac{\varepsilon}{2}+\frac{\varepsilon}{2}
    =
    \varepsilon.
\end{aligned}
\]
Taking the supremum over $u\in K$ gives the first claim.

Finally, by the triangle inequality,
\[
\begin{aligned}
    \|\mathcal G_\theta(u)-\mathcal{G}(u)\|_H
    &=
    \|
    u^M_{\mathrm{next}}(u)+u^R_{\mathrm{next}}(u)
    -
    u^{M,\mathrm{ref}}_{\mathrm{next}}(u)
    -
    u^{R,\mathrm{ref}}_{\mathrm{next}}(u)
    \|_H \\
    &\le
    \|u^M_{\mathrm{next}}(u)
    -
    u^{M,\mathrm{ref}}_{\mathrm{next}}(u)\|_H
    +
    \|u^R_{\mathrm{next}}(u)
    -
    u^{R,\mathrm{ref}}_{\mathrm{next}}(u)\|_H.
\end{aligned}
\]
Hence
\[
    \sup_{u\in K}
    \|\mathcal G_\theta(u)-\mathcal{G}(u)\|_H
    <
    \varepsilon.
\]
This completes the proof.
\end{proof}

\subsection*{B. Recurrent Neural Operator (RNO)}

The Recurrent Neural Operator (RNO) is a neural-operator architecture
for learning the time evolution of PDEs. Unlike a standard
autoregressive neural operator, which predicts the next state solely
from the current solution field, RNO maintains a function-valued hidden
state that carries information across time steps \cite{liu2023tipping}. 

Let $u_t(x)$ denote the solution field at time step $t$. The input is
first mapped to a latent feature field through a pointwise lifting
layer,
\[
    v_t(x)=\mathcal P(u_t)(x).
\]
Here, $\mathcal P$ is a pointwise lifting layer implemented by an MLP.

The hidden state $h_t(x)$ is then updated using a gated recurrent
mechanism. The update gate and reset gate are defined by
\[
    z_t(x)
    =
    \sigma\!\left(
        \Phi_z(v_t)(x)
        +
        \Psi_z(h_t)(x)
        +b_z
    \right),
\]
and
\[
    r_t(x)
    =
    \sigma\!\left(
        \Phi_r(v_t)(x)
        +
        \Psi_r(h_t)(x)
        +
        b_r
    \right),
\]
respectively. The candidate hidden state is
\[
    \widetilde h_t(x)
    =
    \varphi\!\left(
        \Phi_h(v_t)(x)
        +
        \Psi_h\bigl(r_t\odot h_t\bigr)(x)
        + b_h
    \right),
\]
where $\sigma$ denotes the sigmoid function, $\varphi$ is the SELU
activation function, and $\odot$ denotes pointwise multiplication.

The next hidden state is obtained by interpolating between the previous
hidden state and the candidate state:
\[
    h_{t+1}(x)
    =
    \bigl(1-z_t(x)\bigr)\odot h_t(x)
    +
    z_t(x)\odot\widetilde h_t(x).
\]
Thus, the update gate $z_t$ controls how much of the previous hidden
state is retained, while the reset gate $r_t$ controls how the previous
hidden state contributes to the candidate state.

The operators
\[
    \Phi_z,\ \Phi_r,\ \Phi_h,
    \qquad
    \Psi_z,\ \Psi_r,\ \Psi_h
\]
are neural operators. In the implementation considered here, each of
them is realized by a single FNO layer, allowing
the recurrent update to incorporate nonlocal spatial interactions
through spectral convolution.

Finally, the predicted solution at the next time step is obtained by
projecting the updated hidden state back to the solution space:
\[
    u_{t+1}(x)
    =
    \mathcal Q(h_{t+1})(x),
\]
where $\mathcal Q$ is a pointwise projection layer.

\end{document}